\documentclass[11pt]{article}

\RequirePackage[l2tabu, orthodox]{nag}
\usepackage[T1]{fontenc}
\usepackage[utf8]{inputenc}
\usepackage{lmodern}
\usepackage[frenchb,english]{babel}
\usepackage{color} 
\usepackage{amsthm,amssymb,amsmath}
\usepackage{graphicx}
\usepackage{enumerate}

\def\comment#1{}

\newcommand{\R}{\mathbf{R}}
\newcommand{\N}{\mathbf{N}}

\newcommand{\A}{\mathbf{A}}

\newcommand{\OO}{\mathcal{O}}
\newcommand{\cV}{\mathcal{V}}

\newcommand{\Z}{\mathbf{Z}}
\newcommand{\Ss}{\mathbf{S}}

\newcommand{\cA}{\mathcal{A}}

\newcommand{\cW}{\mathcal{W}}
\newcommand{\cL}{\mathcal{L}}
\newcommand{\cR}{\mathcal{R}}
\newcommand{\cX}{\mathcal{X}}
\newcommand{\cY}{\mathcal{Y}}
\newcommand{\cK}{\mathcal{K}}
\newcommand{\cU}{\mathcal{U}}
\newcommand{\x}{\mathbf{x}}
\newcommand{\eps}{\varepsilon}

\newcommand{\fhi}{\varphi}

\def\to{\mathop{\rightarrow}}
\def\dans{\mathop{\subset}}
\newcommand{\moins}{\setminus}

\newcommand{\Diff}{\mathrm{Diff}}
\newcommand{\per}{\mathrm{Per}}
\newcommand{\Ind}{\mathrm{Ind}}
\newcommand{\Zero}{\mathrm{Zero}}
\newcommand{\fix}{\mathrm{Fix}}
\newcommand{\Spec}{\mathrm{Spec}}
\newcommand{\ind}{\mathrm{Ind}}
\newcommand{\proj}{\mathrm{proj}}
\newcommand{\hol}{\mathrm{hol}}
\newcommand{\dist}{\mathrm{dist}}
\newcommand{\Col}{\mathrm{Col}}
\newcommand{\Int}{\mathrm{Int}}
\newcommand{\cN}{\mathcal{N}}
\newcommand{\Kta}{K_{nl}^T}
\newcommand{\Ktr}{K_{nl}^{\perp}}
\newcommand{\vide}{\emptyset}
\def\dans{\mathop{\subset}}
\def\To{\mathop{\longrightarrow}}
\newcommand{\Map}{\longmapsto}

\newtheorem{thm}{Theorem}[section]
\newtheorem{thmA}{Theorem}
\newtheorem*{thm*}{Theorem}

\newtheorem{lem}[thm]{Lemma}
\newtheorem{prop}[thm]{Proposition}
\newtheorem{cor}[thm]{Corollary}

\theoremstyle{definition}
\newtheorem{dfn}[thm]{Definition}

\theoremstyle{remark}
\newtheorem{rem}[thm]{Remark}

\usepackage[margin=2.5cm]{geometry}

\usepackage[colorlinks=true,linkcolor=blue,citecolor=magenta]{hyperref}

\title{
Existence of common zeros for commuting vector fields on $3$-manifolds II. Solving global difficulties.
}

\date{}
\author{S{\'e}bastien Alvarez \and Christian Bonatti \and Bruno Santiago}

\begin{document}

\maketitle

\begin{abstract}
We address the following conjecture about the existence of common zeros for commuting vector fields in dimension three: if $X,Y$ are two $C^1$ commuting vector fields on a $3$-manifold $M$, and $U$ is a relatively compact open such that $X$ does not vanish on the boundary of $U$ and has a non vanishing Poincar\'e-Hopf index in $U$, then $X$ and $Y$ have a common zero inside $U$. We prove this conjecture when $X$ and $Y$ are of class $C^3$ and every periodic 
orbit of $Y$ along which $X$ and $Y$ are collinear is partially hyperbolic. We also prove the conjecture, still in the $C^3$ setting, assuming that the flow $Y$ leaves invariant a transverse plane field. These results shed new light on the $C^3$ case of the conjecture. 
\end{abstract}

{\footnotesize
	}

\section{Introduction}

It is a challenging and open problem to determine when a smooth action of the abelian group $\R^k$ on a given manifold $M$ possesses a fixed point. Such an action is determined by the data of $k$ complete vector fields $X^1,...,X^k$ on $M$ that \emph{commute}, i.e. $[X^i,X^j]=0$ for every pair $(i,j)$, where $[.\, ,.]$ denotes the usual Lie bracket of vector fields. A fixed point for the $\R^k$-action is a common zero of the corresponding vector fields $X^1,...,X^k$.

When $k=1$ the index theory developed by H. Poincar\'e and H. Hopf relates the topological properties of $M$ and the existence of fixed points. We shall define the so-called Poincar\'e-Hopf index in Section \ref{s:strategy}. Such an index theory is not available for more general $k$.

The first relation between the topological properties of a manifold $M$ and the existence of fixed points for a given action of $\R^k$ on $M$ was obtained by E.L. Lima in 1964. In \cite{Lima_geral,Lima_esfera}, he proved that \emph{an action of $\R^k$ on a closed surface with non-zero Euler characteristic has necessarily a fixed point.}

The generalization of Lima's theorem in dimension $3$ faces an immediate difficulty: the Euler characteristic of a $3$-dimensional manifold always vanishes. The relevant topological properties of $M$ are no longer \emph{global}, but rather \emph{semi-local}: this is the content of the following conjecture. It was addressed by the last two authors in \cite{BS}, and was stated in \cite{Bonatti_analiticos} as a problem (Probl\`eme 2 of that reference). It concerns the case $k=2$ and $\dim M=3$.

\paragraph{Conjecture --}  \emph{Let $X,Y$ be two commuting vector fields of class $C^1$ on a $3$-dimensional manifold $M$ and let $U\dans M$ be a relatively compact open set such that $\Zero(X)\cap\partial U=\vide$. 
	Assume moreover that the Poincar\'e-Hopf index of $X$ in $U$ does not vanish: $\Ind(X,U)\neq 0$. Then $X$ and $Y$ have a common zero inside $U$.}

This conjecture was confirmed \emph{in the case $X$ and $Y$ are analytic} by the second author in 1992 (see \cite{Bonatti_analiticos}). It remains widely open since then in the context of lower regularity. A door was recently opened by the last two authors. In \cite{BS}, they solved the conjecture when $X$ and $Y$ are of class $C^1$, but assuming that a natural dynamical object (the collinearity locus of the two vector fields) has a special geometric property: see Theorem \ref{t:BS} for the precise statement.

The main intuition behind the conjecture is the following. If $X$ and $Y$ have no common zeros, then, \emph{close to the compact set} $K=\Zero(X)\cap U$, the vector field $X$ must commute with the \emph{non-vanishing vector field} $Y$. So either $X$ is collinear to $Y$ all along a $Y_t$-orbit, or it is never so. This should prevent $X$ of turning in all directions, which would give $\Ind(X,U)=0$. Motivated by this we intend to attack the conjecture under the following form.

\paragraph{Alternative form --}\emph{Let $X,Y$ be two commuting vector fields of class $C^1$ on a $3$-dimensional manifold $M$. Let $U\dans M$ be a relatively compact open set such that $\Zero(X)\cap\partial U=\Zero(Y)\cap U=\vide$, then $\Ind(X,U)=0$.}

In order to see that the latter statement actually implies the conjecture, we use the following property of Poincar\'e-Hopf index. If $K\dans\Zero(X)$ is an isolated compact set, then $\Ind(X,U)$ is independent of the isolating neighbourhood $U$ of $K$ (where one says that $U$ is isolating if it is relatively compact, if $K=\Zero(X)\cap U$ and if $\Zero(X)\cap\partial U=\vide$).

Now, assume that the Alternative Form has been established. Let $X,Y$ be two $C^1$-commuting vector fields on a $3$-dimensional manifold $M$. Assume that $\Ind(X,U)\neq 0$, where $U$ is a relatively compact open subset of $M$, but $K=\Zero(X)\cap U$ is disjoint from $\Zero(Y)$. Thus, we can choose a smaller neighbourhood $U_1\dans U$ of $K$ whose closure is disjoint from $\Zero(Y)$. By the property stated above, $\Ind(X,U_1)=\Ind(X,U)$, and by the Alternative Form, $\Ind(X,U_1)=0$, which is impossible.

\paragraph{Main results --} When we explained above our main intuition, it was implicit that the main character of the paper is the \emph{collinearity locus} 
$$\Col_U(X,Y)=\bigcup_{c\in\R}\Zero(X-cY)\cap\overline{U},$$
of $X$ and $Y$ inside $\overline{U}$. If the collinearity locus is ``too big'' (for instance if it contains an open neighbourhood of $K$) there is not enough space for $X$ to turn and the index must vanish. If on the contrary it is ``too small'' (for example if it is reduced to a periodic orbit of $Y$) then $X$ can be approached by non-vanishing vector fields (i.e. the vector fields $X-cY$ for $|c|$ small) so it must have zero index. When it is not too big nor too small, the key lies in the detailed analysis of the dynamics of $Y$ in the neighbourhood of $\Col_U(X,Y)$. In particular the existence of stable/unstable sets for $Y$, which must be invariant by $X$ since it commutes with $Y$, prevents $X$ to turn around in all directions. This idea was successfully implemented in the paper \cite{BS}, where the following was proven, by constructing stable sets for $Y$.

\begin{thm}[Bonatti-Santiago]
	\label{t:BS}
	Let $M$ be a $3$-dimensional manifold and $X,Y$ be two commuting vector fields of class $C^1$. Let $U$ be a relatively compact open set such that $\Zero(Y)\cap U=\Zero(X)\cap\partial U=\vide$. Assume that $\Col_U(X,Y)$ is contained in a closed and boundaryless two-dimensional submanifold of $M$. Then
	$$\Ind(X,U)=0.$$
\end{thm}
A special case of this result - but still an important step for the proof, and which we will use in our paper - was to consider a pair $X$ and $Y$ for which the collinearity locus was an annulus foliated by periodic orbits of $Y$ such that these orbits have stable (or unstable) \emph{manifolds}. In this case, the zero set of $X$ reduces to a single periodic orbit of $Y$, and $U$ is a tubular neighbourhood of this periodic orbit. The collinearity locus is not too big, but the stable manifolds, which comprise a surface foliation  to which $X$ is tangent, leave no room for $X$ to turn in all directions. {See Theorem~\ref{t.bsdenovo} for a precise statement and Figure~\ref{f.anneau}.}

The aim of the present paper is to address \emph{global} difficulties which may arise for an arbitrary configuration of the collinearity locus, but still taking advantage of the existence of invariant manifolds for periodic orbits.   

\begin{thmA}\label{t.valeurpropre} Let $X,Y$ be two commuting vector fields of class $C^3$ on a $3$-dimensional manifold $M$ and let $U\dans M$ be a relatively 
	compact open set such that $\Zero(X)\cap\partial U=\vide$. 
	Assume moreover that the properties below hold true
	\begin{itemize}
		\item $Y$ does not vanish in $U$;
		\item 
		if $\gamma\subset U$ is a periodic orbit of $Y$ contained in the collinearity locus $\Col_U(X,Y)$ then
		the Poincar\'e map of $\gamma$ has at least one eigenvalue of modulus different from $1$. 
	\end{itemize}
	Then the Poincar\'e-Hopf index of $X$ in $U$  vanishes: $\Ind(X,U)= 0$.
\end{thmA}

Theorem \ref{t.valeurpropre} combines the difficulties of two ideal cases. The first one is that of Theorem \ref{t:BS}. The second one is treated by our Theorem \ref{Th.MS} (see Section \ref{mainsectouille}). There we deal with a global dynamical configuration in $\Col_{U}(X,Y)$ and we need to introduce a new idea: a \emph{Glueing Lemma} to glue stable with unstable manifolds of heteroclinically connected $Y_t$-periodic orbits included in $\Col_U(X,Y)$ (see Lemma \ref{l.collage}). Non-trivial difficulties appear when trying to combine these two cases and we need techniques from partially hyperbolic dynamics, such as the Center Manifold Theorem, to treat them.

On the other hand, we can also prove the conjecture, for $C^3$ vector fields, with no hypothesis on the local dynamics, but imposing that the flow of vector field $Y$ preserves a transverse plane field.

\begin{thmA}\label{t.Normal} Let $X,Y$ be two commuting vector fields of class $C^3$ on a $3$-dimensional manifold $M$ and let $U\dans M$ be a relatively 
	compact open set such that $\Zero(X)\cap\partial U=\vide$. 
	Assume moreover that
	\begin{itemize}\item $Y$ does not vanish in $U$;
		\item the flow of $Y$ leaves invariant a $C^3$-plane field in $U$ transverse to $Y$.  
	\end{itemize}
	Then the Poincar\'e-Hopf index of $X$ in $U$ vanishes: $\Ind(X,U)= 0$.
\end{thmA}

The proof of this  theorem provides a clear illustration of our intuition: if $X$ commutes with the non-singular vector field $Y$, it cannot turn in all directions. As an easy consequence of the above result we obtain

\begin{cor}
	\label{suspensioncase}
	Let $\Sigma$ be a compact and boundaryless surface of class $C^3$, and $f\in\Diff^3(\Sigma)$. Let $M$ be the suspended manifold and $Y$ be the suspended flow on $M$.
	
	Assume that there exists a vector field $X$ on $M$ which is of class $C^3$ and commutes with $Y$. Then for every isolated compact set $K\dans\Zero(X)$ we have $\Ind(X,K)=0$.
\end{cor}

\begin{rem}
	One can easily check from the proof of Theorem \ref{t.Normal} that Corollary~\ref{suspensioncase} is in fact true when $f$ is of class $C^2$.
\end{rem}

\begin{rem}
	Corollary~\ref{suspensioncase} gives a natural class of examples to ``test'' the conjecture. One might think that it is too simple to study the topological behaviour of a vector field invariant under a suspension flow. Nevertheless, our proof is quite delicate and we are not aware of a direct argument to prove Corollary~\ref{suspensioncase}.  
\end{rem}

%

\paragraph{Overview of the article --} 
In Section \ref{s:strategy} we show how to use the $C^3$-hypothesis to perform a reduction in the proof of the {Conjecture}, showing that if there exists a counterexample to the $C^3$ alternative form then there exists special counter-examples, that we call \emph{prepared}. The goal of this reduction is to produce a foliation of $U$ by surfaces to which $Y$ is \emph{almost} tangent. Then, projecting down $Y$ on these surfaces, we shall benefit from arguments of surface dynamics. This will allow us to describe the dynamical and geometrical structure of the collinearity locus. 

The end of the section is devoted to a discussion of the ideas of our proofs and in the remaining sections we prove our theorems by contradiction, assuming the existence of a prepared counter-example.   

In Section \ref{suspoupouille} we prove Theorem~\ref{t.Normal}. The main feature of this case is that by our construction of prepared counterexamples, the vector field $Y$ preserves the leaves of the foliation. We show how the dynamics of $Y$ on the leaves prevents $X$ from turning. 

In Section \ref{mainsectouille} we prove Theorem~\ref{Th.MS}, which is the main step towards the proof of Theorem~\ref{t.valeurpropre}, where we introduce one of our main ideas: the \emph{Glueing Lemma}. In Section \ref{s:final_section} we prove Theorem~\ref{t.valeurpropre}.

{Section \ref{s.conclusion} is the conclusion. There we give the details of a strategy to treat the general $C^3$-case of the conjecture. We hope to explain there how far we are from an answer to the general case.
}

\section{Prepared triples}\label{s:strategy}

This section is devoted to the reduction of the proof of the conjecture, in its Alternative Form, to the treatment of special vector fields, which we call \emph{prepared}. {For these vector fields the} collinearity locus enjoys nice geometrical properties. For such prepared vector fields $X,Y$, we establish a simple formula to compute the index $\Ind(X,U)$. In the final paragraph we give the ideas of proofs of our main results.

\subsection{The Poincar\'e-Hopf index}
\label{sub_PHindex}

Let $M$ be a smooth manifold of dimension $d$ and $X$ be a $C^r$ vector field on $M$, $r\geq 1$. The \emph{set of zeros} of $X$ shall be denoted by $\Zero(X)$.

\paragraph{Index at an isolated zero --} Let $x\in M$ and $\phi:U\dans M\to\R^d$ be a local coordinate around $x$. Assume that $x$ is an isolated zero of $X$, i.e. that there exists a ball $B\dans U$ centred at $x$ such that $x$ is the only zero of $X$ in $\overline{B}$. In particular $X$ does not vanish on $\partial B$ and the following \emph{Gauss map} is well defined (the norm $||.||$ is chosen to be Euclidean in the coordinates given by $\phi$)
$$\alpha:\begin{array}[t]{ccl}
\partial B &\To  & \Ss^{d-1}\\
y      &\Map & \frac{X(y)}{||X(y)||}
\end{array}.
$$

By definition the \emph{Poincaré-Hopf index} of $X$ at $x$ is the topological degree of the Gauss map $\alpha$. It is independent of the choice of a ball $B$ and shall be denoted by $\Ind(X,x)$.

\paragraph{Index in an open set --} Now assume that $X$ does not vanish on the boundary $\partial U$ of a relatively compact open set $U$.

If $X'$ is a vector field close to $X$ in the $C^0$-topology, then $X'$ does not vanish on $\partial U$. Moreover one can choose $X'$ so that it has only a finite number of zeros inside $U$. The \emph{Poincar\'e-Hopf} index of $X$ in $U$ is the integer

$$\Ind(X,U)=\sum_{x\in\Zero(X')\cap U}\Ind(X',x).$$

This number is independent of the choice of such an $X'$, and we have $\Ind(X',U)=\ind(X,U)$ whenever $X$ and $X'$ are close enough in the $C^0$-topology. In particular, we have the following:

\begin{lem}
 \label{l.zerovazio}
Let $X,Y$ be two vector fields on $M$ of class $C^r$. There exists $\delta>0$ such that if $\Zero(X-c Y)\cap U=\emptyset$ for some $|c|<\delta$ then $\Ind(X,U)=0$.
\end{lem}

Assume that $\partial U$ is a codimension $1$ submanifold of $M$ on which $X$ does not vanish and that there is a continuous map associating to every $x\in\overline{U}$ a basis of $T_xM$ denoted by $\beta(x)=(e_1(x),e_2(x),...,e_d(x))$. This provides $\overline{U}$ with an orientation, and allows one to define the Gauss map of $X$ in $\partial U$. One can prove that $\Ind(X,U)$ is the topological degree of the Gauss map defined in $\partial U$, which is independent of the choice of the neighbourhood $U$ and of the basis $\beta$.

\paragraph{Index at an isolated compact set --} A compact set $K\dans\Zero(X)$ is said to be \emph{isolated} if there exists a neighbourhood $U$ of $K$, called \emph{isolating neighbourhood}, such that $\Zero(X)\cap U=K$ and $\Zero(X)\cap\partial U=\vide$.

The integer $\Ind(X,U)$ is independent of the isolating neighbourhood $U$ and shall be denoted by $\Ind(X,K)$. It shall be called the \emph{Poincar\'e-Hopf index at} $K$.

\begin{rem}
\label{r.aditive} 
Note that the index is additive: if $K_i\dans\Zero(X)$, for $i=1,...,n$ are disjoint isolated compact sets, then $\Ind(X,\cup_{i=1}^nK_i)=\sum_{i=1}^n\Ind(X,K_i)$.
\end{rem}

{
\subsection{General dynamical notions for vector fields}
\label{s.general_dynamics}

Let us recall some basic dynamical concepts and fix accordingly some notations that we shall use throughout the paper. 

In the following, given a $C^1$ vector field over $M$ we denote by $X_t$ the flow it generates. 
For any $x \in M$ and any interval $I \subset \R$, we also let $X_I(x):=\{X_t(x):t \in I\}$. In particular, we denote by $\mathcal O(x):=X_\R(x)$  the orbit of the point $x$ under $X$. 

Given $x\in M$, we consider its $\omega$-limit set $\omega(x)=\{y\in M;\exists\:t_n\to+\infty;X_{t_n}(x)\to y\}$. One defines similarly the $\alpha$-limit set of $x$ by replacing $+\infty$ with $-\infty$. Whenever clarity is required we write $\mathcal O_X(x),\alpha_X(x)$ and $\omega_X(x)$ to refer to the appropriate vector field. We say that a compact set $K\subset M$ is a \emph{minimal} set for $X$ if $X_t(K)=K$, for every $t$ and $x\in K$ implies $\omega(x)=K$. The classification of minimal sets for smooth surface vector fields of Denjoy and Schwartz \cite{Denjoy,Schwartz} will be a fundamental tool for us.

A point $x\in M$ is \emph{periodic} if there exists $T>0$ such that $X_T(x) = x$. In this case, we refer to the orbit of $x$ as being a \emph{periodic orbit}. 

If $\gamma=\mathcal O(x)$ is a periodic orbit we consider its \emph{stable set} $W^s(\gamma)=\{y\in M;\omega(y)=\gamma\}.$ We also consider the \emph{unstable set} $W^u(\gamma)=\{y\in M;\alpha(y)=\gamma\}$. Observe some of these sets can be empty for a given periodic orbit. If $W^s(\gamma)$ contains a neighbourhood of $\gamma$ we say that $\gamma$ is a \emph{sink}. If it is $W^u(\gamma)$ instead we say that $\gamma$ is a \emph{source}.}

\subsection{Basic properties of commuting vector fields}
\label{sub:basic_properties}

\paragraph{Commuting vector fields --}
One says that the two vector fields $X$ and $Y$ \emph{commute} if their Lie bracket vanishes everywhere, i.e. $[X,Y]=0$.

When $X$ and $Y$ are complete (for example when $M$ is compact), this is equivalent to the following equality holding true for every $s,t\in\R$
$$X_t\circ Y_s=Y_s\circ X_t.$$

\begin{center}
\emph{Until the end of this article all vector fields are complete.}
\end{center}

\paragraph{Normal component, quotient function --} Let $X$ and $Y$ be two commuting vector fields of class $C^r$, $r\geq 1$ on a manifold $M$. Suppose $U\dans M$ is a relatively compact open set where $Y$ does not vanish. Let $\Pi$ be a $C^r$ plane field in $\overline{U}$ transverse to $Y$. We can write
$$X=N+\mu Y$$
where $N$ is a $C^r$-vector field tangent to $\Pi$, called the \emph{normal component of $X$} and $\mu:\overline{U}\to\R$ is a real function of class $C^r$, called the \emph{quotient function}.

{
\begin{rem}
\label{rem.liberte}
These objects (normal component and quotient function) are \emph{not} canonical. { For example the choice of any Riemannian metric yields a plane field $\Pi$ transverse to $Y$ (the orthogonal plane field) and therefore yields a normal component and a quotient function. Even if they are not canonical,} they are important tools in the reduction we shall perform in this section. Whenever we need to modify our choice of the plane field $\Pi$ so that $N$ and $\mu$ have some desired property we shall do so and state it accordingly. In particular, in the proof of Theorem~\ref{t.Normal}, we shall consider the plane field which is invariant under the flow of $Y$, given by the assumption of that theorem, and we shall explore what this additional property gives about the associated normal component $N$ and quotient functions $\mu$.

Moreover, the fact that we have some freedom to choose these auxiliary objects is a subtle technical detail of our arguments that we would like to emphasise. Indeed, observe that for every $c\in\R$ one has 
{\[\Zero(X-cY)\cap\overline{U}\subset\mu^{-1}(c)\:\:\:\:\textrm{and}\:\:\:\:\Zero(N)=\bigcup_{c\in\R}\Zero(X-cY)\cap\overline{U},\]}
and that these facts hold \emph{for every choice of normal component and quotient function}, whenever both sides make sense. Thus, for instance, if we deduce a dynamical result about the sets $\Zero(X-cY)$ (which are the 1 dimensional orbits of the $\R^2$ action induced by $X$ and $Y$) we can change the function $\mu$ but the dynamical property remains true. In the same spirit, the set of zeros of a normal component does not depend on a particular choice. This simple observations will be used throughout the paper.
\end{rem}
}

\paragraph{The collinearity locus --} 
The \emph{collinearity locus} plays a fundamental role in our strategy. In all the paper, we will use the following notation

\begin{equation}
\label{eq:zero_U}
K=\Zero(X)\cap U\neq\vide.
\end{equation}

The collinearity locus is defined inside $\overline{U}$ as
$$\Col_U(X,Y)=\bigcup_{c\in\R}\Zero(X-cY)\cap \overline{U}.$$
Since $X$ and $Y$ commute, the sets 
\begin{equation}
\label{eq:zero_c}
K_c=\Zero(X-cY)\cap \overline{U}
\end{equation}
are $Y_t$-invariant, and form a partition of the collinearity locus by orbits of $Y_t$.

\paragraph{Level sets --} The study of the partition of $U$ given by level sets $\mu^{-1}(c)$ will be crucial in the paper. In particular, part of our simplification will consist in coming down to the case where it is a foliation  by surfaces, which will allow us to reduce the dimension, and use arguments from surface dynamics. Note that, { as was already mentioned in Remark \ref{rem.liberte},} for every parameter $c$ we have
$$K_c\dans\mu^{-1}(c).$$

The normal component and quotient function depend on a transverse plane field $\Pi$ in particular there is no reason why $Y$ should leave them invariant. The next paragraph gives a precise analysis of this defect of invariance.

\paragraph{Holonomies --} Let $\Sigma_0,\Sigma_1\dans U$ be two cross sections of $Y$ tangent to $\Pi$ at $x_0$ and $x_1$, such that there exists a holonomy map $P:\Sigma_0\to\Sigma_1$ along $Y$ {with $P(x_0)=x_1$}.

The \emph{hitting time} is the function $\tau:\Sigma_0\to(0,\infty)$ defined by
\begin{equation}\label{eq.hitting_time}
P(x)=Y_{\tau(x)}(x).
\end{equation}

The next lemma states two fundamental consequences of the identity $[X,Y]=0$. The first one is the invariance of the normal component by holonomy. The second one relates the defect of invariance of $\mu$ with the variation of the hitting time. The proofs can be found in \cite[Corollary 5.6, Lemma 5.7]{BS} as well as in the third author's PhD thesis (see \cite{Santiague_these})

\begin{lem}
\label{calculusholonomies}
Assume that $X$ and $Y$ commute. {Let $\Sigma_0,\Sigma_1$ be two transverse sections of $Y$ such that there exists a holonomy map $P:\Sigma_0\to\Sigma_1$. Assume that $\Sigma_0$ and $\Sigma_1$ are tangent to $\Pi$ at $x$ and $P(x)$ respectively.} Then
\begin{enumerate}
\item $D_xP\,N(x)=N(P(x));$
\item $-D_x\tau\,N(x)=\mu(P(x))-\mu(x).$
\end{enumerate}
\end{lem}

\subsection{Prepared triples}
\label{sub:prepared_ce}

\paragraph{Simplification of triples --} Consider a $C^r$-\emph{triple} $(U,X,Y)$ (with $r\geq 1$). It is the data of $U$, a relatively compact open set of a $3$-dimensional Riemannian manifold $M$, and of $X,Y$, two commuting vector fields of class $C^r$ satisfying
$$\Zero(Y)\cap U=\Zero(X)\cap\partial U=\vide.$$

Let $\Pi$ be a plane field {transverse} to $Y$. Consider the normal component $N$ and the quotient function $\mu$ of $X$ so the following equality holds in $\overline{U}$
$$X=N+\mu Y.$$

Such a triple is a \emph{counterexample to the Conjecture} if we have
$$\Ind(X,U)\neq 0.$$

{Our goal now is, \emph{assuming the existence of a counterexample to the conjecture}, to show that there exists a counterexample to the conjecture having some additional properties which will give us a more detailed description of the collinearity locus. The properties we are seeking are summarized in the definition below.}    

\begin{dfn}[Prepared triples]
\label{d:prepared_triple}
 Let $(U,X,Y)$ be a $C^r$-triple, and $N,\mu$ be respectively the normal component and the quotient function. We say that $(U,X,Y)$ is a ($C^r$)-\emph{prepared triple} if
\begin{enumerate}
\item $\overline{U}$ is trivially foliated by level sets $\mu^{-1}(c)$, $c\in[-\eps,\eps]$, which are diffeomorphic to the same compact and connected surface $S$, possibly with boundary.
\item $Y$ is nowhere orthogonal to the level sets $\mu^{-1}(c)$.
\item $0$ is a continuity point of the map $Z:c\mapsto K_c=\Zero(X-cY)\cap\overline{U}$ in the Hausdorff topology.
\item For every $c\in[-\eps,\eps]$, $K_c\cap\partial\mu^{-1}(c)=\vide$
\end{enumerate}
We say that a $C^r$-triple $(U,X,Y)$ is a $C^r$-\emph{prepared counterexample} to the conjecture if this is a prepared triple satisfying
$$\Ind(X,U)\neq 0.$$
\end{dfn}

{Since this definition is somewhat technical, let us give an informal explanation of what it means. If $(U,X,Y)$ is a $C^r$-prepared counterexample then the level sets $\mu^{-1}(c)$ of the quotient function are compact surfaces almost tangent to $Y$, {which trivially foliate $U$}. Moreover, the ``level $c$ of the collinearity locus'', i.e. the compact set $K_c=\Zero(X-cY)\cap U$, is included in the level $\mu^{-1}(c)$ and does not touch its boundary. {This means that the partition of $\Col_U(X,Y)$ by sets $K_c$ can be extended to a foliation of $U$ by surfaces}. One important issue is Item 3, which concerns the dependence of the sets $K_c$ with respect to $c$. Although in general it will not be continuous, for a prepared counterexample, { Item 3 provides some control on the configuration of $K_c$ for $c$ close to $0$}. In particular, all $K_c$ for $c$ close to $0$ must be included in a ``tubular'' neighbourhood of $K$. The simplest yet non trivial image to have in mind is depicted below in Figure 1. 

The main reason to consider prepared triples is that we can project $Y$ {orthogonally} to the level sets of the quotient function. The dynamics of this projected vector field will allow us to give a very precise dynamical description of the sets $K_c$ (see \S~\ref{sub:simplif_col_loc} below).}

\begin{rem}
The last item in the definition above is of course empty if level sets of $\mu$ are boundaryless. We will show in Corollary  \ref{c.index_zero} that there is no prepared counterexample satisfying that level sets are boundaryless.
\end{rem}

\paragraph{Prepared counterexamples --} The first step of our strategy is,  assuming that there is a counterexample to the Conjecture, to construct a prepared counterexample. The remainder (and also the majority) of the work will then consist in proving that such a prepared counterexample  does not exist.

\begin{thm}[Simplification of counterexamples]
\label{t:prepared}
Assume the existence of a $C^3$-counterexample $(U,X,Y)$ to the Conjecture. Then there exists $\widetilde{U}\dans U$ and a $C^3$-vector field $\widetilde{X}$ commuting with $Y$ such that 
$$\Col_{\widetilde{U}}(\widetilde{X},Y)=\Col_{U}(X,Y)\cap\widetilde{U}$$
and $(\widetilde{U},\widetilde{X},Y)$ is a $C^3$-prepared counterexample to the conjecture.
\end{thm}

\begin{rem}
It will become clear in the next paragraphs that the $C^3$-hypothesis made in Theorem \ref{t:prepared} is fundamental.
\end{rem}

\begin{rem}
\label{rem_prep_ex}
In the theorem above we modify the open set $U$ and the vector field $X$ (in the direction of $Y$) but we keep $Y$ and the collinearity locus unchanged. Hence if we assume the existence of a counterexample to the conjecture, adding some additional hypothesis on $Y$ or on $\Col_U(X,Y)$, then there exists a prepared counterexample satisfying the same hypothesis. In particular, if there exists a counterexample to Theorem \ref{t.valeurpropre} or \ref{t.Normal}, then there exists a prepared counterexample to that theorem.
\end{rem}

\paragraph{Orientability issues --} We show below how to get down to the case where our prepared counterexamples, if they exist, may be chosen with good orientability properties.

\begin{prop}
\label{p:orientable}
Let $(U,X,Y)$ be a prepared counterexample to Theorem \ref{t.valeurpropre} (resp. Theorem \ref{t.Normal}). Then there exists a prepared counterexample $(\widetilde{U},\widetilde{X},\widetilde{Y})$ to that theorem such that
\begin{enumerate}
\item The vector field $\widetilde{Y}$ is transversally oriented.
\item Level sets $\widetilde{\mu}^{-1}(c)$ are oriented, where $\widetilde{\mu}$ denotes the quotient function.
\end{enumerate}
\end{prop}

\begin{proof}
Take $\Pi:\widetilde{U}\to U$, a double cover. The lifted vector fields $\widetilde{X}$ and $\widetilde{Y}$ still commute and the index of $\widetilde{X}$ in $\widetilde{U}$ is multiplied by $2$.

Let $x\in\per(Y)$. Every $\widetilde{x}\in\Pi^{-1}(x)$ is a periodic point of $\widetilde{Y}$. Let $P$ and $\widetilde{P}$ denote the corresponding Poincar\'e maps. 
Then the eigenvalues of $D_{\widetilde{x}}\widetilde{P}$ coincide with those of $D_xP$ 
(if its orbit projects down $1:1$) or with their squares (if it projects down
 $2:1$). Moreover the pull-back of a $Y_t$-invariant plane field, if it exists, is a $\widetilde{Y}_t$-invariant plane field.

Therefore, if needed, we can consider successively a double orientation cover for $Y^{\perp}$ and a double orientation cover for the foliation by level sets.
\end{proof}

\subsection{Simplifying $\Zero(X)$}
\label{sub:simplif_col_loc}

From now on $(X,Y,U)$ will be a triple in the sense of \S \ref{sub:prepared_ce}. As we have already mentioned $\Col_U(X,Y)$ is partitioned by sets $K_c$ which are saturated by $Y_t$ and form therefore a lamination. The first idea is to modify $X$ without changing $\Ind(X,U)$ so that this lamination may be extended to a foliation of $U$ by surfaces (possibly with boundary). For this we will use Sard's theorem.

\paragraph{A foliation containing $\Col_U(X,Y)$ --} Here, we must assume that $X$ and $Y$ are of class $C^3$. Endow $M$ with a Riemannian metric $g$ and let $\Pi$ be the $C^3$ plane field normal to $Y$ ({or use any plane field transverse to $Y$ given a priori}). We write
$$X=N+\mu Y$$
where $N$ is the normal component corresponding to $\Pi$ and $\mu:\overline{U}\to\R$ is the corresponding quotient function. These objects are of class $C^3$.

\begin{figure}[h!]
\label{f.satanou}
\centering
\includegraphics[scale=0.6]{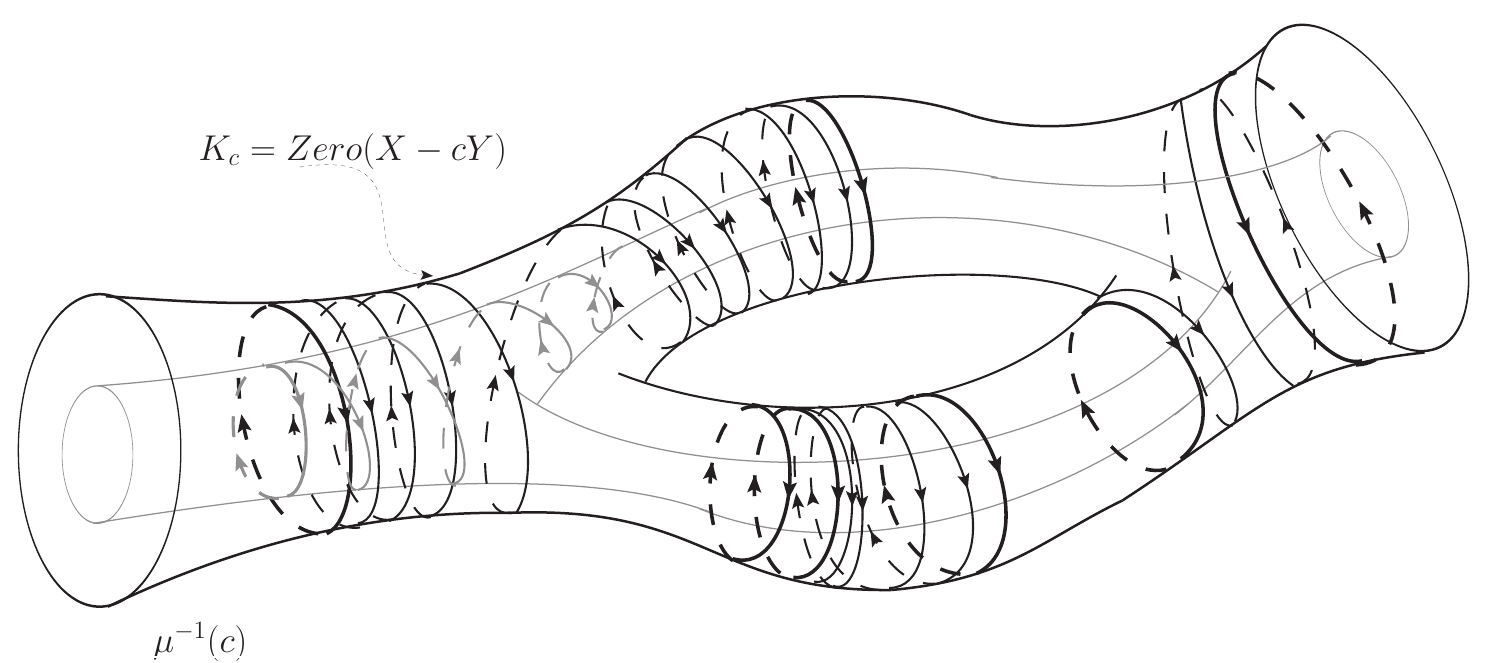}
\caption{The collinearity locus is tangent to a surface foliation}
\end{figure}

\begin{lem}
	\label{l:munonzero}
	There exists $\delta>0$ such that if $\mu^{-1}(c)=\emptyset$, for some $|c|<\delta$, then $\Ind(X,U)=0$ 
\end{lem}
\begin{proof}
	Apply Lemma~\ref{l.zerovazio} to $X$ in order to find the number $\delta>0$. The result now follows from the inclusion $K_c=\Zero(X-cY)\cap\overline{U}\subset\mu^{-1}(c)$. 
\end{proof}
In particular we may rule out the case where $\mu$ is constant in some neighbourhood of $K$.

\begin{lem}
	\label{l:sardouille}
	For every $\eps_0>0$ there exists $0<\eps_1<\eps_2<\eps_0$ such that $  [\eps_1,\eps_2]$ is an interval containing only regular values of $\mu$.
\end{lem}

\begin{proof}
	Consider the function $\mu:U\to\R$. This is a function of class $C^3$ between a $3$-dimensional manifold and a $1$-dimensional one, which by assumption vanishes in $U$ and is not constant on $U$, after Lemma~\ref{l:munonzero}.
	By Sard's theorem there exist regular values, which can be made positive (up to changing $X$ by $-X$ which still commutes with $Y$ and doesn't vanish on $\partial U$) and arbitrarily small. Since $\overline{U}$ is compact, the set of regular values is an open set, the lemma follows.
\end{proof}

\begin{cor}
	\label{l:modifyX}
	There exist $0<\eps<\xi$, such that the following properties hold true:
	\begin{enumerate}
		\item $\Zero(X-\xi Y)\cap\partial U=\emptyset$;
		\item $\Ind(X-\xi Y,U)=\Ind(X,U)$;
		\item $[\xi-\eps,\xi+\eps]$ contains only regular values of $\mu$. In particular, each connected component of $\mu^{-1}([\xi-\eps,\xi+\eps])$ is diffeomorphic to a product $[\xi-\eps,\xi+\eps]\times S$, where $S$ is a compact surface, possibly with boundary.
	\end{enumerate}
\end{cor}

\begin{proof}
	The index of vector fields in $U$ being locally constant there exists $\eps_0>0$ such that the first two properties hold for every $\xi\in [0,\eps_0]$. 
	Using Lemma~\ref{l:sardouille}, one obtains that $[0,\eps_0]$ contains an interval $[\eps_1,\eps_2]$ of regular values. 
	Moreover, by Lemma~\ref{l:munonzero} if $\eps_0$ is small enough, $\mu$ is surjective on this interval. We take $\xi=(\eps_1+\eps_2)/2$ and $\eps=(\eps_1-\eps_2)/2$.
	It follows that $\mu^{-1}([\eps_1,\eps_2])$ is diffeomorphic to a finite union of the form $\cup_i[\eps_1,\eps_2]\times S_i$, where each $S_i$ is a connected component of $\mu^{-1}(\xi)$ and is a compact surface, which might have boundary components. 
\end{proof}

Note moreover that $X^{(1)}=X-\xi Y$ commutes with $Y$. Hence if $(U,X,Y)$ was a counterexample, and $U^{(1)}=\mu^{-1}(\eps_1,\eps_2)\dans U$ then $(U^{(1)},X^{(1)},Y)$ is still a counterexample. Note that we did not change $Y$ nor $\Col_U(X,Y)$.

For that reason \emph{we may assume that the number $\xi$ we just found out is in fact equal to $0$ and that $U=\mu^{-1}(\eps,\eps)$ is trivially foliated by level sets $\mu^{-1}(c)$}.

Let $U_1,...,U_n$ be the connected components of $U$. Each $U_i$ is foliated by connected compact surfaces, which are the connected components of level sets $\mu^{-1}(c)$. The additive property of the index implies $\Ind(X,U)=\sum_{i=1}^n\Ind(X,U_i)$. So if $(U,X,Y)$ is a counterexample, there exists $i$ such that $(U_i,X,Y)$ is a counterexample satisfying the first item of Definition \ref{d:prepared_triple}, i.e. we are down to the case where \emph{$\overline{U}$ is trivially foliated by compact and connected surfaces, the level sets $\mu^{-1}(c)$, such that for every $c$, $K_c=\Zero(X-cY)\cap\overline{U}\dans\mu^{-1}(c)$.}

\paragraph{Projecting $Y$ on level sets --} As we mentioned before $Y$ needs not be tangent to level sets. However we saw that $\Col_U(X,Y)$ is saturated by the orbits of $Y$. So a continuity argument shows that in a neighbourhood of $\Col_U(X,Y)$, $Y$ is quasi-tangent to the level sets.

\begin{lem}
\label{l.anglouille}
There exists $W\dans \overline{U}$, a neighbourhood of $\Col_U(X,Y)$, such that for every $c\in(-\eps,\eps)$ if $x\in W\cap\mu^{-1}(c)$ then
$$\left|\sphericalangle(Y(x),T_x[\mu^{-1}(c)])\right|\leq\frac{\pi}{4}.$$
\end{lem}

\begin{proof}
Since $[X,Y]=0$ the collinearity locus $\Col_U(X,Y)$ is $Y_t$-invariant. We deduce that $Y$ is tangent to $\mu^{-1}(c)$ at every point of $K_c$. The lemma clearly follows from the continuity of the vector field.
\end{proof}

Such a neighbourhood of $\Col_U(X,Y)$ is a neighbourhood of $K$ and $X$ does not vanish on its boundary. This allow us to construct from a counterexample $(U,X,Y)$ such as constructed in the previous paragraph a new counterexample $(U^{(2)},X^{(2)},Y)$ satisfying the first and the second item of Definition \ref{d:prepared_triple}. Here again, we did not change $Y$ nor the collinearity locus.

\subsection{Semi-continuity for the Hausdorff topology}
\label{sub:semi_cty} 

We will also need a genericity argument to reduce our study to the case where $0$ is a continuity point of $Z:c\mapsto K_c=\Zero(X-cY)\cap\overline{U}$ for the so-called \emph{Hausdorff topology} that we introduce below.

\paragraph{Definition and semi-continuity lemma --} Recall that the set $\cK$ of compact subsets of a compact metric space $(\cX,\dist)$ is in itself a compact metric space when endowed with the \emph{Hausdorff distance} $\dist_H$ defined as follows. If $K,L$ are compact subsets of $\cX$, $\dist_H(K,L)$ is the maximum of the two numbers $\dist_K(L)$ and $\dist_L(K)$ where by definition
$$\dist_K(L)=\sup_{x\in L}\dist(x,K).$$

\begin{dfn}[Lower semi-continuity]
\label{d.lcs}
Let $(\cY,d)$ be a metric space. Say a function $Z:\cY\to\cK$ is \emph{lower semi-continuous} at $y_0\in\cY$ if for every open set $V\dans\cX$ intersecting $Z(y_0)$ there exists $\delta>0$ such that if $y\in\cY$ 
satisfies $d(y_0,y)<\delta$, then
$$Z(y)\cap V\neq\vide.$$
\end{dfn}

\begin{dfn}[Upper semi-continuity]
\label{d.ucs}
Let $(\cY,d)$ be a metric space. Say a function $Z:\cY\to\cK$ is \emph{upper semi-continuous} at $y_0\in\cY$ if for every neighbourhood $V\dans\cX$ of $Z(y_0)$ there exists $\delta>0$ such that if $y\in\cY$ satisfies $d(y_0,y)<\delta$, then
$$Z(y)\dans V.$$
\end{dfn}

Note that this notion is coherent: a function $Z:\cY\to\cK$ which is both lower and upper semi-continuous at a point is continuous at this point (with respect to the Hausdorff distance). For the next result, we refer to \cite[pp. 70-71]{Kurouille}. See also \cite{Santiague_note}, which is an unpublished note of the third author, for a proof in a slightly more general context (i.e. assuming only the separability of $\cY$). Recall that a subset of a topological space is called \emph{residual} if it can be written as a countable intersection of dense open sets. 

\begin{thm}[Semi-continuity Lemma]
\label{t.sclemma}
Let $(\cX,\dist)$ and $(\cY,d)$ be two compact metric spaces, and $\cK$ be the space of compact subsets of $\cX$ endowed with the Hausdorff distance. Let $Z:\cY\to\cK$ be upper (resp. lower) semi-continuous at all points $y\in\cY$. Then the set of continuity points of $Z$ is residual, and hence dense by Baire's theorem.
\end{thm}

\paragraph{Regularity of $\Zero(X-cY)$ --}
Let us study how the compact set consisting of elements $x\in\overline{U}\dans M$ such that $X(x)=c Y(x)$ varies with the parameter $c$.

\begin{lem}
\label{l.scinferiously}
The map
$$Z:c\in\R\mapsto K_c=\Zero(X-cY)\cap \overline{U}$$
is upper semi-continuous in the Hausdorff topology.
\end{lem}

\begin{proof}
Let $c\in(-\eps,\eps)$ and $c_n\in(-\eps,\eps)$ be a sequence converging to $c$. Consider a sequence $x_n\in Z(c_n)$, in such a way that $X(x_n)=c_nY(x_n)$. Any accumulation point $x$ of $x_n$ must belong to $\overline{U}$ and satisfy $X(x)=cY(x)$ and so, must belong to $Z(c)$. 

Now, assume by contradiction that $Z$ is not upper semi-continuous at $c\in(-\eps,\eps)$. Then, there exists $V$ a neighbourhood of $Z(c)$ and $(c_n)_{n\in\N}$, a sequence converging to $c$, such that there exists a sequence of points $x_n\in Z(c_n)\setminus\overline{V}$. However, since all accumulation points of $x_n$ must belong to $Z(c)$, we have $x_n\in V$ for $n$ large enough, which is absurd.
\end{proof}

\begin{cor}
\label{c.ctinuitypointsgeneric}
The set of continuity points of $Z$ is a residual subset of {some interval in} $\R$.
\end{cor}

\begin{rem}
\label{r.othercpcts}
For every compact subset $F\dans\overline{U}$ the function $c\in\R\mapsto\Zero(X-cY)\cap F$ is upper semi-continuous and the set of its continuity points is residual inside $\R$. This map will always be denoted by $Z$.
\end{rem}

As we shall see, the continuity properties of the function $Z$ have some important dynamical consequences. Let us state right now a simple one.

\begin{lem}[No sinks/sources]
 \label{l.nosinks}
Let $X,Y$ be two commuting vector fields. Assume that there exists a periodic orbit for $Y$, included in $\Zero(X-cY)\cap U$, which is a sink or a source. Then $c$ is a discontinuity point of $Z$. As a consequence, the set of $c\in\R$ such that $\Zero(X-cY)\cap U$ does not contain such a sink or source is residual.
\end{lem}

\begin{proof}
Suppose there exists (say) a sink $\gamma$ of $Y$ contained in $K_c=\Zero(X-cY)\cap U$, for some $c\in\R$. Let $\cU\dans U$ be a neighbourhood of $\gamma$ contained in $W_Y^s(\gamma)$. In particular, for every $x\in\cU$ we have $\omega_Y(x)=\gamma$.

Supppose the existence of $x\in\Zero(X-c'Y)\cap\cU$. On the one hand we have $\omega_Y(x)=\gamma\dans\Zero(X-cY)$. On the other hand $\Zero(X-c'Y)$ is $Y_t$-invariant since $X$ and $Y$ commute, so we have $\omega_Y(x)\dans\Zero(X-c'Y)$. Since $Y$ does not vanish in $\cU\dans U$ we have $\Zero(X-cY)\cap\Zero(X-c'Y)\cap \cU=\vide$ unless $c=c'$.

We deduce that there is no $c\neq c'$ satisfying $\Zero(X-c'Y)\cap\cU\neq\vide$. This implies that $Z$ is not lower semi-continuous at $c$. This proves the first part of the lemma.

Since continuity points of $Z$ are residual, the second part of the lemma follows.
\end{proof}

\paragraph{Construction of prepared counterexamples --} We are now in position to prove Theorem \ref{t:prepared}. Given a $C^3$-counterexample, we know how to construct a counterexample $(U,X,Y)$ satisfying the first two properties of Definition \ref{d:prepared_triple} without modifying $Y$ nor the collinearity locus.

Corollary \ref{c.ctinuitypointsgeneric} says that continuity points of $Z$ are dense in $(-\eps,\eps)$. Hence we can perform, as we have already done in \S \ref{sub:simplif_col_loc}, a small modification of $X$ in the direction of $Y$ without changing the index, $Y$ or the collinearity locus, so that $0$ is a continuity point. Shrinking $U$ we obtain a new counterexample $(U^{(3)},X^{(3)},Y)$ satisfying the first three properties of the definition.

Now we just have to see how to deduce the fourth property from the others. This is a continuity argument.
 Since $Z$ is upper semi-continuous at $0$ and since $K=Z(0)$ is disjoint from $\partial U$ we must have $Z(c)\cap\partial U=\vide$ for $|c|$ smaller than some positive number $\eps'$. 
 Let $\widetilde{U}$ denote $\mu^{-1}(-\eps',\eps')$ and $\widetilde{X}=X^{(3)}$. The triple $(U,X,Y)$ is a prepared counterexample and Theorem \ref{t:prepared} is proven.

\subsection{An index formula}

We will now prove a general index formula which simplifies the computation of $\Ind(X,U)$ for prepared triples. It generalizes a formula which was used in \cite{BS} in a crucial way (see  \cite[Proposition 5.9]{BS}).

We assume that $(U,X,Y)$ is a prepared triple (see Definition \ref{d:prepared_triple}). Let us use coordinates $\x=(x,c)\in S\times[-\eps,\eps]$ to describe a point of $\overline{U}$, $S$ being a connected compact surface possibly with boundary. We will also require the vector field $Y$ to be transversally oriented and $S$ to be oriented in $U$ (up to multiplying $\Ind(X,U)$ by an integer: see the proof of Proposition \ref{p:orientable}).

The second item of Definition \ref{d:prepared_triple} implies that $Y$ is never orthogonal to the surface $S\times\{c\}$. Note that $\Zero(N)=\Col_U(X,Y)$ so our last reduction implies
\begin{equation}
\label{eq:ZeroNbord}
\Zero(N)\cap(\partial S\times[-\eps,\eps])=\vide.
\end{equation}

In coordinates $\x=(x,c)$ one can write
$$X(\x)=N(\x)+cY(\x).$$

\paragraph{Choice of a basis --} We will set $e_3=Y$. The vector field $e_3$ is never orthogonal to $S$.  Using the orientation of the surface implies that there exists a vector field $e_2$ tangent to $S$ which is orthogonal to $Y$. Finally define $e_1=e_2\wedge e_3$. This define a continuous basis $\x=(x,c)\mapsto\beta(\x)=(e_1(\x),e_2(\x),e_3(\x))$ of $T_\x\overline{U}$. Since $(e_1,e_2)$ is a basis of $Y^\perp$ we may write
$$X(\x)=\alpha_1(\x) e_1(\x)+\alpha_2(\x) e_2(\x)+ce_3(\x),\,\,\,\,\,N(\x)=\alpha_1(\x)e_1(\x)+\alpha_2(\x)e_2(\x).$$

Consider the Gauss maps in $\partial U$ given by

$$\alpha:\begin{array}[t]{ccl}
\partial U &\To  & \Ss^2\\
\x=(x,c)      &\Map & \frac{(\alpha_1(\x),\alpha_2(\x),c)}{\sqrt{\alpha_1(\x)^2+\alpha_2(\x)^2+c^2}}
\end{array}.
$$
and, recalling that $N\neq 0$ on $\partial S\times \{0\}$,
$$\nu:\begin{array}[t]{ccl}
\partial S\times\{0\} &\To  & \Ss^1\\
\x=(x,0)      &\Map & \frac{(\alpha_1(\x),\alpha_2(\x))}{\sqrt{\alpha_1(\x)^2+\alpha_2(\x)^2}}
\end{array}.
$$
Here we see the oriented circle $\Ss^1$ embedded in $\Ss^2$ as the equator $\{c=0\}$ oriented as the boundary of the northern hemisphere $\{c>0\}$.

{\begin{rem}\label{r.deg_easier}
Note that with our choice of coordinates, the vector fields $e_1$, $e_2$ and $e_3$ are at least $C^1$ so the Gauss map $\alpha$ itself is $C^1$. This makes the computation of $\deg(\alpha)$ easier.
\end{rem}
}
\paragraph{The index formula --} Let $\gamma$ be a connected component of $\partial S$ oriented with the boundary orientation.

\begin{dfn}[Linking number]\label{d:linknumb}
The topological degree of the restriction of $\nu$ to $\gamma$ is called the \emph{linking number} of $N$ along $\gamma$. We denote it by $l(\gamma)$.
\end{dfn}

\begin{figure}[hbtp]

\centering
\includegraphics[scale=0.55]{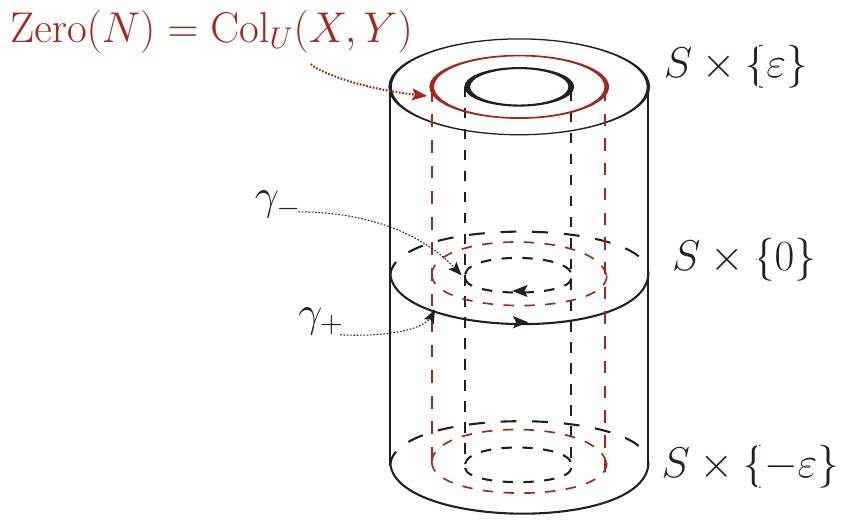}
\caption{Case where $S$ is an annulus. Here $\partial S$ is a disjoint union of two curves $\gamma_+$ and $\gamma_-$, oriented with the boundary orientation. The formula becomes $\Ind(X,U)=l(\gamma^+)+l(\gamma_-)$.}
\label{des_indice}
\end{figure}

The principal theorem of this section is

\begin{thm}[The Index Formula]\label{t:indexf}
$$\Ind(X,U)=\sum_{\gamma} l(\gamma)$$
where the sum is taken on all the boundary components $\gamma$ of $S$, oriented with the boundary orientation.
\end{thm}

\begin{rem}
In \cite[Proposition 5.9]{BS} the authors prove this formula \emph{when $S$ is an annulus} (see figure \ref{des_indice}). The presentation of the proof we give here simplifies the one given in that reference.
\end{rem}

\begin{proof}[Proof of Theorem \ref{t:indexf}]
{We decompose the boundary of $U$ as
$$\partial U=\left(\bigcup_{\gamma} C_\gamma \right)\cup S_\eps\cup S_{-\eps},$$
where the union is taken on all boundary components of $S$ and where we have set $C_\gamma=\gamma\times[-\eps,\eps]$ and $S_{\pm\eps}=S\times\{\pm\eps\}$. These surfaces have disjoint interiors and if two of them have nonempty intersection, it must be one of the $\gamma\times\{\pm\eps\}$ where $\gamma$ is a boundary component of $S$. We claim that

\begin{equation}
\label{eq.additouille_degrouille}\deg(\alpha)=\sum_{\gamma}\deg\left(\alpha|_{C_\gamma}\right)+\deg\left(\alpha|_{S_\eps}\right)+\deg\left(\alpha|_{S_{-\eps}}\right).
\end{equation}
Recall that $\alpha$ is $C^1$ (see Remark \ref{r.deg_easier}) so $\deg(\alpha)$ can be computed as the number of preimages of any regular value of $\alpha$ counted with multiplicity: \cite[Chapter 5]{Milnor}. Note also that there exists a regular value of $\alpha$ without a preimage in $\bigcup_\gamma \gamma\times\{\pm\eps\}$. This uses Sard's theorem, and the fact that $\alpha(\bigcup_\gamma \gamma\times\{\pm\eps\})$ has zero Lebesgue measure {(recall that the image of a compact $1$-dimensional manifold by a $C^1$ map between two surfaces has zero Lebesgue measure)}. So its preimages are in $\bigsqcup_\gamma\Int(C_\gamma)\sqcup \Int(S_\eps)\sqcup \Int(S_{-\eps})$. Using the definition of degree given above and this regular value we deduce that Formula \eqref{eq.additouille_degrouille} holds.

The Gauss map $\alpha$ sends $S_\eps$ and $S_{-\eps}$ inside the northern and southern hemisphere respectively (this is because $\mu>0$ on $S_\eps$ and $\mu<0$ on $S_{-\eps}$). In particular its topological degree is zero in restriction to these surfaces.

Now let us choose a component $\gamma$ and compute $\deg(\alpha|_{C_\gamma})$. Let us consider
$$\alpha_s(x,c)=\frac{(\alpha_1(x,(1-s)c),\alpha_2(x,(1-s)c),(1-s)c)}{\sqrt{\alpha_1(x,(1-s)c)^2+\alpha_2(x,(1-s)c)^2+(1-s)^2c^2}}.$$
We use here that $N$ does not vanish on $C_\gamma\dans(\partial S\times[-\eps,\eps]$. This defines a retraction of $\alpha$ to $\nu$. We deduce that
$$\deg\left(\alpha|_{\gamma\times(-\eps,\eps)}\right)=l(\gamma).$$
This ends the proof.
}
\end{proof}

\begin{cor}
\label{c.index_zero}
With the previous hypotheses, assume that $S$ is a boundaryless surface. Then
$$\Ind(X,U)=0.$$
\end{cor}

\subsection{Dynamics of the projected vector field}

Let $(U,X,Y)$ be a $C^3$-prepared counterexample. It will be very useful to define a new non-vanishing vector field of $U$, denoted by $Y'$, which is tangent to the level sets by setting
$$Y'(x)=\proj_{T_x[\mu^{-1}(c)]}(Y(x)),$$
$\proj_{T_x[\mu^{-1}(c)]}$ denoting the orthogonal projection on $T_x[\mu^{-1}(c)]$.

\begin{rem}
Since at each point of $\Col_U(X,Y)$, $Y$ is tangent to the corresponding level set, we have $Y=Y'$ on $\Col_U(X,Y)$.
\end{rem}

The vector field $Y'$ is of class $C^3$ on level sets $\mu^{-1}(c)$. Introducing the vector field $Y'$ has the following interest. The dynamics of a non-vanishing vector field of class $C^r$, $r\geq 2$, on a surface is rather simple thanks to the \emph{Denjoy-Schwartz theorem}, which we shall use here in the following version (see the Corollary in page 457 of \cite{Schwartz}) 
{
\begin{thm}[Denjoy-Schwartz \cite{Denjoy,Schwartz}]
	\label{t.diegosalgado}
Let $S$ be an oriented surface and $V$ a $C^2$ non-vanishing vector field on $S$. Assume that $S$ is not a minimal set for $V$. Then, for every $x\in S$ its $\omega$ and $\alpha$-limit sets are periodic orbits.
\end{thm}

The main consequence of this result for us is collected below. Recall our notation for $K=\Zero(X)\cap U\subset\mu^{-1}(0)$.}

\begin{lem}
\label{l:Denjouill_Scwhartzouille}
Let $(U,X,Y)$ be a $C^3$-prepared triple. There exists a $C^3$-prepared triple $(\widetilde{U},X,Y)$ with $\widetilde{U}\dans U$ (we denote by $\widetilde{\mu}$ the restriction of $\mu$ to $\widetilde{U}$) such that the following dichotomy holds true
\begin{enumerate}
\item either $K=\widetilde{\mu}^{-1}(0)$, in which case $\widetilde{\mu}^{-1}(0)$ is boundaryless;
\item or $\partial\widetilde{\mu}^{-1}(0)\neq\vide$ and for every $x\in K$, $\alpha_Y(x)$ and $\omega_Y(x)$ are periodic orbits of $Y$.
\end{enumerate}
\end{lem}
\begin{proof}
{Let us consider the dynamics of the non-vanishing vector field $Y'$ on the surface $\mu^{-1}(0)$. Distinguish two cases. 

\emph{Case 1. $K=\mu^{-1}(0)$}. Since $(U,X,Y)$ is prepared we have $K\cap\partial\mu^{-1}(0)=\vide$ (see Item 4. of Definition \ref{rem.liberte}) so and $\partial\mu^{-1}(0)=\vide$.

\emph{Case 2. $K\subsetneq\mu^{-1}(0)$}. Since $K$ is closed and invariant, $\mu^{-1}(0)$ is not a minimal set of $Y'|_{\mu^{-1}(0)}$. In particular Theorem \ref{t.diegosalgado} implies that that for every $x\in K$, $\alpha_Y(x)=\alpha_{Y'}(x)$ and $\omega_Y(x)=\omega_{Y'}(x)$ are periodic orbits of $Y'$ and $Y$ (note that $Y'=Y$ when restricted to $K$). A priori $\mu^{-1}(0)$ could be boundaryless (if it has boundary we are done). In that case, since $K$ is strictly contained in the level set, we can take a small neighbourhood $W$ of $K$ inside $\mu^{-1}(0)$ (with boundary) and  consider $\widetilde{U}=W\times[-\eps,\eps]\dans U$. Since $\partial W\neq\emptyset$, by Theorem~\ref{t.diegosalgado}, Item 2 is satisfied as well as the conditions of Definition \ref{d:prepared_triple}. This ends the proof.}
\end{proof}

\begin{rem} In the first case we know by Corollary \ref{c.index_zero} that $\Ind(X,U)$ must be zero. {So if there is a prepared counterexample, we are placed in the second case of Lemma~\ref{l:Denjouill_Scwhartzouille}. By continuity of the foliation by levels, we know that $\partial\widetilde{\mu}^{-1}(c)\neq\emptyset$ for every $c$ small and applying Theorem~\ref{t.diegosalgado} once more, and arguing identically as in the proof above, we deduce that} the sets $K_c$ consist of orbits whose $\alpha$ and $\omega$-limit sets are periodic orbits. This simplifies the structure of the collinearity locus.
\end{rem}

\subsection{Strategies of proof of the main results}

\paragraph{Strategy of the proof of Theorem \ref{t.Normal} --} If there exists a counterexample to Theorem B, we have seen that there exists also a prepared counterexample. The first step is to show that, {when we consider the normal component $N$ and quotient function $\mu$ associated to the transverse invariant plane field, given by assumption,} then the level sets of the quotient function are invariant under the flow of $Y$ {(see item (2) of Lemma~\ref{fundprop})}. This implies that $Y$ is tangent to the level sets of $\mu$. After that, the main idea is to perform a suitable modification of the isolating neighbourhood so that we can compute the linking numbers of $N$ along curves which are included in the union of finitely many paths which are either periodic orbits of $Y$ or contained in the {stable/unstable sets} of periodic orbits of $Y$ included in $K$. {This is achieved in the Key Lemma~\ref{keylemouille}}. 

{Then, the effective implementation of this idea uses in a crucial way the fact that} $Y$ is tangent to level sets of $\mu$. Indeed, to prove that all the linking numbers vanish, the important step is to show that the $C^2$ function which measures the derivative of $\mu$ along the normal direction $N$ is $Y_t$-invariant {(which is proven in Lemma~\ref{PhiY})}. This will imply that along one boundary component (with the above dynamical property) the vector field $N$ is either tangent to the level set or it is never tangent. Using a basis containing a vector field tangent to the level set, one then deduces that the corresponding linking number vanishes {(see Lemma~\ref{endproofsusp})}.

\paragraph{Strategy of the proof of Theorem \ref{t.valeurpropre} --} {The proof of Theorem~\ref{t.valeurpropre} is much more involved, it will occupy both sections \ref{mainsectouille} and \ref{s:final_section}}. As in the previous theorem, we assume by contradiction the existence of a prepared counterexample to Theorem \ref{t.valeurpropre}. Since our isolated compact set $K$ of zeros is formed by periodic orbits of $Y$ and heteroclinic connections between them, we can distinguish two types of periodic orbits in the level 0. Those that are the alpha/omega limit set of an element in $K$, which we call \emph{linked periodic orbits}, and those that are not, which we call \emph{non-linked} periodic orbits. 

We show that every linked periodic orbit has a stable (or unstable) manifold which is tangent to the level set {(which is done in Proposition~\ref{p.tangent})} and that we can decompose $K$ in a disjoint union { into compact sets} $K_{ms}\sqcup K_{nl}$, where $K_{ms}$ is formed by finitely many linked periodic orbits and heteroclinic connections, while $K_{nl}$ is the union of non-linked periodic orbits. Using the additivity of Poincaré-Hopf index, we deduce that we are reduced to show that $\ind(X,K_{ms})=\ind(X,K_{nl})=0$. {A key role in the proof of this decomposition is played by the fact (Lemma~\ref{l.lesorbitesliessonfini}) that there is only a finite number of linked periodic orbits.} 

{To show that $\ind(X,K_{nl})=0$ we further decompose the set $K_{nl}$ according to the position of the stable/unstable manifolds of the periodic orbits, proving (in Lemma~\ref{l.structurenonlinked}) that $K_{nl}=K_{nl}^T\sqcup K_{nl}^{\perp}$, where $K_{nl}^T$ is a finite union of periodic orbits having a stable/unstable manifold tangent to the level set while $K_{nl}^{\perp}$ is a compact set of periodic orbits having a stable/unstable manifold transverse to the level. In both cases, the strategy to prove the vanishing of the index is to use the Center Manifold Theorem to organize the periodic orbits of the collinearity locus and then apply a technical version of Bonatti-Santiago's Theorem \ref{t:BS}, stated in Theorem~\ref{t.bsdenovo}. However, in the transverse case we need a specific compactness argument (Lemma~\ref{structurethm}) to nicely cover $K_{nl}^{\perp}$ and a difficulty to apply Theorem~\ref{t.bsdenovo} is that the center manifold of a periodic orbit could be tangent to the level. For overcoming this a new application of Sard's Theorem is required (see Lemma~\ref{transversecase}). }   

Our main idea to prove that $\ind(X,K_{ms})=0$ is the content of Theorem \ref{Th.MS}. We show that whenever an unstable periodic orbit is linked with a stable one by an orbit in $K$, we can flow the local unstable manifold along a heteroclinic connection and glue it with the stable manifold of the second periodic orbit (see The Gluing Lemma~\ref{l.collage}). This allows us to define a new open neighbourhood of $K_{ms}$ endowed with a trivial foliation by surfaces, such that both $X$ and $Y$ are tangent to the leaves. By an appropriate choice of basis (see Lemma~\ref{l.base} for details), we deduce that the Gauss map is not surjective, which proves that the index is zero.

\section{The case of a flow with a transverse invariant plane field}
\label{suspoupouille}

This section is devoted to proving Theorem \ref{t.Normal}. {Recalling the notations of \S~\ref{sub:prepared_ce}, the assumptions of Theorem~\ref{t.Normal} give a $C^3$-triple $(U,X,Y)$ with the additional property that there exists a $C^3$ plane field $\Pi$ in $U$ everywhere transverse to $Y$, {and invariant by the flow of $Y$}. We assume by contradiction that Theorem~\ref{t.Normal} is not true. Thus, the triple $(U,X,Y)$ also satisfies $\ind(X,U)\neq 0$.  

Applying Theorem \ref{t:prepared}, we find a new vector field $\tilde{X}$ and a smaller isolating neighbourhood $\tilde{U}\subset U$ so that the $C^3$-triple $(\tilde{U},\tilde{X},Y)$ is a $C^3$-prepared counterexample to the conjecture. 

Therefore}, it is enough to prove that if $(U,X,Y)$ is a prepared triple  such that $Y$ has a $C^3$ invariant plane field $\Pi$ then $\Ind(X,U)=0$. {We shall then consider the normal component $N$ and the quotient function $\mu$ associated with this $Y_t$-invariant plane field $\Pi$. Recall that they are related with the vector fields $X$ and $Y$ by the equation 
\[X=N+\mu Y,\]
and $N(x)\in\Pi(x)$ for every $x\in U$.

} We will also assume the orientability properties of Proposition \ref{p:orientable}. The idea is  to use our index formula and take advantage of the $Y$-invariance of $\Pi$.

Let $(U,X,Y)$ be such a prepared triple and $\Pi$ be a transverse $Y_t$-invariant plane field. We define
\begin{equation}\label{notationtheta}
	\theta(x)=\sup\{t>0;Y_s(x)\in U\,\,\text{for}\,\,|s|\leq t\}.
\end{equation}
so that $D_xY_t(\Pi(x))=\Pi(Y_t(x))$ when $|t|\leq\theta(x)$.

\subsection{Invariant foliation}
\label{Invariouille}

\paragraph{Invariance properties --}  The normal component and the quotient function have the following fundamental property.

\begin{lem}
	\label{fundprop}
	The following properties hold true for every $x\in U$ and $|t|\leq\theta(x)$ (see \eqref{notationtheta})
	\begin{enumerate}
		\item $D_xY_t\,N(x)=N(Y_t(x))$;
		\item $\mu(Y_t(x))=\mu(x)$.
	\end{enumerate}
\end{lem}

\begin{proof}
	Let $x\in U$ and $t\in(0,\theta(x))$. {Let $\Sigma_0$ be a transverse section to $Y$ tangent to $\Pi$ at $x$. Let $\Sigma_t=Y_t(\Sigma_0)$. This is a transverse section to $Y$ tangent to $\Pi$ at $Y_t(x)$, and the time-$t$ map of $Y$ provides a holonomy map $P:\Sigma_0\to\Sigma_t$. In particular, by invariance, $\Sigma_t$ is tangent to $\Pi$ at $Y_t(x)$ and the hitting time $\tau:\Sigma_0\to(0,\infty)$ (defined by \eqref{eq.hitting_time}) is constant equal to $t$. We deduce that $D_x\tau=0$ and $D_xP=D_xY_t$ on $\Pi(x)$.}
	
	Applying Item 1. of Lemma \ref{calculusholonomies}, one sees that $D_x Y_t\,N(x)=D_xP\, N(x)=N(P(x))=N(Y_t(x))$.
	
	Applying Item 2. of Lemma \ref{calculusholonomies}, one sees that $\mu(x)=\mu(P(x))=\mu(Y_t(x))$. This is enough to conclude the proof.
\end{proof}

Let us now consider the variation of $\mu$ in the normal direction, i.e. define the $C^2$-map $\phi:U\to\R$ by the formula
$$\phi(x)=D_x\mu\,N(x).$$

\begin{lem}
	\label{PhiY}
	For every $x\in U$ and $t\leq\theta(x)$ we have $\phi(Y_t(x))=\phi(x)$
\end{lem}

\begin{proof}
	By taking derivatives on both sides in Item 2. of Lemma \ref{fundprop} with respect to $t$, we find $D_x\mu\,Y(x)=0$ for every $x\in U$. We deduce that
	$$D_x\mu\,X(x)=D_x\mu\,N(x)+\mu(x)\,D_x\mu\,\,Y(x)=\phi(x).$$
	
	By taking derivatives on both sides in Item 2. of Lemma \ref{fundprop} with respect to $x$, we find $D_x\mu=D_{Y_t(x)}\mu\circ D_xY_t$ when $|t|\leq\theta(x)$. Since $X$ and $Y$ commute we have that $D_xY_t(x)X(x)=X(Y_t(x))$, for every $|t|\leq\theta(x)$ and $x\in U$.
	Thus, combining these two facts we obtain
	$$\phi(Y_t(x))=D_{Y_t(x)}\mu\,X(Y_t(x))=\left(D_{Y_t(x)}\mu\circ D_xY_t(x)\right)\,X(x)=D_x\mu\, X(x)=\phi(x),$$
	which ends the proof.
\end{proof}

\subsection{Computation of the index}
\label{sub:compouopouille}

\paragraph{Appropriate isolating neighbourhood --}  According to Lemma \ref{fundprop}, $Y$ is tangent to the foliation of $\overline{U}$ by level sets $\mu^{-1}(c)$. We know that if $(U,X,Y)$ is a $C^3$-prepared counterexample then the second item of Lemma \ref{l:Denjouill_Scwhartzouille} holds. So we will now suppose that $\partial\mu^{-1}(0)\neq \vide$ and that  for every $x\in K$, $\alpha_Y(x)$ and $\omega_Y(x)$ are periodic orbits of $Y_t$. We will compute the index at $K$ by choosing a more appropriate isolating neighbourhood.

{ Recall from \S \ref{s.general_dynamics}  that by definition the stable and unstable sets of a periodic orbit $\gamma$ of a vector field $Y$ on a manifold $M$
$$W^s(\gamma)=\left\{x\in M;\,\omega_Y(x)=\gamma\right\}\,\,\,\,\,\,\,\text{and}\,\,\,\,\,\,\,W^u(\gamma)=\left\{x\in M;\,\alpha_Y(x)=\gamma\right\}.$$
}

The following key lemma will provide us with the desired isolating neighbourhood. The proof will be postponed until the end of the section.

\begin{lem}[Key lemma]
	\label{keylemouille}
	Let $Y$ be a $C^1$ vector field of a compact surface $S$, and $K\dans S$ be an invariant compact set with the following property: for every $x\in K$, $\alpha_Y(x)$ and  $\omega_Y(x)$ are periodic orbits of $Y$.
	
	Then there exists a neighbourhood $W$ of $K$ whose boundary is a finite union of simple closed curves, which are contained in a finite union of curves $(C_i)_{i=1}^k$ such that for every $i\in\{1,\ldots,k\}$ one of the two following alternatives holds
	\begin{enumerate}
		\item either $C_i$ is a periodic curve for $Y$;
		\item or we have
		{$$C_i\dans\bigcup_{y\in\per(Y)\cap K}W^s(\OO_Y(y))\cup  W^u(\OO_Y(y)).$$
		}
	\end{enumerate}
\end{lem}

Apply the key lemma with $K$ being our isolated compact subset of $\Zero(X)$, which is $Y_t$-invariant, and $S$ being the level set $\mu^{-1}(0)$ containing $K$, endowed with the vector field $Y_{|\mu^{-1}(0)}$. 
{Since the second case of Lemma~\ref{l:Denjouill_Scwhartzouille} holds, as we mentioned above,} it is clear that the assumptions of the key lemma are satisfied. So,
consider 
$W\dans U\cap S$ such as in Lemma~\ref{keylemouille}.

Define a tubular neighbourhood $V\dans U$ of $W$, that we identify with $W\times(-\delta,\delta)$ for a sufficiently small $\delta>0$, and which has the property that $\mu=c$ on the component $W\times\{c\}$. Recall that if $\gamma$ is a boundary component of $W$ (with the boundary orientation) $l(\gamma)$ denotes the linking number of $N$ along $\gamma$ (see Definition \ref{d:linknumb}). Our index formula (Theorem \ref{t:indexf}) gives
$$\Ind(X,U)=\Ind(X,K)=\Ind(X,V)=\sum_{\gamma\dans\partial W}l(\gamma).$$
Our proof will consist in showing that for every boundary component $\gamma$, the Gauss map $\nu|_\gamma:\gamma\to\Ss^1$ is not surjective, which will imply that $l(\gamma)=0$ and the result will follow.

\paragraph{End of the proof of Theorem \ref{t.Normal} --} We associate continuously to $x\in V$ an orthonormal  basis $\beta(x)=(e_1(x),e_2(x),e_3(x))$ of $T_xM$ where $e_3=Y$ and $e_2$ is tangent to the level set of $\mu$ containing $x$. Note that $e_3$ is also tangent to level sets of $\mu$ so $e_1$ is orthogonal to the level sets.

We use coordinates $(x_1,x_2,x_3)$ on the sphere $\Ss^2$ and identify the oriented circle $\Ss^1$ with the equator $\{x_3=0\}$ oriented as the boundary of the north hemisphere $\{x_3>0\}$. Let $\gamma$ be a boundary component of $W$ and $N(x)=\alpha_1(x)e_1(x)+\alpha_2(x)e_2(x)$. The Gauss map of $N$ is defined at $x\in\gamma$ as
$$\nu(x)=\frac{(\alpha_1(x),\alpha_2(x))}{\sqrt{\alpha_1(x)^2+\alpha_2(x)^2}}.$$

Theorem \ref{t.Normal} is a consequence of the following result. 

\begin{lem}
	\label{endproofsusp}
	Let $W$, $V$, $\beta$ and $\nu$ be the objects previously constructed. Then for every boundary component $\gamma$, the Gauss map $\nu|_\gamma:\gamma\to\Ss^1$ is not surjective. In particular $l(\gamma)=0$ for every such $\gamma$ and $\Ind(X,K)=0$.

\end{lem}

\begin{proof}
By Lemma \ref{keylemouille} there are finitely many curves $C_i$, which may be periodic orbits of $Y$ or contained inside the {stable/unstable sets} of periodic orbits of $K$, such that each boundary component of $W$ is included inside $\bigcup C_i$.

	Distinguish the following two points of the equator $A=(0,1,0)$ and $B=(0,-1,0)$. Pick a point $x\in C_i$. 
	
	\emph{Case 1. There exists $y\in\per(Y)\cap K$ such that {$C_i\dans W^{s}(\OO_Y(y))\cup W^{u}(\OO_Y(y))$.}}
	
	Recall the function $\phi(x)=D_x\mu N(x)$. By Lemma~\ref{PhiY} we have that $\phi(Y_t(x))=\phi(x)$, for every $t\in\R$. By continuity of $\phi$ it follows that $\phi(x)=\phi(p)$, for every $p\in\omega_Y(x)\cup\alpha_Y(x)$. 
	Since for every $y\in(\omega_Y(x)\cup\alpha_Y(x))\cap K$ we have $N(y)=0$ and we conclude that $\phi(x)=\phi(y)=0$. 
	
	Therefore, $N(x)\in\Pi(x)=\langle e_1(x),e_2(x)\rangle$ is tangent to $\mu^{-1}(0)$. Hence $N(x)$ belongs to the intersection $\Pi(x)\cap T_x\mu^{-1}(0)$ which is generated by $e_2(x)$. This implies that
\begin{equation}
\label{e.imagem}
\nu(x)\in\{A,B\}.
\end{equation}

	\emph{Case 2. $C_i$ is a periodic orbit of $Y_t$.} 
	
	In this case, we use again the invariance of $\phi$ which implies that there exists $\fhi_i\in\R$ such that $\phi(y)=\fhi_i$ for every $y\in C_i$.
	
	If $\fhi_i=0$, we obtain once more $D_x\mu\,N(x)=0$ and we conclude again that (\ref{e.imagem}) holds. 
	
	Suppose now that $\fhi_i\neq 0$ so that $D_y\mu\,N(y)=\fhi_i\neq 0$ for every $y\in C_i$. In other words, $N(y)$ is \emph{never} tangent to the level set $\mu^{-1}(0)$. In particular, it is not collinear with $e_2$ over $C_i$.  
	This implies that $\nu(C_i)$ is a compact subset of the equator $\{x_3=0\}$ which contains neither $A$ nor $B$.
	
We conclude that $\nu(\bigcup C_i)$ is included in the union of $\{A,B\}$ with a compact set disjoint from $\{A,B\}$. In particular $\nu$ is not surjective when restricted to any boundary component of $W$.
\end{proof}

\subsection{Proof of the key lemma}

In order to construct the open set $W$ we first construct a neighbourhood of every periodic orbit in $K$ by annuli whose boundary components are either periodic orbits or contained in its {stable/unstable set}. Then we cover every point of heteroclinic connection by a disc of the {stable/unstable set} crossing each other transversally. A compactness argument will show that the boundary of the union of these neighbourhoods satisfies the desired property.

Before starting the proof let us state a consequence of the proof of Poincar\'e-Bendixson's theorem.

\begin{lem}
\label{l:bassin_open}
Let $Y$ be a vector field on the surface $S$. For every periodic orbit $\gamma$ of $Y$ the sets {$W^s(\gamma)\moins\gamma$ and $W^u(\gamma)\moins\gamma$} are open sets (which may be empty).
\end{lem}

Pick a point $x\in K$ assume that $x$ belongs to a periodic orbit $\gamma\dans K$. Take a small transverse arc $I$ containing $x$, and let $P:J\to I$ denote the first return map of the flow $Y_t$ to the section where $J$ is a subarc of $I$ containing $x$. Let $J^{+}$ and $J^{-}$ denote the connected components of $J\setminus\{x\}$. If $J$ is small  enough we have three possibilities

\begin{enumerate}
	\item\label{i.um} {$J^+\dans W^s(\gamma)$,}
	\item\label{i.dois} {$J^+\dans W^u(\gamma)$,}
	\item\label{i.tres} $J^+$ contains a fixed point of $P$.
\end{enumerate}

If case~\ref{i.um} or case~\ref{i.dois} holds, there exists a simple closed curve $C^+$ crossing $J^+$, disjoint from $\gamma$, and included inside {$W^s(\gamma)$ and $W^u(\gamma)$} respectively. If case~\ref{i.tres} holds there exists a periodic orbit $C^+$ of $Y$ crossing $J^+$, which must also be disjoint from $\gamma$. 

With an analogous reasoning we can build a simple closed curve $C^{-}$ disjoint from $\gamma$, which is either a periodic orbit of $Y$ or it is contained in {$W^{\sigma}(\gamma)$}, $\sigma=s,u$ and crosses $J^-$. Hence every periodic orbit $\gamma$ is contained in an annulus $A_\gamma$ whose boundary components are either periodic orbits or contained inside {$W^s(\gamma)\cup W^u(\gamma)$}.

Every other point $x\in K$ is a heteroclinic connection between periodic orbits of $K$ and thus belongs to {$W^u(\gamma)\moins\gamma$} for some periodic orbit $\gamma\dans K$. By Lemma \ref{l:bassin_open} there exists a disc {$D_x\dans W^u(\gamma)$} centred at $x$.

Using the compactness of $K$ we deduce that it is covered by an open set $W$ which is a finite union of annuli $A_{\gamma_i}$ and discs $D_{x_j}$. Since {stable/unstable sets} are open, we can suppose that intersections between boundaries of these sets are empty or transverse. Hence $\partial W$ is a finite union of simple closed curves included in the union of finitely many periodic orbits of $Y$ and {stable/unstable sets} of periodic orbits of $Y$ included in $K$, as desired. \qed

\section{The Morse-Smale case}
\label{mainsectouille}
\subsection{The Morse-Smale hypothesis}
\label{Morse-Smouille}

As we saw in Theorem \ref{t:prepared} and Remark \ref{rem_prep_ex} in order to prove Theorem \ref{t.valeurpropre} it is enough to prove that there is no $C^3$-prepared counterexample to this statement ({the plane field, normal component and quotient functions can be supposed here to come from some Riemannian metric}). Using Proposition \ref{p:orientable} it is enough to prove that there is no $C^3$-prepared counterexample with $Y$ being transversally oriented and with level sets $\mu^{-1}(c)$ being oriented.

Given a $C^3$-prepared triple with the right orientability conditions we introduce now a stronger hypothesis than that of Theorem \ref{t.valeurpropre} that we call the \emph{Morse-Smale hypothesis}, and concerns \emph{the projected vector field} $Y'$ (see \S \ref{sub:simplif_col_loc}). We will then prove that $\Ind(X,U)=0$ under this assumption.

So let $Y'$ be the projection of $Y$ on level sets $\mu^{-1}(c)$. We say that $Y$ satisfies (MS), the \emph{Morse-Smale hypothesis close to $K$}\ if

\begin{enumerate}
\item[(MS)] 
\emph{The vector field $Y'|_{\mu^{-1}(0)}$ is Morse-Smale.}
\end{enumerate}

We are now ready to state the main result of this section.
\begin{thmA}
\label{Th.MS}
Let $(X,Y,U)$ be a $C^3$-prepared triple such that $Y$ satisfies the Morse-Smale hypothesis close to $K$. Then
$$\Ind(X,K)=0.$$
\end{thmA}

\begin{rem}
\label{rem:relYYprima}
At first sight, Theorem \ref{Th.MS} and Hypothesis (MS) concerns only the projection $Y'$. Note however that $Y=Y'$ when restricted to $K$. The heart of the argument is actually to prove the result when $Y$ satisfies (MS$'$), a stronger condition, under which it will be proven that the Poincar\'e maps for $Y$ and $Y'$ at a periodic point of $K$ have the same derivative (Proposition \ref{l.specpoincouille}).
\end{rem}

\subsection{The topology of the collinearity locus}\label{s.topo_col_locus}
Until the end of the section we assume that $(U,X,Y)$ is a $C^3$-prepared triple  with orientability hypotheses of Proposition \ref{p:orientable} such that $Y$ satisfies (MS). Our goal is to prove that $\Ind(X,K)=0$.

\paragraph{Combinatorics of the zero set --}  By hypothesis, for every $x\in K\dans\mu^{-1}(0)$, $\alpha_{Y'}(x)$ (resp. $\omega_{Y'}(x)$) is an unstable (resp. stable) periodic orbit of $Y'|_{\mu^{-1}(0)}$. 
Since moreover $[X,Y]=0$ and $Y=Y'$ in restriction to $K$, we see that they are periodic orbits of $Y$ contained in $\Zero(X)$. 

Thus $K$ is formed by stable and unstable periodic orbits of $Y'|_{\mu^{-1}(0)}$ and by orbits of $Y'|_{\mu^{-1}(0)}$ that have an unstable periodic orbit as alpha limit set and a stable periodic orbit as omega limit set. 
As explained above, these orbits are orbits of $Y$.

We shall denote the stable periodic orbits of $Y'|_{\mu^{-1}(0)}$ contained in $K$ by $\gamma^s_1,\ldots,\gamma^s_{k_s}$ and the unstable ones, by $\gamma^u_1,\ldots,\gamma^u_{k_u}$.

\begin{dfn}
\label{deflink}
We say that an unstable periodic orbit $\gamma_i^u$ is \emph{linked} to a stable orbit $\gamma_j^s$ if there exists $x\in K$ such that $\alpha_Y(x)=\gamma_i^u$ and $\omega_Y(x)=\gamma_j^s$. In that case we say that the orbit of $x$ \emph{links} $\gamma_i^u$ and $\gamma_j^s$.
\end{dfn}

The \emph{combinatorics of the set $K$} is given by the oriented graph whose vertices are periodic orbits $\gamma^u_i,\gamma^s_j$ such that there is an oriented arrow from $\gamma_i^u$ to $\gamma_j^s$ if the two orbits are linked.

\paragraph{Continuation of periodic orbits --} We want now to understand the topology of the collinearity locus. The next lemma uses the partial hyperbolicity gained with the Morse-Smale hypothesis in order to obtain the continuation of periodic orbits inside $K$.

\begin{lem}
\label{l.continuation}

There exists $\eps>0$, such that for every $j\in\{1,\ldots,k_s\}$ there exists $\cA_j^s\dans\mu^{-1}[-\eps,\eps]$ and a $C^3$-diffeomorphism $\Phi_j^s:[-\eps,\eps]\times\Ss^1\to\cA_j^s$ satisfying the following properties

\begin{enumerate}
\item for every $c\in[-\eps,\eps]$, we have $\gamma^{s,c}_j\dans\mu^{-1}(c)$ where we set $\gamma^{s,c}_j=\Phi_j^s(c,\Ss^1)$;
\item we have $\gamma^{s,0}_j=\gamma^s_j$;
\item for every $c\in[-\eps,\eps]$, $\gamma^{s,c}_j$ is a stable periodic orbit of $Y'_{|\mu^{-1}(c)}$.
\end{enumerate}
\end{lem}

\begin{proof}
We could invoke Hirsch-Pugh-Shub's theory and the center manifold theorem. In our context however the proof is quite elementary.

Let $V_j^s\dans\mu^{-1}(0)$ be an annular neighbourhood of $\gamma_j^s$ (recall that $\mu^{-1}(0)$ is oriented). Let $\cN$ be a tubular neighbourhood of $V_j^s$ whose intersection with level sets $\mu^{-1}(c)$, $|c|\leq\eta$, are annuli which trivially foliate it. Hence there is $C^3$-system of coordinates $\Psi:[-\eta,\eta]\times\A\to\cN$, where $\A=(-1,1)\times\Ss^1$ is an annulus, such that $\Psi(0,\A)=V_j^s$ and $V^{s,c}_j:=\Psi(c,\A)\dans\mu^{-1}(c)$.

Recall that $Y'$ is tangent to the level sets $\mu^{-1}(c)$. Pulling back $Y'$ by $\Psi$ we get a smooth (of class $C^3$) $1$-parameter family $\xi_c$ of vector fields on $\A$ indexed by $c\in[-\eta,\eta]$. Take a simple arc $I$ in $\A$ cutting the stable periodic orbit of $\xi_0$. The corresponding Poincar\'e map $P_0$ lies inside a smooth $1$-parameter family of maps $P_c$ which must be uniformly contracting (up to shrinking $I$ and taking $|c|\leq\eps$). By Picard's fixed point theorem every map $P_c$ has a unique fixed point $x_c$ in $I$. A fairly direct application of the implicit function theorem implies that the variation of $x_c$ with the parameter $c$ is $C^3$.

Saturating by the flow gives a smooth family of embedding $\phi_c:\Ss^1\to\A$ with $|c|\leq\eps$, such that for every $|c|\leq\eps$, $\phi_c(\Ss^1)$ is a stable periodic orbit of $\xi_c$.

The embedding $\Phi^s_j(c,z)=\Psi(c,\phi_c(z))$ for $(c,z)\in[-\eps,\eps]\times\Ss^1$ is the desired one.
\end{proof}

Now, in order to get that the continuations obtained above are contained in $K_c$, we need to use the continuity of the map $c\mapsto K_c$ at $0$.

\begin{lem}
\label{l.orbitsofY}
The number $\eps>0$ obtained in Lemma \ref{l.continuation} may be chosen small enough so that for every $j\in\{1,...,k_s\}$, and $c\in[-\eps,\eps]$
$$\gamma^{s,c}_j\dans\Zero(X-cY).$$
\end{lem}

\begin{figure}[hbtp]
\centering
\includegraphics[scale=0.6]{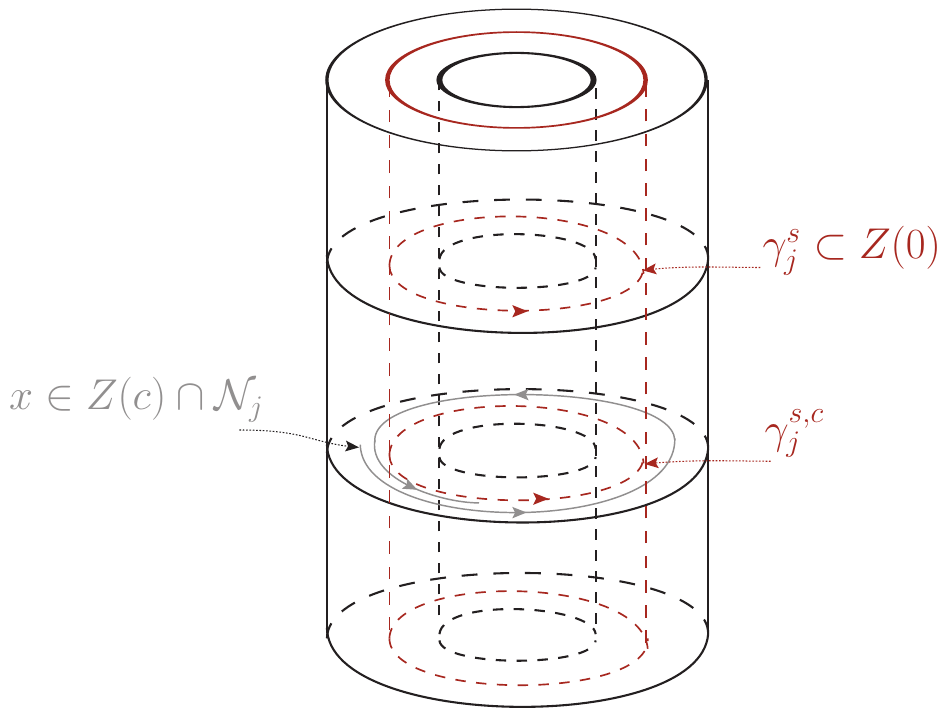}
\caption{The orbits $\gamma^{s,c}_j$ are zeros of $X-cY$}
\label{des_continuation}
\end{figure}

\begin{proof}
As in the proof of the previous lemma, consider an annular neighbourhood of $\gamma^s_j$ denoted by $V^s_j\dans\mu^{-1}(0)$. Consider also a tubular neighbourhood of this annulus having the form $\cN_j=\bigcup_{|c|<\eps} V^{s,c}_j$ where $V^{s,c}_j\dans\mu^{-1}(c)$ is an annular neighbourhood of $\gamma^{s,c}_j$. We can assume that these annuli $V^{s,c}_j$ are \emph{attracting neighbourhoods} in the sense that for every $x\in V^{s,c}_j$, we have $\omega_{Y'}(x)=\gamma^{s,c}_j$.

Recall that $\gamma^s_j\dans K\dans\Zero(X)$. We will use the \emph{lower semi-continuity} at $0$ of $c\mapsto K_c=\Zero(X-cY)\cap\overline{U}$. Indeed it implies that, if $\eps$ is chosen small enough, for every $|c|\leq\eps$, $K_c\cap\cN_j\neq\vide$.

Now let $c\in[-\eps,\eps]$ and $x\in K_c\cap\cN_j$: in particular $x\in V^{s,c}_j$ which is an attracting region. Since $X$ and $Y$ commute, the whole $Y_t$-forward orbit of $x$ belongs to $\Zero(X-cY)$, and accumulates to a subset of $\Zero(X-cY)$. Since $Y'=Y$ in restriction to $\Zero(X-cY)$, this implies in particular that $\gamma^{s,c}_j\dans\Zero(X-cY)$, concluding.
\end{proof}

Similarly one can perform the continuation of the unstable periodic orbits $\gamma^u_i$, thus obtaining a $C^3$ family of unstable periodic orbits of $Y'_{|\mu^{-1}(c)}$, $|c|\leq\eps$. These unstable periodic orbits are denoted by $\gamma^{u,c}_i$ and are subsets of $\Zero(X-cY)$. Finally for a fixed $i$ we set $\cA_i^u=\bigcup_c\gamma^{u,c}_i$, which is a $C^3$-embedded annulus.

\begin{rem}
The orbits $\gamma^{u,c}_i$ and $\gamma^{s,c}_j$ are periodic orbits of $Y$.
\end{rem}

Below, we consider the open set $\mu^{-1}(-\eps,\eps)$ and still denote it by $U$.

\paragraph{Combinatorics of neighbouring levels --}
We shall apply the fact that $0$ is a continuity point of the map $Z:c\mapsto K_c$ to show that we can define a combinatorics for the neighbouring levels and that such combinatorics is the same as that of $K$.

\begin{lem}
\label{l.memecombinato} 
There exists $\eps>0$ small enough such that if $|c|<\eps$ then
$K_c\cap\per(Y)$ is formed by the stable and unstable periodic  orbits $\gamma_i^{u,c}$ and $\gamma_j^{s,c}$ of $Y'|_{\mu^{-1}(c)}$. Moreover, there exists $x\in K_c$ such that $\alpha_Y(x)=\gamma^{u,c}_i$ and 
$\omega_Y(x)=\gamma^{s,c}_j$ if, and only if, $\gamma^u_i$ and $\gamma^s_j$ are linked.
\end{lem}

\begin{figure}[!h]
\centering
\includegraphics[scale=0.55]{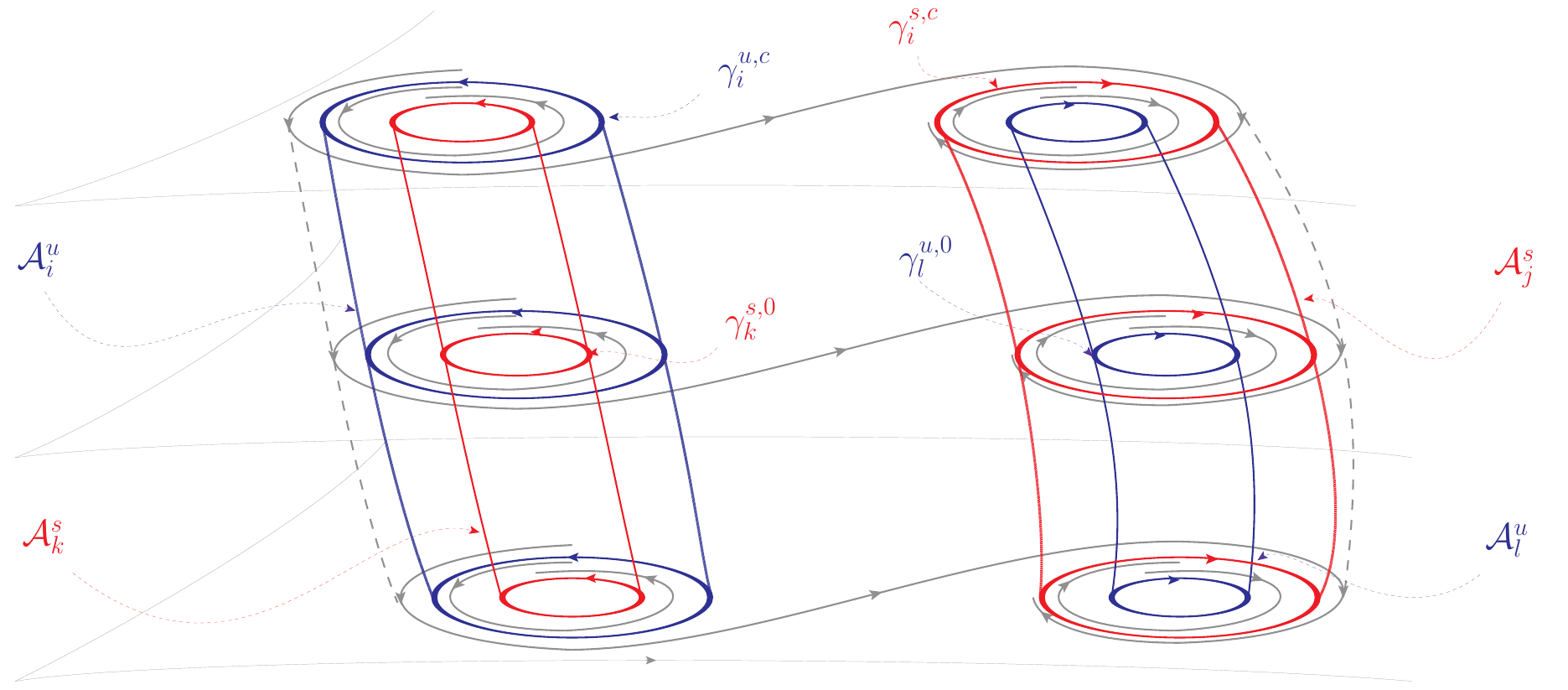}
\caption{The combinatorics of periodic orbits}
\label{combinouille}
\end{figure}

\begin{proof}
Notice that $K_c$ is invariant under $Y'$, since it is invariant under $Y$ and the two vector fields coincide over $K_c$.
{Moreover, applying Theorem~\ref{t.diegosalgado} to $Y'|_{\mu^{-1}(c)}$ it comes that for every $x\in K_c$, $\omega_Y(x)=\omega_{Y'}(x)$ and $\alpha_Y(x)=\alpha_{Y'}(x)$ are periodic orbits, still included in $K_c$.} 

Let us denote by $K^1,...,K^m$ the connected components of $K$. Consider $U^1,...,U^m$, disjoint open neighbourhoods of $K^1,...,K^m$, respectively. 
By upper-semicontinuity of $Z$ there exists $\eps>0$ small enough such that if $|c|<\eps$ then the union of all the $U^l$ contain $K_c$. Also, we can choose the neighbourhoods $U^l$ small enough so that
they are formed by tubular neighbourhoods of attracting annuli $\cN^s_j=\bigcup_{|c|<\eps} V^{s,c}_j$, as in the proof of Lemma~\ref{l.orbitsofY}, and tubular neighbourhoods $\cL_{j,i}$ of orbit segments linking $\gamma^u_j$ with $\gamma^s_i$. 
By the long tubular flow theorem \cite{PdM} one can choose these $\cL_{j,i}$ small enough so that for every $x\in\cL_{j,i}$, its forward $Y_t'$-orbit hits $\cN_j^s$ and its backward orbit hits $\cN^u_i$. With this description of open sets $U^l$ is it clear that their union doesn't contain any periodic orbit of $Y'$ other than the $\gamma_i^{u,c}$ and $\gamma_j^{s,c}$ given by Lemma~\ref{l.continuation}.

\emph{Non-linked case. }We first consider the case of an isolated stable periodic orbit i.e. a periodic orbit $\gamma^s_j$ which is not linked to any other $\gamma^u_i$. In that case we must have $K\cap V^s_j=\gamma^s_j$ since if there existed $x\in K\cap V^s_j\moins\gamma^s_j$, it couldn't be a periodic point (since $\per(Y')\cap V^s_j=\gamma^s_j$) and its backward orbit would accumulate to a periodic orbit contained inside $K$, contradicting the hypothesis.

By continuity, we may assume that $\eps$ is small enough so that $K_c\cap\partial\cN^s_j=K_c\cap\partial V^{s,c}_j=\vide$ when $|c|<\eps$. Now, for $|c|<\eps$, if $\gamma^{s,c}_j$ is linked to some periodic orbit $\gamma^{u,c}_i$, which must be outside of $V^{s,c}$, there exists $x\in K_c\cap V^{s,c}_j$ whose backward orbit meets $\partial V^{s,c}_j$. This is absurd and $\gamma^{s,c}_j$ is not linked to any orbit $\gamma^{u,c}_i$. The case of an isolated unstable periodic orbit follows from a symmetric argument.

\emph{Linked case. } Consider now two periodic orbits of $Y'$ $\gamma^u_i$ and $\gamma^s_j$ inside $K$ which are linked. By lower semicontinuity of $Z$ there exists a point $x\in K_c\cap\cL_{j,i}$. 
The forward $Y_t'$-orbit of $x$ hits $\cN^s_j$, in particular $\omega_{Y'}(x)=\gamma^{s,c}_j$.  With a similar argument we prove that $\alpha_{Y'}(x)=\gamma^{u,c}_i$. In particular this proves that $\gamma^{u,c}_i$ and $\gamma^{s,c}_j$ are linked.

\emph{Reverse implication. } Let |c| be small enough. By our choice of the $U^l$, if  $x\in K_c$ is such that $\alpha_{Y'}(x)=\gamma^{u,c}_i$ and $\omega_{Y'}(x)=\gamma^{s,c}_j$
then iterating forward or backward we find a point $y=Y_t'(x)\in\cL_{j,i}\setminus(\cN^s_j\cup\cN^u_i)$. This shows that $\cL_{j,i}\setminus(\cN^s_j\cup\cN^u_i)\neq\emptyset$, and so $\gamma^u_i$ and $\gamma^s_j$ are linked.
\end{proof}

\paragraph{The (MS$'$)-hypothesis --} { We shall prove below that the non-linked periodic orbits inside $K$ contribute with zero index, therefore we can discard them in our arguments. The proof is an automatic consequence of a technical version of Theorem~\ref{t:BS}, which is also proven in \cite{BS}. Since this technical statement needs some definitions that have no place in this section, we refer the reader to Theorem~\ref{t.bsdenovo} below.}


\begin{lem}
\label{l.isolated}
If there exists $j\in\{1,...,k_s\}$ such that $\gamma^s_j$ is linked with no $\gamma^u_i$, $i=1,...,k_u$. Then $\gamma^s_j$ is isolated in $K=\Zero(X)$ and 
$$\ind(X,\gamma^s_j)=0.$$
\end{lem}
\begin{proof}
Let $\gamma^s_j$ be such an orbit. It follows from the proof of the previous lemma (the non-linked case) that there exists a neighbourhood $\cN^s_j$ of $\gamma^s_j$ such that $\Col_U(X,Y)\cap\cN^s_j$ is a smooth (see Lemma \ref{l.continuation}) annulus foliated by periodic orbits for $Y'$ (thus of $Y$), $\gamma^{s,c}_j$. Using Bonatti-Santiago's Theorem~\ref{t.bsdenovo} (proven in \cite{BS}) the lemma follows.
\end{proof}

\begin{rem}
A symmetric argument deals with the case of unstable periodic orbits in $K$ not linked with any stable orbit in $K$.
\end{rem}

From now on, we shall assume that there are no isolated components for the combinatorics of $K$. In particular, we can assume the following hypothesis 
\begin{enumerate}
	\item[(MS$'$)]  
	\emph{$Y$ satisfies (MS) and for every $\gamma\dans\per(Y')\cap K$ there exists $x\in K\moins\per(Y')$ such that $\omega_{Y'}(x)=\gamma$ or $\alpha_{Y'}(x)=\gamma$.}
\end{enumerate}

As a consequence, in what remains of the section we will assume that $(U,X,Y)$ is a prepared triple with $U=\mu^{-1}(\eps,\eps)$ such that for every $|c|<\eps$, $K_c\cap U$ consists of stable, unstable periodic orbits of $Y$ and of heteroclinic connections. Furthermore, the combinatorics of all $K_c$, i.e. the graphs of heteroclinic connections, is independent of $c$.

\subsection{Derivatives of first return maps}\label{s.der_first_return}
We now show that all orbits $\gamma^{u,c}_j$ and $\gamma^{s,c}_j$ possess unstable and stable manifolds respectively for $Y_t$. This is done by identifying the derivatives of first return maps of $Y,Y'$ at those periodic orbits.
 
\paragraph{Invariant subspaces for the derivative of holonomy maps --}
 
 We start by an elementary lemma that will be useful in the sequel.

\begin{lem}
\label{invariantspace}
Let $\Sigma_1,\Sigma_2$ be two open sets of $\R^2$ and $P:\Sigma_1\to\Sigma_2$ be a $C^1$-diffeomorphism on its image. Let $i=1,2$, $z_i\in\Sigma_i$ satisfying $P(z_1)=z_2$, and $\alpha_i\dans\Sigma_i$ be a $C^1$-arc passing through $z_i$. Let $(x_n)_{n\in\N}\in\alpha_1^{\N}$ with $x_n\neq z_1$ for every $n\in\N$. Assume that $x_n\to z_1$ as $n\to\infty$.
\begin{enumerate}
\item Assume that for every $n$, $P(x_n)\in\alpha_2$. Then $D_{z_1}P(T_{z_1}\alpha_1)\dans T_{z_2}\alpha_2$.
\item Assume moreover that $\Sigma_1=\Sigma_2=\Sigma$, $\alpha_1=\alpha_2=\alpha$, $z_1=z_2=z$. Then there exists $\lambda\in\R$ {(which depends on $\alpha$ and the sequence $x_n$ only)} such that for every $P:\Sigma\to\Sigma$ satisfying for every $n\in\N$, $P(x_n)=x_{n+1}$, we have that $\lambda$ is an eigenvalue of $D_zP$ in the direction $T_z\alpha$.
\item With the hypothesis of the item above, if we have for every $n\in\N$, $P(x_n)=x_n$ then $D_zP$ is the identity on $T_z\alpha$.
\end{enumerate}
\end{lem}

\begin{proof}
For $i=1,2$, we  consider a parametrization $\alpha_i:(-1,1)\to\Sigma_i$ with $\alpha_i(0)=z_i$. Let $y_n=P(x_n)$. By injectivity of $P$ we have $y_n\neq z_2$ for every $n$.

Let $s_n,t_n\in(-1,1)\moins\{0\}$ be such that $\alpha(s_n)=x_n$ and $\alpha_2(t_n)=y_n$. We must have $s_n,t_n\to 0$ as $n\to\infty$. Note that
\begin{equation}
\label{petiteequation}
\frac{P\circ\alpha_1(t_n)-P\circ\alpha_1(0)}{t_n}=\frac{s_n}{t_n}\frac{\alpha_2(s_n)-\alpha_2(0)}{s_n}.
\end{equation}

The left-hand side converges to $(D_{z_1}P)\dot{\alpha}_1(0)$. The second factor of the right-hand side converges to $\dot{\alpha}_2(0)$. This proves the first item.

Assume now that  $\Sigma_1=\Sigma_2=\Sigma$, $\alpha=\alpha_1$, $z=z_1=P(z_1)=z_2$ and for every $n\in\N$, $P(x_n)=x_{n+1}$. Then, $s_n=t_{n+1}$ and \eqref{petiteequation} becomes $(P\circ\alpha(t_n)-P\circ\alpha(0))/t_n=(t_{n+1}/t_n)(\alpha(t_{n+1})-\alpha(0))/t_{n+1}$.

 Letting $n\to\infty$ we find $(D_zP)\dot{\alpha}(0)=\lambda\dot{\alpha}(0)$ where $\lambda=\lim(t_{n+1}/t_n)$ depends only on $\alpha$ and $x_n$. The second item follows. The exact same argument proves the last item.
\end{proof}

\paragraph{Coincidence of the derivatives --} Let us come back to our context. Let $\cA^s_j\dans U$ be as constructed in Lemma \ref{l.continuation}: this is the union of stable periodic orbits $\gamma^{s,c}_j$ of $Y'$. Let $\Sigma^s_j\dans U$ be a small { two dimensional} section transverse to this annulus which is \emph{everywhere} {transverse} to $Y$.

The section $\Sigma^s_j$ is transverse to { all $\gamma_j^{s,c}$, which are} stable periodic orbits of $Y'$.
Note that since these orbits are periodic orbits of $Y$ there is a well defined first return map $P^s_j:S^s_j\to\Sigma^s_j$ for $Y$ where $S^s_j\dans \Sigma^s_j$ is an open neighbourhood of $\Sigma^s_j\cap\cA^s_j$.
For  such a $j$ set $x_j^{s,c}=\Sigma^s_j\cap\gamma^{s,c}_j$. Since $Y$ and $Y'$ are $C^0$-close in the neighbourhood  of $\cA^s_j$ we can assume that $Q^s_j$, the first return map to $\Sigma^s_j$ of $Y'$, is defined in the same neighbourhood.

\begin{lem}
 \label{l.specpoincouille}
 For every such $j$ we have
 $$D_{x_j^{s,c}}P^s_j=D_{x_j^{s,c}}Q^s_j.$$
In particular there exists $\lambda^{s,c}_j\in (0,1)$ such that
 $$\Spec\left(D_{x_j^{s,c}}P^s_j\right)=\{1,\lambda_j^{s,c}\}.$$
\end{lem}

\begin{proof}
Consider $\alpha=\cA^s_j\cap\Sigma^{s}_j$. This is an embedded arc which by Lemma \ref{l.continuation} consists of fixed points of both Poincar\'e maps. Hence their derivatives both have eigenvalue $1$ in the direction of $\alpha$.

Since $Y$ satisfies (MS$'$) and by Lemma \ref{l.memecombinato} there exists $x\in\Sigma^s_j\cap K_c$ which is not periodic under any of the flows and such that $\omega(x)=\gamma^{s,c}_j$. It cuts the arc $\Sigma^{s}_j\cap\mu^{-1}(c)$ at a monotone sequence of points $x_0=x,x_1,x_2,....,x_n,...$. Since $x\in K_c$ these points are on the same orbit of $P^s_j$ and $Q^s_j$. Using Item 2 of Lemma \ref{invariantspace} we deduce that the derivative $P^s_j$ and $Q^s_j$ have the same eigenvalue in the direction of the level set $\mu^{-1}(c)$. Note finally that since $\mu^{-1}(c)$ is orientable the eigenvalue $\lambda^{s,c}_j$ is positive. This concludes the proof of the lemma.
\end{proof}

\begin{rem}
An analogous statement for unstable periodic orbits may be proven by a symmetric argument
\end{rem}

{
\begin{rem}\label{r.change_metric}
A priori the plane field $\Pi$, the normal component $N$, the quotient function $\mu$ and the vector field $Y'$ depend on the choice of a Riemannian metric. However, after Lemma \ref{l.specpoincouille}, we see that the (MS')-hypothesis does not depend 
on this choice: the differential of the first return maps at periodic orbits of all $Y'$ are equal. As a consequence we have the freedom to change the Riemannian metric if needed.
\end{rem}
}

\subsection{Computation of the index}\label{s.calc-index}

The set $K$ is a disjoint union of finitely many compact connected components, which are unions of periodic orbits and of heteroclinic connections between them:
$$K=K^1\sqcup...\sqcup K^m.$$

We know that 

$$\ind(X,K)=\sum_{l=1}^m\ind(X,K^l).$$

Thus, we may restrict ourselves to the case where $K$ is connected. We can therefore consider a prepared triple $(U,X,Y)$ where the properties listed above hold and furthermore

\begin{center}
\emph{The sets $K_c$, $c\in(-\eps,\eps)$ are connected.}
\end{center}

\begin{rem}
\label{r.bs}
Recall that, by (MS$'$), for every $c$, $K_c$ contains more than one periodic orbit and that every stable periodic orbit is linked with some unstable one, and vice-versa.
\end{rem}

\paragraph{Stable manifolds --} Lemma \ref{l.specpoincouille} implies that every orbit $\gamma^{s,c}_j$ possesses a local stable manifold for $Y$ which is an embedded open annulus. 

The union for all $|c|<\eps$ of these local stable manifolds is an open neighbourhood of $\cA^s_j$, included in the open set $U$, foliated by surfaces: we denote by $\cW^s_j$ this foliation and by $W^{s,c}_j$ its leaves.

The dynamical meaning of these surfaces is the following. A point $x\in U$ belongs to $W^{s,c}_j$ if and only if $\omega_Y(x)=\gamma^{s,c}_j$. The next lemma states that these stable manifolds are invariant by the flows $X_t$ and $N_t$. 

{ Let $\Sigma^s_j$ be a small 2-dimensional transverse section of $Y$ cutting the annulus $\cA^s_j$ transversally. We also have transverse sections $\Sigma^u_i$ of $Y$ cutting $\cA^u_i$ transversally. We can choose them so they are trivially foliated by local stable and unstable manifolds $W^{s,c}_j$, $W^{u,c}_i$ respectively. We can modify the metric close to these sections so they are \emph{everywhere} orthogonal to $Y$: {by Remarks~\ref{rem.liberte} and \ref{r.change_metric} even if $N$ and $\mu$ are modified, this does not change relevant dynamical properties (like the combinatorics of $K_c$) nor the assumption (MS').} This implies in particular that we can take $N$ to be tangent to all sections $\Sigma^s_j$ and $\Sigma^u_i$.
}

\begin{lem}
 \label{l.invstableXN}
 For every $j$ the following assertions hold true.
 
\begin{enumerate}
\item For every $x\in W^{s,c}_j$, we have
$$X(x)\in T_xW^{s,c}_j.$$
\item For every $x\in\Sigma^s_j\cap W^{s,c}_j$ we have
 $$N(x)\in T_x \left(\Sigma^s_j\cap W^{s,c}_j\right).$$
\end{enumerate} 
\end{lem}

\begin{figure}[!h]
\centering
\includegraphics[scale=0.5]{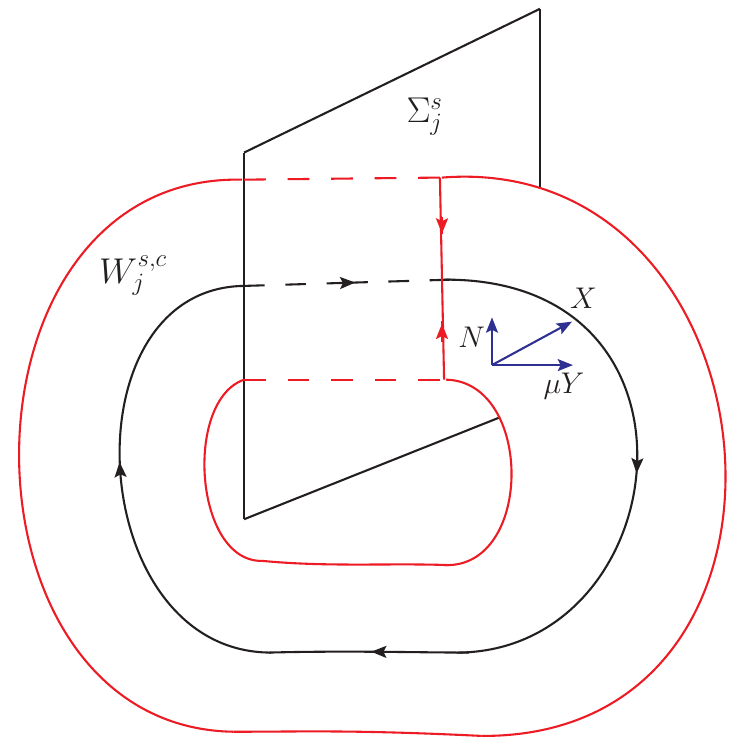}
\caption{$N$ is tangent to the stable manifolds}
\label{Nouille}
\end{figure}

\begin{proof}
Assume that $x\in W^{s,c}_j$: there exists $x_0\in\gamma_j^{s,c}$ such that $\dist(Y_t(x),Y_t(x_0))\to 0$ as $t\to\infty$. The flows of $X$ and $Y$ commute so, if $h$ is small enough so that $X_h(x)\in U$, we have
\begin{eqnarray*}
\dist(Y_t(X_h(x)),Y_t(X_h(x_0)))&=&\dist(X_h(Y_t(x)),X_h(Y_t(x_0)))\\
  &\leq&\sup_{x\in\overline{U}} ||D_x X_h||\,\,\dist(Y_t(x),Y_t(x_0))\to_{t\to\infty} 0.
\end{eqnarray*}
Since $x_0\in\gamma^{c,s}_j\dans\Col_U(X,Y)$ we have $X_h(x_0)\in\gamma_j^{s,c}$ which proves that $X_h(x)\in W^{s,c}_j$, and the first item follows.

In order to prove the second item, recall that $N$ is tangent to $\Sigma^s_j$. Clearly, it is enough to treat the case where $N(x)\neq 0$.
Since the set of zeros of $N$ is precisely $\Col_U(X,Y)$ this implies that the subspace of $T_xM$ spanned by $X(x)$ and $Y(x)$ has dimension $2$. By the first item, this subspace is contained in $T_xW^{s,c}_j$. 
Since $W^{s,c}_j$ has dimension $2$ this proves that $T_xW^{s,c}_j$ is the subspace spanned by $X(x)$ and $Y(x)$. Thus,
$$N(x)=X(x)-\mu(x)Y(x)\in T_xW^{s,c}_j.$$
This completes the proof. 
\end{proof}

\begin{rem}
\label{r.decompositionstablemfds}
 The intersection  $\Sigma^s_j\cap W^{s,c}_j$ is then a union of regular orbits and zeros of the vector field $N$. The zeros of $N$ correspond to intersections of $\Sigma^s_j$ with the collinearity locus $\Col_U(X,Y)$.
\end{rem}

\begin{rem}
Lemma \ref{l.invstableXN} is purely topological and the proof does not need the Morse-Smale hypothesis. If the {stable or unstable sets} of a periodic orbit exists (in the sense of \S \ref{s.general_dynamics}) then it must be invariant by the flow of $X$ as an immediate consequence of the commutation. If furthermore one knows that this set is a manifold, it follows that $X$ must be tangent to it.
\end{rem}

The same argument shows that orbits $\gamma^{u,c}_i$ also possess unstable manifolds $W^{u,c}_i$, which foliate an open subset of $V$: $\cW^u_i$ denotes this foliation. Moreover given a section $\Sigma^u_i$ of $\cA^u_i$ chosen everywhere orthogonal to $Y$, the intersection $W^u_i\cap\Sigma^u_i$ also writes as a union of regular orbits and zeros of $N$.

\begin{figure}[!h]
\centering
\includegraphics[scale=0.6]{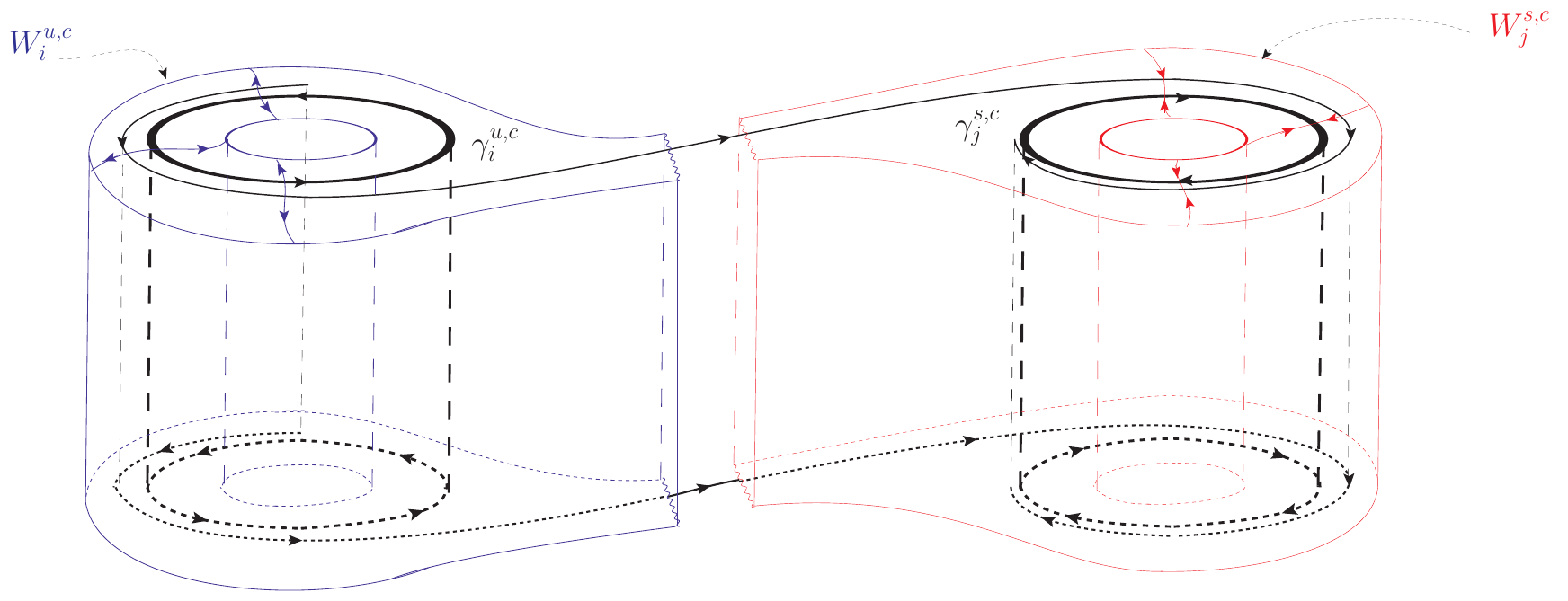}
\caption{Glueing stable and unstable manifolds}
\label{glouglouille}
\end{figure}

\paragraph{Glueing stable and unstable manifolds --} Our goal from now on is to use stable and unstable manifolds in order to build a neighbourhood of $K$ which is foliated by certain surfaces to which $X$ and $Y$ are tangent.

We can assume that every section $\Sigma^s_j$ is foliated by local stable manifolds $W^{s,c}_j$ of periodic orbits $\gamma^{s,c}_j$, and similarly every section $\Sigma^u_i$ is foliated
by unstable manifolds $W^{u,c}_i$ of periodic orbits $\gamma^{u,c}_i$, for $c\in(-\eps,\eps)$. Hence there are open neighbourhoods of the annuli $\cA^s_j$ and $\cA^u_i$, denoted by $\cV^s_j$ and $\cV^u_i$, 
which are foliated by annuli which are respectively local stable and unstable manifolds of $Y$. By Lemma \ref{l.invstableXN} $X$ is tangent to these annuli. In order to achieve our goal, it remains to foliate in a coherent
way neighbourhoods of heteroclinic connections by surfaces to which $X$ is tangent.

For every $i$ the manifold $(\Sigma^u_i\cap W^{u,0}_i)\moins\per(Y)$ has two connected components. There are exactly two possibilities.

\begin{enumerate}
\item Either only one of these components contains a point $x\in K\moins\per(Y)$.
\item Either the two connected components contain one.
\end{enumerate}

Fix such an $i$, and assume that the first property holds for $\Sigma^u_i\cap W^{u,0}_i$. Pick a point $x_i\in K\cap\Sigma^u_i$ whose orbit is a heteroclinic connection between the unstable periodic orbit $\gamma_i^{u,0}$ and some stable periodic orbit for $Y$. We consider the fundamental domain $J_i:=(P^u_i(x_i),x_i]\subset \Sigma^u_i\cap W^{u,0}_i$.

If the second property holds for $i$, we will consider two zeros of $X$, $x_i^+$ and $x_i^-$, that belong to different components of $(\Sigma^u_i\cap W^{u,0}_i)\moins\per(Y)$ as well as the two different corresponding fundamental domains $J_i^+,J_i^-$. In that case we set $J_i=J_i^+\cup J_i^-$.

\begin{lem}
 \label{l.domain}
If $x\in K\setminus\per(Y)$ then there exists some $i$ such that the $Y_t$-orbit of $x$ intersects $J_i$.  
\end{lem}

\begin{proof}
By our hypothesis if $x\in K$ there exists $i$ such that $\alpha_Y(x)=\gamma^{u,0}_i$. In particular its $Y_t$-orbit must intersect the section $\Sigma^u_i$ at some point $y$.

By definition the iterates of $J_i$ by the first return map $P_i^u$ cover the union of the connected components of $(\Sigma^u_i\cap W^{u,0}_i)\moins\per(Y)$ that contain an element of $K$. In particular an iterate of $y$ must lie inside $J_i$. Since this iterate belongs to the $Y_t$-orbit of $x$, the proof of the lemma is over.
\end{proof}

Using Lemma~\ref{l.domain} as well as the compactness of $\overline{J_i}$, we deduce the existence of a finite open cover $\{U^l_i\}_{l=1}^k$ of $J_i\cap K$ in $\Sigma^u_i$ having the following property

\begin{itemize}
 \item there exists a holonomy map for the flow of $Y$, $\hol^l_i:U_i^l\dans\Sigma^u_i\to\Sigma^s_j$ where $j=j(l)$ (Figure \ref{holtube}).
\end{itemize}

By reducing $\eps$ and the open sets $U_i^l$, if necessary, we can require further that

\begin{itemize}
 \item The sets $W^{u,c}_i\cap U^l_i$ form a codimension one foliation of each open set $U^l_i$.
 Then each $U^l_i$ is foliated by segments $L^{u,c}_i$ of local unstable manifolds $W^{u,c}_i$, for $c\in(-\eps,\eps)$. 
\end{itemize}

The following result is a key one and shows how one can glue together stable and unstable manifolds.

\begin{figure}[!h]
\centering
\includegraphics[scale=0.4]{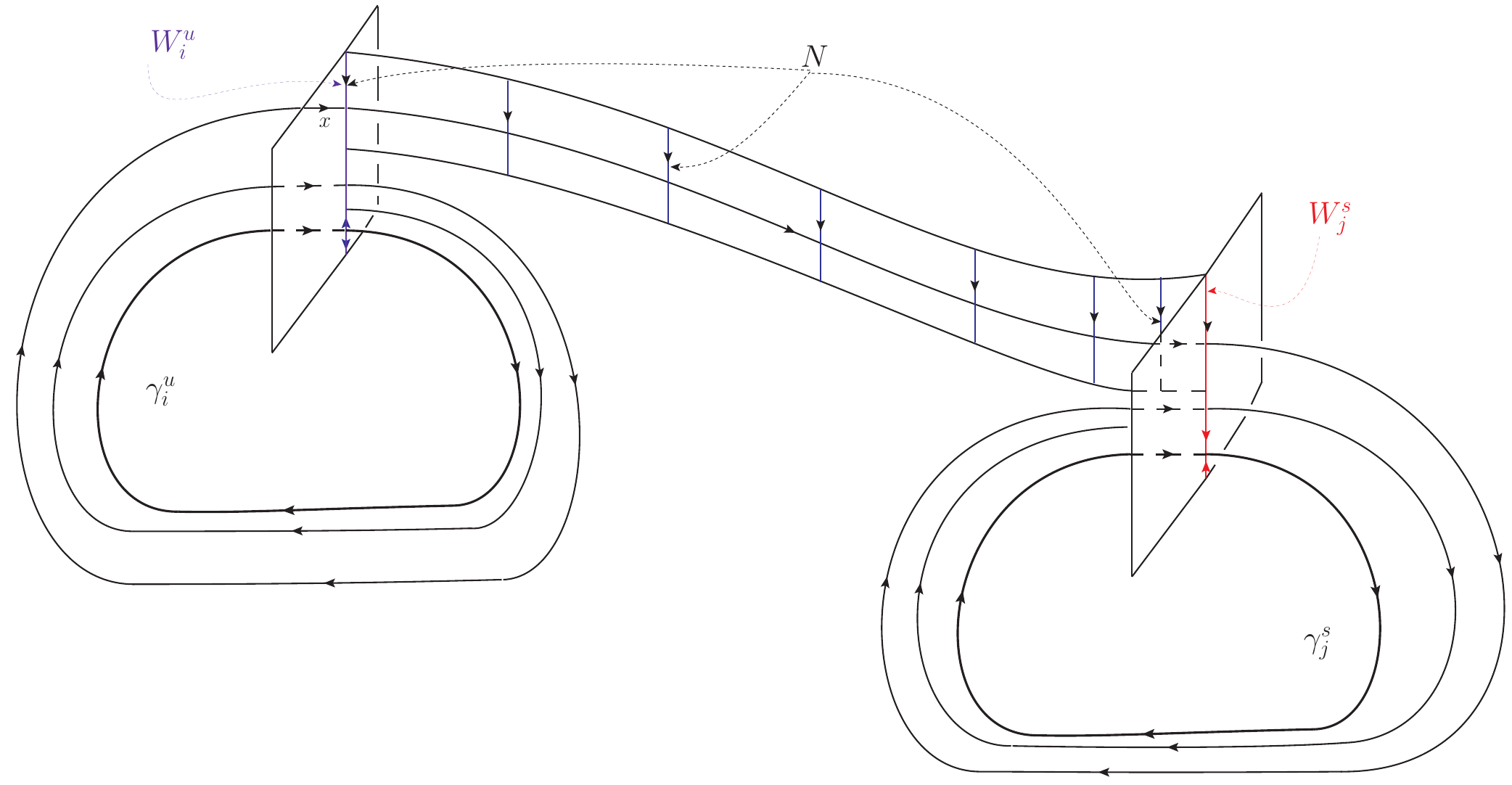}
\caption{The glueing lemma}
\label{glouglouglouille}
\end{figure}

\begin{lem}[Glueing Lemma]
 \label{l.collage}
Every holonomy map $\hol^l_i:U^l_i\dans\Sigma^u_i\to\Sigma^s_j$ carries the leaves $L^{u,c}_i$ onto segments $F^{s,c}_j$ of local stable manifolds $W^{s,c}_j$.    
\end{lem}

\begin{proof}
This key lemma is a simple consequence of Lemma \ref{calculusholonomies} and of Remark \ref{r.decompositionstablemfds}.

Indeed, the latter remark implies that segments $L^{u,c}_i$ are unions of regular orbits and zeros of $N$. Then Lemma \ref{calculusholonomies} implies that the holonomy map $\hol^l_i$ carries respectively regular orbits and zeros of $N$ onto regular orbits and zeros of $N$. A last application of Remark \ref{r.decompositionstablemfds} proves that these form segments of local stable manifolds: {see Figure \ref{glouglouglouille}}. The lemma follows.
\end{proof}

Using our Glueing Lemma, we now show how to glue stable and unstable manifolds. For each $i,l$ let $O^l_i$ denote the \emph{holonomy tube} of the transverse section $U^l_i$ {(see Figure \ref{holtube})}. 
More precisely, for each $x\in U^l_i$, there exists a smallest positive time $\tau(x)$ such that
$Y_{\tau(x)}(x)\in\Sigma^s_j$. Then, we put 

$$O^l_i:=\{Y_t(x);x\in U^l_i,0\leq t\leq\tau(x)\}.$$

For each leaf $L^{u,c}_i$ its holonomy tube is a surface which may be pasted smoothly with the stable manifold. By Lemma \ref{l.invstableXN} these surfaces are tangent to $X$ and $Y$.

Recall that the annuli $\cA_i^u$ and $\cA_j^s$ possess neighbourhoods with disjoint closures, that we denoted by $\cV^u_i$ and $\cV^s_j$, which are respectively foliated by unstable and stable annuli.

Finally using our Glueing Lemma \ref{l.collage} we obtain the following

\begin{cor}
The set 
$$\cV=\bigcup_{i,j,l}\cV^u_i\cup\cV^s_j\cup O_i^l$$
is an open neighbourhood of $K$, satisfying $\Zero(X)\cap\partial\cV=\vide$, which admits a foliation by surfaces $S_c$, $c\in(-\eps,\eps)$ such that, for every $x\in S_c$ 
$\langle X(x),Y(x)\rangle \dans T_xS_c.$
\end{cor}

\begin{figure}[!h]
\centering
\includegraphics[scale=0.5]{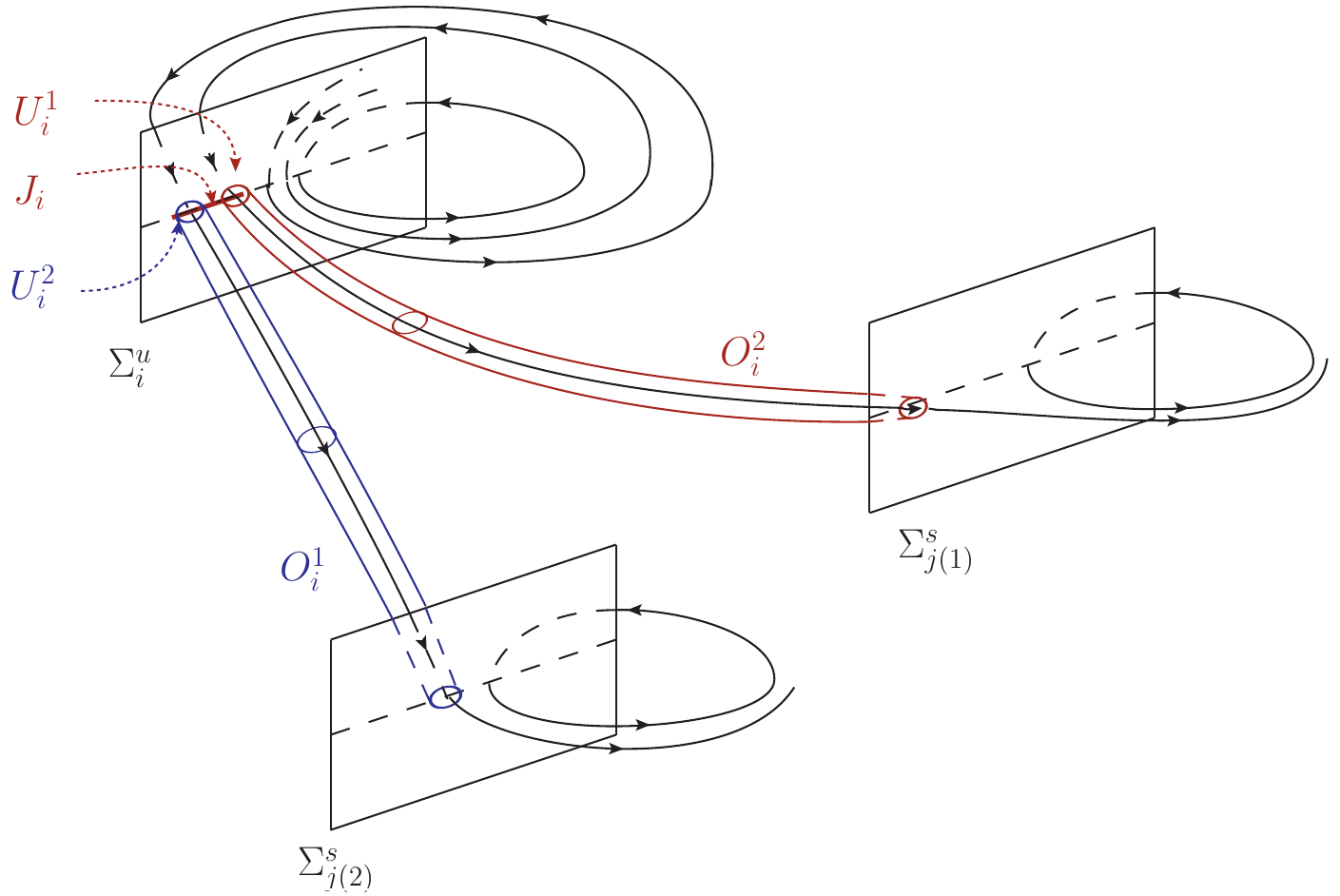}
\caption{Holonomy tubes}
\label{holtube}
\end{figure}

\paragraph{End of the proof of Theorem \ref{Th.MS} --}
In the next lemma, whose proof is automatic from the previous section, we build a $C^0$ basis for the tangent bundle over of $M$ restricted to $\overline{\cV}$.

\begin{lem}
\label{l.base}
There exists $C^0$ vector fields $\{e_1,e_2,e_3\}$ over $\overline{\cV}$ with the following properties
\begin{enumerate}
 \item $e_1=Y$ everywhere
 \item for every $x\in S_c$ one has $e_2(x)\in T_xS_c$ and $\langle e_1(x),e_2(x)\rangle=T_xS_c$.
 \item for every $x\in S_c$ one has $e_3(x)\neq 0$ and $\langle e_3(x)\rangle\cap T_xS_c=\{0\}$.
\end{enumerate}

\end{lem}

We can write, for each $x\in\overline{\cV}$, $X(x)=\sum_{l=1}^3\alpha_l(x)e_l(x)$. Notice that, since $X(x)$ is tangent to $S_c$, $\alpha_3(x)$ vanishes everywhere. As a consequence the Gauss map 
$$\alpha:x\in\partial \cV\mapsto\frac{(\alpha_1(x),\alpha_2(x),\alpha_3(x))}{\sqrt{\alpha_1(x)^2+\alpha_2(x)^2+\alpha_3(x)^2}}\in\mathbf{S}^2$$
takes its values in the equator $x_3=0$ and therefore has zero topological degree. We deduce that
$$\Ind(X,K)=0,$$
allowing us to conclude the proof of Theorem \ref{Th.MS}.

\section{Proof of Theorem \ref{t.valeurpropre}}
\label{s:final_section}

\subsection{Plan of the proof}

Until the end of the section $(U,X,Y)$ is a $C^3$-prepared triple such that
\begin{center}
$(\ast)\,\,\,\,$ \emph{The Poincar\'e map of every periodic orbit of $Y$ included in the collinearity locus $\Col_U(X,Y)$ has at least one eigenvalue of modulus different from $1$.}
\end{center}
{
Recall the function $Z:c\mapsto K_c=\Zero(X-cY)\cap\overline{U}$, which maps each small parameter $c$ to the compact set $K_c\subset\mu^{-1}(c)$, and that we consider the Hausdorff topology in the space of compact subsets of $M$ (see \S~\ref{sub:semi_cty}).} The set
\begin{equation}
\label{eq.residual}
\cR=\{\textrm{regular}\,\,\,\textrm{values}\,\,\,\textrm{of}\,\,\,\mu\}\cup\{\textrm{continuity}\,\,\,\textrm{points}\,\,\,\textrm{of}\,\,\,Z\}
\end{equation}
is a residual subset of a small interval centred at $0$. As $(U,X,Y)$ is prepared, we have $0\in\cR$. Here again we will assume that $Y$ is transversally oriented and that level sets $\mu^{-1}(c)$ are oriented.

\paragraph{Linked and non-linked periodic orbits --} The difference between Hypotheses (MS) and ($\ast$) is the following. Under hypothesis (MS), periodic orbits lying in $K$ are linked by non-periodic orbits, or are isolated. Under Hypothesis ($\ast$) we can have simultaneously isolated periodic orbits and accumulation of
non-linked periodic orbits and the dynamics of $Y'$ on $K_c$ might be wilder: we must study this phenomenon.

Recall that a $Y_t$-periodic orbit $\gamma\dans K_c$ is said to be \emph{linked} if there exists $x\in K_c\moins\per(Y)$ such that $\gamma=\omega(x)$ or $\gamma=\alpha(x)$. A periodic orbit which is not linked is called \emph{non-linked}.

We shall adopt the following notations.

\begin{itemize}
\item $K_l$ denotes the union of \emph{linked} periodic orbits of $Y$ included in $K$.
\item $K_{nl}$ denotes the union of \emph{non-linked} periodic orbits of $Y$ included in $K$.
\item $K_{ms}$ denotes $K\moins K_{nl}$. 
\end{itemize}

{Applying Theorem~\ref{t.diegosalgado} we see} that $K_{ms}$ is formed by \emph{linked} periodic orbits and by non-periodic orbits in $K$ whose $\alpha$ and $\omega$-limit sets are periodic orbits included in $K$.

\paragraph{Plan of the proof --} The proof of Theorem \ref{t.valeurpropre} goes along the following lines.
\begin{enumerate}
\item We prove that $K_{ms}$ and $K_{nl}$ are disjoint compact sets so that $K=K_{ms}\sqcup K_{nl}$ and
$$\Ind(X,K)=\Ind(X,K_{ms})+\Ind(X,K_{nl}).$$
\item We prove that $Y$ satisfies the hypothesis (MS) close to $K_{ms}$ and deduce that $\Ind(X,K_{ms})=0$.
\item We prove that there exist finitely many open sets $\cU_1,...,\cU_m\dans U$ which cover $K_{nl}$ enjoying the following properties.
      \begin{itemize}
              \item For $i=1,\ldots,m$, $\Zero(X)\cap\partial\cU_i=\vide$.
              \item $\Ind(X,K_{nl})=\sum_{i=1}^m\Ind(X,\cU_i)$.
              \item $\Col_U(X,Y)\cap\cU_i$ is an open annulus consisting of periodic orbits of $Y$.
      \end{itemize}
\item Using \cite{BS} we deduce that $\Ind(X,\cU_i)=0$ for every $i=1,...,m$ and, consequently, that $\Ind(X,K_{nl})=0$.
\end{enumerate}

\subsection{Position of stable and unstable manifolds with respect to level sets of $\mu$}

Until the end of the paper we assume that Hypothesis ($\ast$) holds. It asserts that every periodic orbit of $\Col_U(X,Y)$ possesses a local stable or unstable manifold. Under Hypothesis (MS$'$), all these stable and unstable manifolds are \emph{tangent} to the level set $\mu^{-1}(0)$. We will first establish this property for every periodic orbit of $K_l$.

\begin{prop}
\label{p.tangent}
Let $\gamma$ be a $Y_t$-periodic orbit included in $K_l$ and $z\in\gamma$. Then the local stable (resp. unstable) manifold at $z$ is tangent to $\mu^{-1}(0)$.
\end{prop}

\subsubsection{Nice tubular neighbourhoods} 
{In order to perform some of our arguments below (which are somewhat based on a careful account of the dimension of the objects we are looking for) we need to foliate in a very specific way small neighbourhoods of collinearity periodic orbits. }

\begin{lem}
\label{l.nicetub}
Let $c_0\in\cR$ and $\gamma\dans K_{c_0}$ be a periodic orbit of $Y_t$. Then, there exists a neighbourhood $\cU$ of $\gamma$ satisfying the following properties.
\begin{enumerate}
	\item $\cU$ fibers over $\gamma$ and the fiber $\Sigma(z)$ over any $z\in\gamma$ is an embedded disc transverse to $Y$ and $Y'$.
	\item The foliation by level sets of $\mu$ induces a trivial foliation of $\cU$ by annuli, denoted by $\cU_c=\mu^{-1}(c)\cap\cU$. These annuli are transverse to the sections $\Sigma(z)$.
	\item There exists $\eta>0$ such that for every $z\in\gamma$, the first return map $P$ to $\Sigma(z)$ is well-defined on $\Sigma_{\eta}(z)$, the $\eta$-neighbourhood of $z$ in $\Sigma(z)$. Moreover for every $y\in\Sigma_{\eta}(z)$, the forward $Y'$-orbit of $y$ does not leave $\cU$ before hitting $\Sigma(z)$ again.
\end{enumerate}
\end{lem}  

\begin{dfn}
	\label{nicenbd}
	A tubular neighbourhood of a periodic orbit of $Y$ satisfying the properties above will be called a \emph{nice tubular neighbourhood}.
\end{dfn} 
 
{Let us explain now how to construct nice tubular neighbourhoods. The argument is classical and elementary so we only give a sketch. 
 
\begin{proof}
The foliation by level sets $\mu^{-1}(c)$ is trivial in $U$, in particular the curve $\gamma$ is holonomy-free inside $\mu^{-1}(c_0)$. This foliation is sub-foliated by integral curves of $Y'$, so there exist charts trivializing both foliations simultaneously. Therefore using the compactness of $\gamma$, the fact that $Y$ and $Y'$ are close in a neighbourhood of $\Col_U(X,Y)$ for the $C^1$-topology and adapting the proof of the Long Tubular Flow Theorem (see \cite[Chapter 3, Proposition 1.1]{PdM}) we can construct the fibration $\Sigma$ over $\gamma$. The remaining properties follows from the same consideration.  	
\end{proof} 
}

\paragraph{Notations --} Until the end of the article we will adopt the following notations.

\begin{itemize}
\item For every $z\in\gamma$ the fiber $\Sigma(z)$ will be denoted by $\Sigma$ if there is no ambiguity.
\item $\cU_c$ will denote the annulus $\mu^{-1}(c)\cap\cU$.
\item For $z\in\gamma$, $I_c(z)$ will denote the embedded interval $\Sigma(z)\cap\cU_c=\Sigma(z)\cap\mu^{-1}(c)$. These intervals foliate $\Sigma(z)$. When there is no ambiguity we will write $I_c=I_c(z)$.
\end{itemize}

\subsubsection{Invariant subspaces for the derivatives of holonomy maps}

We state now a consequence of Lemma \ref{invariantspace}.  

\begin{lem}
\label{simtransv}
Let $\gamma\dans K$ be a periodic orbit of $Y$ such that there exists $x\in K\moins\per(Y)$ with $\omega_Y(x)=\gamma$. The following properties hold true. 
\begin{enumerate}
\item Let $\Sigma_1,\Sigma_2$ be two transverse sections of $Y_t$ cutting $\gamma$ at $z_1$ and $z_2$ respectively, such that there exists a holonomy map along $\gamma$ denoted by $P:(\Sigma_1,z_1)\to(\Sigma_2,z_2)$. Then the derivative $D_{z_1}P$ sends $T_{z_1}[\Sigma_1\cap\mu^{-1}(0)]$ to $T_{z_2}[\Sigma_2\cap\mu^{-1}(0)]$.
\item Let $\Sigma$ be a transverse section of $Y$ at $z\in\gamma$, and $P$ be the first return map to $\Sigma$. Then the eigenvalues of $D_zP$ are real.
\item If $W_{loc}^s(z)$ (resp. $W^u_{loc}(z)$) is transverse to $\mu^{-1}(0)$ at the point $z$, then for every $y\in\gamma$, $W^s_{loc}(y)$ (resp. $W^u_{loc}(y)$) is transverse to $\mu^{-1}(0)$.
\end{enumerate}
\end{lem}

\begin{proof}

Let $\alpha_1$ and $\alpha_2$ denote respectively $\Sigma_1\cap\mu^{-1}(0)$ and $\Sigma_2\cap\mu^{-1}(0)$. These are two curves passing through $z_1$ and $z_2$ respectively.

By hypothesis there exist points $x_0,y_0,x_1,y_1,x_2,y_2...$ such that for every $n\in\N$, $x_n\in\Sigma_1$, $y_n=P(x_n)\in\Sigma_2$, and $x_n\to z_1$, $y_n\to z_2$ as $n\to\infty$. These correspond to points of the forward orbit of $x$ meeting successively $\Sigma_1$ and $\Sigma_2$. Since by hypothesis $x\in K$ we must have $x_n,y_n\in K$ for every $n$. In particular, for every $n$ we have $x_n\in \alpha_1$ and $y_n\in \alpha_2$. We deduce the first item by applying Lemma \ref{invariantspace}.

Assume now the hypotheses of the second item. The first item applied to $P$ shows that $D_zP$ preserves a $1$-dimensional subspace of the $2$-dimensional space $T_z\Sigma$. Hence it does not possess a complex eigenvalue, and the second item follows.

The last item clearly follows from the first item and from the invariance of $E^s$ under the flow.
\end{proof}

\begin{rem}
The conclusion of Lemma \ref{simtransv} remain unchanged if $\gamma=\alpha_Y(x)$ for $x\in K\moins\per(Y)$.
\end{rem}

\subsubsection{Confined periodic orbits}

In this section we establish Proposition~\ref{p.tangent}. We will argue by contradiction assuming the existence of $\gamma\dans K$ accumulated by the orbit of some $x\in K$ and whose stable/unstable manifold is transverse to the level set $\mu^{-1}(0)$ at some point $z\in\gamma$. Using a Poincar\'e-Bendixson like 
argument we will prove that nearby levels contain periodic orbits \emph{confined} between the stable/unstable manifold of $\gamma$ and regular non-periodic orbits of collinearity 
between $X$ and $Y$ (see Figure \ref{confouille}), which accumulate to $\gamma$. We will then prove that $x$ belongs to the stable/unstable manifold of every such confined periodic orbit (see Figure \ref{finprouille}), contradicting the fact that these orbits lie on different level sets.

\paragraph{Hypotheses --} Let $\gamma\dans K_l$ be a periodic orbit for $Y$. Without loss of generality we assume that there exists $x\in K$ with $\omega_Y(x)=\gamma$. 

As mentioned before, we argue by contradiction. We  will assume that $\gamma$ possesses a local stable manifold, 
which must be a $2$-dimensional annulus since $\gamma$ is not a sink (see Lemma \ref{l.nosinks}).

By Lemma \ref{simtransv}, since the local stable manifold of $\gamma$ is not everywhere tangent to $\mu^{-1}(0)$ it must be \emph{everywhere} transverse to $\mu^{-1}(0)$.

We consider $\cU$, a nice tubular neighbourhood of $\gamma$. As we mentioned before, $\cU$ is foliated by annuli $\cU_c=\cU\cap\mu^{-1}(c)$ and for every $z\in\gamma$, the fiber $\Sigma(z)$ is trivially foliated by arcs $I_c(z)=\Sigma(z)\cap\mu^{-1}(c)$. 

Since the local stable manifold $W^s_{loc}(z)$ is transverse to $\cU_0$, for every $z\in\gamma$, it must be transverse to $\cU_c$ for $|c|$ small enough, and the intersection $\cU_c\cap W^s_{loc}(\gamma)$ must be a simple essential closed curve. We denote it by $\alpha_c$.

Note that the arcs $I_c(z)$, are simple, connect the two boundary components of $\cU_c$ and that the projected vector field $Y'$, which is tangent to the level sets $\mu^{-1}(c)$, is transverse to each of these arcs.

\paragraph{Monotone sequences and confined periodic orbits --} 

In this paragraph we obtain the main ingredient of the proof of Proposition~\ref{p.tangent}: the existence of confined periodic orbits.

Let us start by observing that the orientation inside $\cU$ provides each $I_c$ with an order $<$. This order is \emph{coherent}: if $y_1<y_2\in I_c$ and $y_1',y_2'\in I_{c'}$ are close enough to $y_1$ and $y_2$ respectively, then $y_1'<y_2'$.

We fix $z_0\in\gamma$, and consider a point $x\in I_{0}(z_0)\cap K$ such that $\OO^+_Y(x)\dans\cU$ and $\omega_Y(x)=\gamma$. We can assume that $z_0<x$, the case
$x<z_0$ is entirely analogous. On $I_0(z_0)$ the sequence $x_0=x,x_1,x_2,...,x_n,...$ given by the successive intersection points $\OO^+_{Y}(x)\cap I_0(z_0)=\OO^+_{Y'}(x)\cap I_0(z_0)$ converges to $z_0$ and is monotone, i.e.
$$z_0<...<x_n<...<x_2<x_1<x_0=x$$
(this is the principal ingredient of the proof of Poincar\'e-Bendixson's theorem: see \cite[Chapter 11, Proposition 1]{HirschSmale}) .

Now since the order $<$ is coherent (in the sense defined above) and since $Y'$ is transverse to the arcs $I_c$, the following holds. If $y\in\Sigma_{\eta}(z_0)\cap \cU_c$ is close enough to $x$ we also must have a monotone sequence
$$...<y_n<...<y_2<y_1<y_0=y$$
of intersection points $\OO_{Y'}^+(y)\cap I_c(z_0)$ (these points might a priori leave $\cU_c$).

When $|c|$ is small enough, the closed curve $\alpha_c$ meets $I_c(z_0)$ at a unique point that we will denote by $z_c$.

The next lemma shows how when $|c|$ is small enough, there exists a periodic orbit $\gamma_c\dans\cU_c$ of $Y'$ included in $K_c$ which is confined between $\alpha_c$ and the forward orbit of a point $y\in K_c\cap \cU_c$.

\begin{lem}[Confined periodic orbits]
\label{confined}
When $|c|$ is small enough there exists $y\in K_c\cap\cU_c$ satisfying the following properties.

\begin{enumerate}
\item The forward orbit $\OO^+_{Y'}(y)$ is included inside $\cU_c$ and $\gamma_c=\omega_{Y'}(y)$ is a periodic orbit for $Y'$ included inside $K_c$.
\item The intersection $\gamma_c\cap I_c(z_0)$ is reduced to a unique point $x_c$ satisfying $y'>x_c>z_c$ for every $y'\in\OO^+_{Y^{\prime}}(y)\cap I_c(z_0)$.
\item The periodic orbit $\gamma_c$ is simple, essential in $\cU_c$ and disjoint from $\alpha_c$.
\end{enumerate}
\end{lem}

\begin{figure}[!h]
\centering
\includegraphics[scale=0.6]{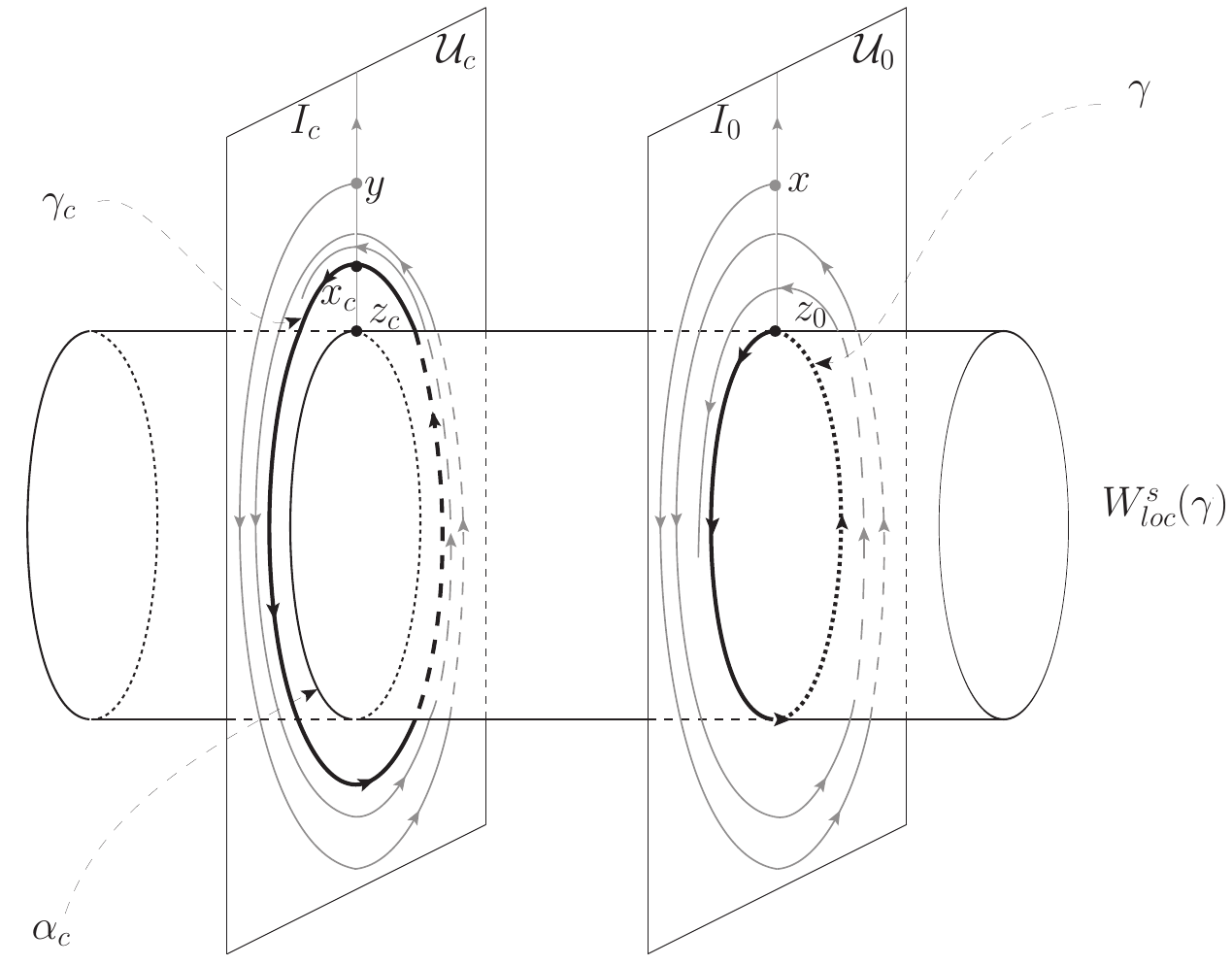}
\caption{Confined periodic orbits}
\label{confouille}
\end{figure}

\begin{proof}
We assumed that $z_0<x$ so by lower-semicontinuity of the function $Z:c\mapsto K_c\cap\cU$, there exists $y\in K_c\cap\cU$ (in particular $y\in\cU_c$) such that $y\in\Sigma_{\eta}(z_c)$ and $z_c<y$. As we explained in a previous paragraph if $y$ is chosen close enough to $x$ the successive points of $\OO^+_{Y'}(y)\cap I_c(z_0)$ form a decreasing sequence 
$$...<y_n<...<y_2<y_1<y_0=y$$
which might a priori be finite, if it were to leave $\cU$.

On the one hand $\OO_{Y'}(y)=\OO_Y(y)$ does not intersect $W^s_{loc}(\gamma)$, since otherwise, the forward $Y_t$-orbit of $y$ would accumulate on $\gamma$. This is impossible because these two orbits are included in different level sets of $\mu$.

On the other hand, $\alpha_c=W^s_{loc}(\gamma)\cap\cU_c$ is closed, simple and essential inside the annulus $\cU_c$. So it must disconnect this annulus. In other words, the forward $Y'_t$ orbit of $y$, which is decreasing and above $\alpha_c$, cannot leave $\cU_c$ without intersecting $\alpha_c$, which is included inside $W^s_{loc}(\gamma)$.

As a consequence the sequence $(y_n)_{n\in\N}$ defined above is infinite, decreasing and satisfies $y_n> z_c$ for every $n\in\N$. Therefore it must have a limit $x_c\in I_c(z_0)$. Since it is true that for every $n\in\N$, $y_{n+1}=P(y_n)$, we must have $P(x_c)=x_c$ and $\gamma_c=\OO_Y(x_c)$ is a periodic orbit of $Y'$.

Since $\OO_{Y^{\prime}}(y)=\OO_Y(y)\dans K_c$ we must have $\gamma_c\dans K_c$, proving the first item. The orbit $\gamma_c$ meets $I_c(z_0)$ at a unique point because it is transverse to the fibers $I_c$ and the intersection with $I_c(z_0)$ can't be monotone (this is another step in Poincar\'e-Bendixson's theorem). This prove the second item. It is simple because it is an orbit of $Y'$, and essential because it is transverse to the fibers $I_c$. Finally, it must clearly be disjoint from $W^s_{loc}(\gamma)$ since it is disjoint from $\gamma$ (the two periodic orbits of $Y$ lie on different level sets of $\mu$). This proves the last item.
\end{proof}

\begin{rem}
\label{convergingorbouilles}
The curves $\gamma_c$ and $\alpha_c$ are both essential, simple inside the annulus $\cU_c$, and they are disjoint. As a consequence they bound an annulus inside $\cU_c$.

The point $y$ obtained in Lemma \ref{confined} can be obtained close to a point of the forward orbit of $x$, which is arbitrarily close to $\gamma$. Hence the orbit $\gamma_c$ may be chosen arbitrarily close to $\gamma$ in the $C^3$-topology.
\end{rem}

\begin{figure}[!h]
\centering
\includegraphics[scale=0.6]{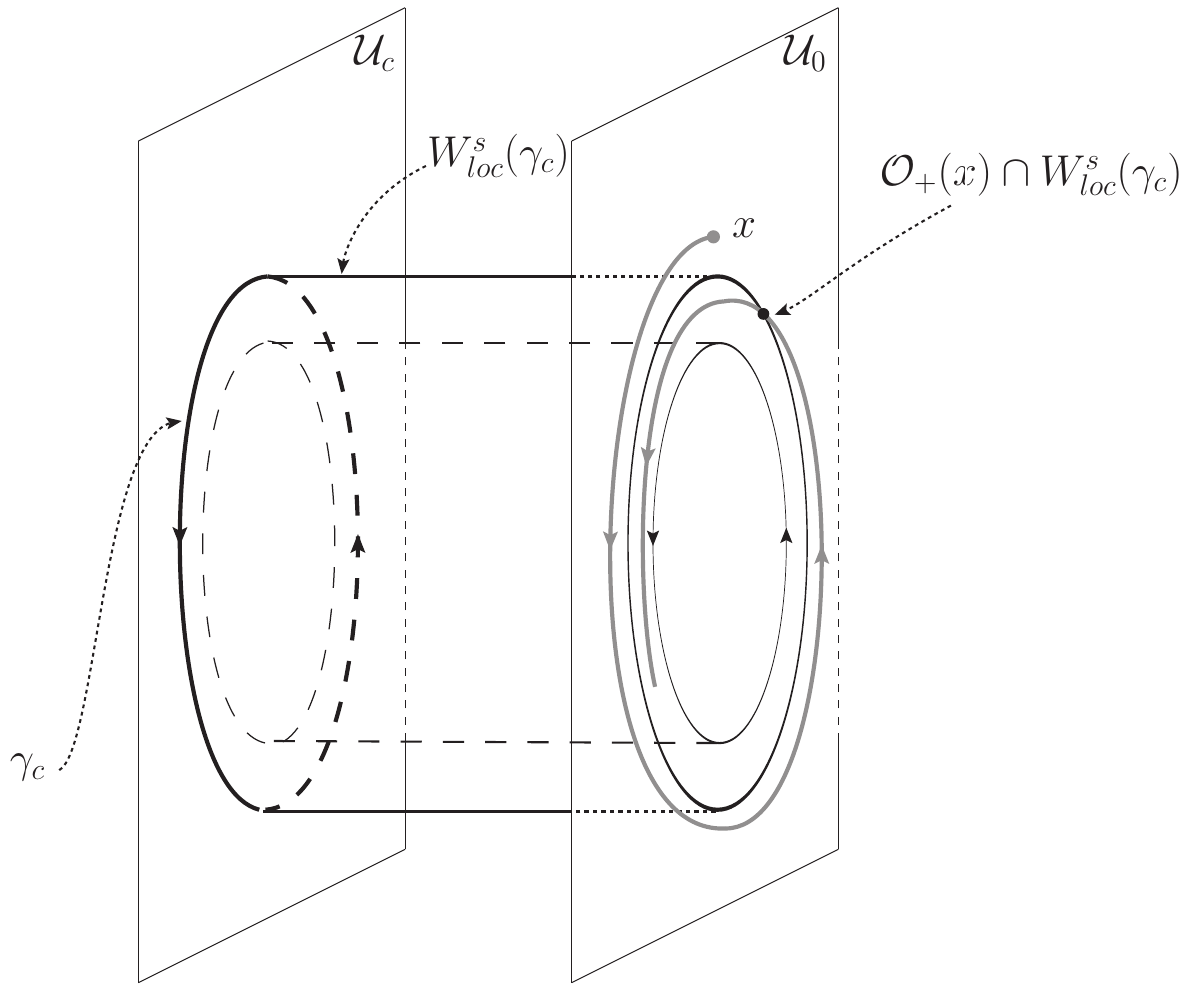}
\caption{Proof of Proposition \ref{p.tangent}}
\label{finprouille}
\end{figure}

\paragraph{End of the proof of Proposition \ref{p.tangent} --} By Remark \ref{convergingorbouilles} the periodic orbit $\gamma_c$ obtained in Lemma \ref{confined} can be chosen arbitrarily close to $\gamma$. In particular,
by continuity of the stable manifold, it is possible to assume that its local strong stable manifold $W^s_{loc}(\gamma_c)$ is transverse to $\cU_c$ and intersects transversally $\cU_0$ (see the stable manifold theorem for partial hyperbolicity \cite{HPS}).

The intersection $W^s_{loc}(\gamma_c)\cap\cU_0$, denoted by $\beta$ must be a simple closed curve, transverse to the fibers $I_0$ (if $c$ is small enough). So this simple closed curve must be essential. In particular it disconnects $\cU_0$. Also it must be disjoint from $\gamma$, because, as we saw in the proof of Lemma \ref{confined}, $\gamma$ cannot meet the stable manifold of $\gamma_c$.

By coherence of the order $<$, the point $x$ lies inside a connected component of $\cU_0\moins\beta$, and $\gamma$ is included inside the other one. We deduce that the forward orbit of $x$, which accumulates on $\gamma$, must intersect $\beta$. This is absurd because then we would have $\omega_Y(x)=\gamma_c$. Proposition \ref{p.tangent} follows.

\subsection{Structure of the linked and non-linked components of the zero set}

In this section, our main goal will be to show that there are only finitely many periodic orbits in $\Col_{U}(X,Y)$ with a stable (resp. unstable) manifold tangent to the level set of $\mu$ (see the Finiteness lemma \ref{l.lesorbitesliessonfini} below). { This will allow us to \textit{isolate} the disjoint sets $K_{nl}$ and $K_{ms}$. This step is fundamental in our strategy because it will permit us to calculate the Poincaré-Hopf index separately. }

\subsubsection{No saddle}

 The first step will be establishing that for a generic $c_0$, $Y_t$ has no hyperbolic periodic orbit included in $K_{c_0}$ (recall that by Lemma \ref{l.nosinks}, we know that for such a $c_0$, there is no attracting nor repelling periodic orbit included in $K_{c_0}$). This is a consequence of Proposition \ref{p.tangent}.
\begin{prop}[No saddle]
\label{p.nosaddle}
Let $c_0\in\cR$, $\gamma\dans K_{c_0}$ be a $Y_t$-periodic orbit and $P$ be the first return map at $z\in\gamma$. Then $D_zP$ has an eigenvalue of modulus $1$.
\end{prop}

The next lemma is an ingredient for the proof of Proposition~\ref{p.nosaddle}, which will also be useful in the sequel. 

\begin{lem}
	\label{orbitesconfinees}
	Let $c_0\in\cR$ be a continuity point of the map $c\mapsto K_c$. Assume that $\gamma\subset K_{c_0}$ is isolated in $K_{c_0}$, and let $\cU$ be an isolating neighbourhood.
	For $c$ close enough to $c_0$ the following properties hold true
	\begin{enumerate}
		\item $K_c\cap\cU\neq\vide$.
		\item For every $x\in K_c\cap\cU$ the complete orbit $\OO_Y(x)$ is included inside $\cU$.
	\end{enumerate}
\end{lem}

\begin{proof}
	The first item of the lemma holds when $c$ is close enough to $c_0$ due to the lower semi-continuity at $c_0$ of the map $Z:c\mapsto K_c$, and to the fact that $\gamma\dans K_{c_0}\cap \cU$.
	
	By our choice of $\cU$ we have $K_{c_0}\cap\partial\cU=\vide$ so, by compactness of $\partial\cU$, when $c$ close enough to $c_0$, we have $K_{c_0}\cap\partial\cU=\vide$.
	
	Consider now $c$, with $|c-c_0|$ small enough such that $K_c\cap\cU\neq\vide$. Take a point $x\in K_c\cap\cU$. By $Y_t$-invariance of $K_c$, we conclude that the orbit $\OO_Y(x)$ does not meet $\partial\cU$. Hence it must be entirely included in $\cU$, proving the second item.
\end{proof}

\begin{proof}[Proof of Proposition~\ref{p.nosaddle}]
Let $\gamma\subset K_{c_0}$ be a $Y_t$-periodic orbit. Let $\cU$ be a nice tubular neighbourhood of $\gamma$. Assume by contradiction that the Poincar\'e map $P:\Sigma_{\eta}(z)\to\Sigma(z)$ at a point $z\in\gamma$ has no eigenvalue of modulus $1$. By Lemma~\ref{l.nosinks}, $\gamma$ is hyperbolic of saddle type. We can assume that $z$ is the only fixed point of $P$ inside $\Sigma(z)$.

We cannot have $\gamma\subset K_l$, since otherwise Proposition~\ref{p.tangent} would imply that both the stable and unstable manifolds of $\gamma$ are tangent to the level set $\mu^{-1}(c_0)$, which is absurd. Therefore, $\gamma\subset K_{nl}$.  

We claim that $\{z\}=K_{c_0}\cap\Sigma(z)=K_{c_0}\cap I_{c_0}$. Indeed, assume by contradiction the existence of another $x\in K_{c_0}\cap I_{c_0}$. By our assumption on $\cU$, the point $x$ is not fixed by $P$. Since the orbit of $x$ under $Y$ and $Y'$ coincide we know that the successive intersection points of the orbit of $x$ with $I_{c_0}(z)$ are monotone. Without loss of generality we can suppose that $x>z$ and that for every $k\geq 0$
\begin{equation}
\label{e.zet}
z\leq P^{k+1}(x)\leq P^k(x).
\end{equation}
This sequence of elements of $K_{c_0}$ must converge to a fixed point $y$ of $P$. So it must be equal to $z$, which then has to belong to $K_l$. As we showed above, this is absurd.

Now we can use Lemma \ref{orbitesconfinees}. If $c$ is close enough to $c_0$ then there exists $x\in K_c\cap\cU$, and its full orbit $\OO_Y(x)$ is contained inside $\cU$. Using one more time the vector field $Y'$ we see that the sequence $(P^n(x))_{n\in\N}$ is monotone in $I_c(z)$ and thus accumulates to a fixed point $y$ of $P$. This contradicts the fact that $z$ is the unique fixed point of $P$ inside $\Sigma(z)$.
\end{proof}

\subsubsection{Finiteness lemma and index at the linked component}

Before stating our next result we need the following theorem stated below, which is a consequence of a general result about codimension $1$ foliations due to Haefliger (see \cite[Proposition 3.1, Th\'eor\`eme 3.2 ]{Haef}).

\begin{thm}[Haefliger]
	\label{l.haef}
	Let $Y$ be a non-vanishing vector field over a compact surface $S$, possibly with boundary. Then, for every $x\in S$ there exists a transverse segment which cuts every periodic orbit of $Y$ in at most one point. Moreover,
	one has that $\overline{\per(Y)}=\per(Y)$.
\end{thm}

\begin{lem}[Finiteness lemma]
  \label{l.lesorbitesliessonfini}
Let $c\in\cR$. Then, there are only finitely many periodic orbits of $Y$ which: 
\begin{enumerate}
 \item are included in $K_c$;
 \item have a stable (or unstable) manifold everywhere tangent to the level set $\mu^{-1}(c)$.
\end{enumerate}
\end{lem}
\begin{proof}[Proof of Lemma~\ref{l.lesorbitesliessonfini}]
We argue by contradiction and assume the existence, for $c\in\cR$, of an infinite sequence of periodic orbits for $Y$, $\gamma_n\subset K_c$ which have (say) a stable manifold everywhere tangent to the level set $\mu^{-1}(c)$. The case where all the $\gamma_n$
have an unstable manifold tangent to the level sets is analogous. 

Pick a sequence $x_n\in\gamma_n$. By compactness of $K_c$, and passing to a subsequence if necessary, we can suppose $x_n\to x$ for some $x\in K_c$. Let $\Sigma$ be a transverse section of $Y$ containing $x$ and let $\gamma$ denote $\OO_Y(x)$. Since $K_c$ is $Y_t$-invariant, we have $\gamma\subset K_c$. Note that, when restricted to $K_c$, the two vector fields $Y$ and $Y'$ coincide, so in particular $\gamma_n\dans\per(Y')$ for every $n$.

Applying Theorem \ref{l.haef} to the vector field $Y'|_{\mu^{-1}(c)}$ we deduce two things. Firstly, $\gamma$ is a periodic orbit for $Y'$. Secondly there exists an arc $I_c\dans\mu^{-1}(c)$ such that for $n$ large enough $I_c\cap \gamma_n$ is reduced to a point $y_n$, which is a fixed point for the first return map $P'$ of $Y'$ to $I_c$. We may  assume $I_c\dans\Sigma$. Since $\gamma_n,\gamma\dans K_c$ we deduce that $\gamma$ is a periodic orbit for $Y$ and that for $n$ large enough, $y_n$ are fixed points of $P$, a first return map of $Y$ to $\Sigma$.

By reducing $\Sigma$ if necessary we can take a trivialization $(e_1,e_2)$ of the tangent bundle of $\Sigma$, such that the vector fields $e_1$ and $e_2$
have the following properties:
\begin{itemize}
\item $e_1$ is everywhere tangent to the intervals $I_c$
\item $e_2$ is everywhere orthogonal to the intervals $I_c$
\end{itemize} 

We deduce from the third item of Lemma~\ref{invariantspace} that $D_xP$ leaves invariant $T_xI_c$ and induces the identity on that space. By our assumption ($\ast$)
$\gamma$ must have a hyperbolic (stable or unstable) invariant subspace for $D_xP$ which is transverse to $T_x I_c\dans T_x\mu^{-1}(c)$.

From these remarks we deduce that the matrix of the linear map $D_xP:T_x\Sigma\to T_x\Sigma$ when written in the basis $(e_1(x),e_2(x))$ has the form
$$\left( \begin{array}{cc}
1 & c \\ 
0 & \beta
\end{array} \right).$$
By our main assumption ($\ast$) we must have $|\beta|\neq 1$. 
On the other hand, the matrix of $D_{x_n}P:T_{x_n}\Sigma\to T_{x_n}\Sigma$ has the form
$$\left( \begin{array}{cc}
\lambda_n & c_n \\ 
0 & \beta_n
\end{array} \right),$$
with $|\lambda_n|<1$ because the periodic orbits $\gamma_n$ have a stable manifold tangent to the level set. 

By continuity we have $\beta_n\to\beta$ and $c_n\to c$. In particular for $n$ large enough the linear map $D_{x_n}P$ have two eigenvalues $\lambda_n$ and $\beta_n$,
both of them with modulus different from $1$.  

This implies that $\gamma_n$ for $n$ large is, either a sink, which contradicts Lemma \ref{l.nosinks}, or a saddle type
hyperbolic periodic orbit, which is in contradiction with Proposition~\ref{p.nosaddle}. This concludes the proof.
\end{proof}

\paragraph{Structure of linked periodic orbits --}
{From the Finiteness Lemma together with Proposition~\ref{p.tangent} we deduce that $K_{ms}$ is isolated from $K_{nl}$.

\begin{cor}
\label{cor.isolados}
There are neighbourhoods $U$ of $K_{ms}$ and $V$ of $K_{nl}$ such that $U\cap V=\emptyset$.
\end{cor}

\begin{figure}[h!]
	\centering
	\includegraphics[width=200pt,height=200pt]{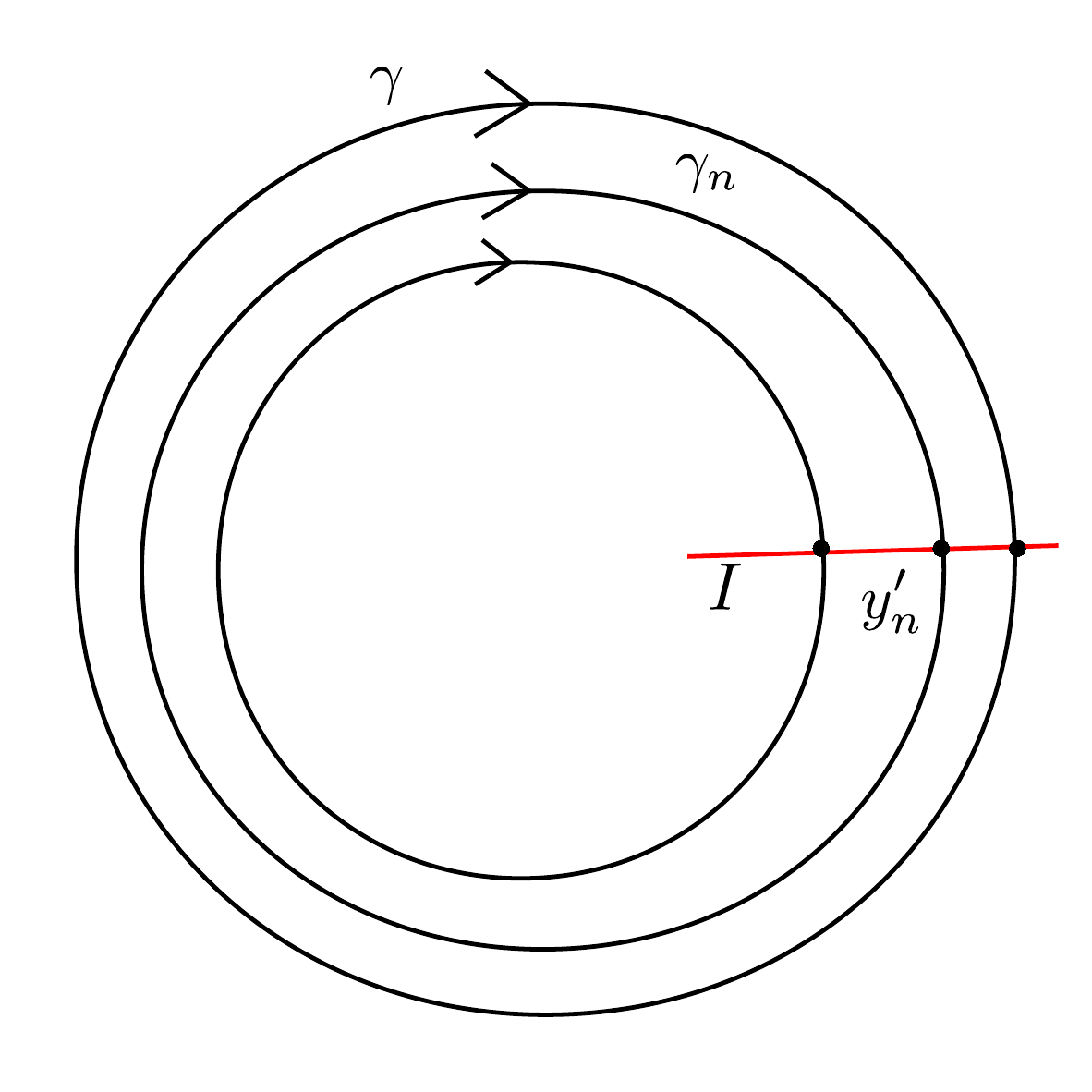}
	\caption{Proposition~\ref{p.tangent} and Lemma~\ref{l.lesorbitesliessonfini} imply the isolation of the sets $K_{ms}$ and $K_{nl}$.}
	\label{f.isol}
\end{figure}

\begin{proof}
From the Finiteness Lemma together with Proposition~\ref{p.tangent} we deduce that $K_{ms}$  is a union of finitely many \emph{linked} periodic orbits and non-periodic orbits linking them. In particular, there is no sequence $x_n\in K_{ms}$ such that $x_n\to x\in K_{nl}$. Assume by contradiction that the result is not true. Then, we necessarily must have a sequence $y_n\in K_{nl}$ converging to a point in $K_{ms}$. By Theorem~\ref{l.haef} the periodic orbit $\gamma_n=\OO(y_n)$ accumulates on some linked periodic orbit $\gamma$. Now, consider a small arc $I\subset\mu^{-1}(0)$ everywhere transverse to $Y'$ intersecting transversally $\gamma$ and infinitely many $\gamma_n$ at points $y^{\prime}_n$, which are fixed points for the first return map. By Lemma~\ref{invariantspace} one deduces that the derivative of the first return map in the direction of $I$ is the identity. But this contradicts Proposition~\ref{p.tangent}, because $\gamma$ is linked.  
\end{proof}

Since $K_{nl}$ is entirely formed by periodic orbits, by combining Theorem~\ref{l.haef} with Corollary~\ref{cor.isolados} we conclude that it is compact. As we have the decomposition 
$$K=K_{ms}\sqcup K_{nl}$$
by definition, it follows that $\ind(X,K)=\ind(X,K_{ms})+\ind(X,K_{nl})$. }

The goal of this section is to establish the result below.

\begin{prop}
	\label{p.linkedcase}
	There exists a neighbourhood $U$ of $K_{ms}$ such that $Y^{\prime}|_{\mu^{-1}(0)\cap U}$
	is Morse-Smale.
\end{prop}

By Theorem~\ref{Th.MS} we obtain

\begin{cor}
	\label{c.casoms}
	$\ind(X,K_{ms})=0$
\end{cor}

The finiteness results of previous section will play a crucial role.

\begin{prop}
\label{linked}
The set $K_l$ of linked periodic orbits is a finite union of periodic orbits $\{\gamma_1,...,\gamma_n\}$ of $Y$.
\end{prop}

\begin{proof}
This is an immediate consequence of Proposition \ref{p.tangent} and Lemma \ref{l.lesorbitesliessonfini}.
\end{proof}

Notice that $K_l\subset\per(Y^{\prime})$. From this remark, the proof of Proposition~\ref{p.linkedcase} will be achieved from the two subsequent lemmas below.

\begin{lem}
	\label{l.hiperbolicidade}
	Every $\gamma_j\dans K_l$ is hyperbolic for $Y^{\prime}|_{\mu^{-1}(0)}$.
\end{lem}

\begin{proof}

	Without loss of generality we may assume that $\gamma_j$ has a stable manifold for $Y$. Choose $\cU$, a nice tubular neighbourhood of $\gamma_j$, and $z\in\gamma_j$. Consider the associated fiber $\Sigma=\Sigma(z)$ and arc $I_0=I_0(z)$. Let $P_j$ and $P'_j$ be respectively the first return maps of $Y$ and $Y'$ to $\Sigma$. Note that the arc $I_0$ is $P_j'$-invariant (recall that $Y'$ is tangent to $\mu^{-1}(0)$).
	
	By hypothesis $\gamma_j\dans K_l$ so (assuming for example that it has a stable manifold) there exists a sequence of points $x_n\in K\cap I_0$ converging to $z$ with $x_{n+1}=P_j(x_n)=P_j'(x_n)$. We deduce two things. By Proposition \ref{p.tangent} $P_j$ has a stable manifold tangent to $I_0$. By Lemma  \ref{invariantspace} $D_zP_j$ and $D_zP_j'$ have the same eigenvalue in the direction of $I_0$. This implies that $\gamma_j$ is stable for $Y'_{\mu^{-1}(0)}$.

\end{proof}

\begin{lem}
	\label{l.tubes}
There exists a neighbourhood $U$ of $K_{ms}$ such that for every $x\in\mu^{-1}(0)\cap U$ there exists $i=i(x)$ and $j=j(x)\in\{1,...,n\}$
such that $\alpha_{Y'}(x)=\gamma_i$ and $\omega_{Y'}(x)=\gamma_j$.	
\end{lem}
\begin{proof}
It follows from the previous lemma that every 
$\gamma_j$ is either a sink or a source for $Y^{\prime}|_{\mu^{-1}(0)}$. As a consequence, for every $j$ there exists a neighbourhood $U_j$ of $\gamma_j$ such that $U_j\cap\mu^{-1}(0)$ is either an attracting or repelling neighbourhood of $\gamma_j$ for $Y'|_{\mu^{-1}(0)}$

{If $x$ belongs to a periodic orbit there noting left to prove. So, assume that $x$ is a non-periodic point. By Theorem~\ref{t.diegosalgado}} for every $x\in K_{ms}\moins K_l$ there exists $\gamma_i$ and $\gamma_j$ such that $\alpha_{Y'}(x)=\gamma_i$ and $\omega_{Y'}(x)=\gamma_j$. By compactness and the long tubular flow theorem, there are finitely many open sets $V_{i,j}$ such that 
$$K_l\subset\bigcup_{i,j}(V_{i,j}\cap\mu^{-1}(0))\bigcup_j(U_j\cap\mu^{-1}(0)),$$ 
and for every $x\in V_{i,j}\cap\mu^{-1}(0)$ one has $\alpha_{Y'}(x)=\gamma_i$ and $\omega_{Y'}(x)=\alpha_j$. The lemma follows.   
\end{proof}

\subsubsection{Structure of the non-linked component}

\paragraph{Decomposition of $K_{nl}$ --} Let $\Kta$ be the subset of $K_{nl}$ consisting of those periodic orbits having a local stable/unstable manifold everywhere tangent to $\mu^{-1}(0)$. Let $\Ktr=K_{nl}\moins\Kta$.

\begin{lem}
\label{l.structurenonlinked}
The following properties hold true.
\begin{enumerate}
\item $\Kta$ is a finite union of periodic orbits;
\item $\Kta$ and $\Ktr$ are compact;
\item $K_{nl}=\Kta\sqcup\Ktr$.
\end{enumerate}
\end{lem}
\begin{proof}
	The third item is automatic from the definition of the sets $\Kta$ and $\Ktr$, while the first item and (consequently) the compactness of $\Kta$ follows from the finiteness lemma (Lemma~\ref{l.lesorbitesliessonfini}). 
	
	Finally, to see that $\Ktr$ is compact we argue as in the proof of Lemma~\ref{l.lesorbitesliessonfini}.  
	Indeed, take a sequence $x_n\in\gamma_n$, with $\gamma_n\subset\Ktr$. If $x=\lim_{n\to\infty}x_n$, by Haefliger's theorem there exists a periodic orbit $\gamma\dans K$ such that $x\in\gamma$. 
	
	As in the proof of Lemma~\ref{l.lesorbitesliessonfini} one obtains from Lemma~\ref{invariantspace} that the first return map induces the identity tangentially to the level $\mu^{-1}(0)$. Therefore, the stable/unstable manifold of $\gamma$ must be transverse to the level set, and so $x\in\Ktr$. This completes the proof.	
\end{proof}

The remaining sections of the paper will be devoted to proving the following proposition.

\begin{prop}
\label{Knouille}
\begin{enumerate}
\item We have $\Ind(X,\Kta)=0$.
\item We have $\Ind(X,\Ktr)=0$.
\end{enumerate}
\end{prop}

As the Poincar\'e-Hopf index is additive we obtain 

\begin{cor}
	\label{c.ultimoindice}
	$\Ind(X,K_{nl})=0$.
\end{cor}

Since  $\ind(X,K)=\ind(X,K_{ms})+\ind(X,K_{nl})$, by Corollaries~\ref{c.casoms} and \ref{c.ultimoindice}  we have that the proof of Theorem~\ref{t.valeurpropre} is reduced to that of Proposition~\ref{Knouille}.

\subsection{Index at the non-linked component}

The objective of this section is to prove Proposition \ref{Knouille}. Recall that we have obtained a decomposition of the zeros of $X$
$$K=K_{ms}\sqcup\Kta\sqcup\Ktr.$$
In particular, $\Kta$ and $\Ktr$ are isolated compact sets of zeros of $X$. Our strategy will be to show that these
sets are included in some $C^1$ surface, and then to apply the results of \cite{BS}. The main ingredient of our argument is the Center Manifold Theorem. 

\subsubsection{The Center Manifold Theorem}

In this section we state a version of the classical Center Manifold Theorem \cite{HPS},\cite{wiggins2013normally} which is suitable for our purposes. 

Before giving the statement, let us give the general context. Let $Y$ be a $C^1$ vector field on a $3$-manifold $M$ and let $\gamma$ be a periodic orbit of $Y$. Let $z\in\gamma$ and consider the first return map $P:\Sigma_\eta(z)\to\Sigma(z)$, to some
section $\Sigma(z)$ everywhere transverse to $Y$. Assume that the derivative $D_zP:T_z\Sigma(z)\to T_z\Sigma(z)$ has a center unstable partially hyperbolic splitting $E^c\oplus E^{u}$, with the respective eigenvalues being $1$ and $\lambda$ with $|\lambda|>1$, or a center stable partially hyperbolic splitting $E^{s}\oplus E^c$ with respective eigenvalues being $\lambda$ and $1$ with $|\lambda|<1$. With these notations, one has the following result.

\begin{thm}[Center Manifold Theorem]
\label{t.center}
There exists a $C^1$ curve $W^c_{loc}(z)$ contained in $\Sigma(z)$ (called the center manifold), which is tangent to $E^c$ at $z$ and
enjoys the following dynamical property: there exists $\eta>0$ such that 
$$\fix(P)\cap\Sigma_{\eta}(z)\subset W^c_{loc}(z).$$  
\end{thm}	  

See also the main theorem in \cite{bonatti_crovisier2015center}, from which the above statement follows as a particular case.

\subsubsection{The tangential case}

By Lemma \ref{l.structurenonlinked} $\Kta$ is the disjoint union of finitely many periodic orbits. Those periodic orbits are isolated inside $K$. Hence $\Ind(X,\gamma)$ is well defined for every $Y_t$-periodic orbit included in $\Kta$. Hence, item (1) of Proposition~\ref{Knouille} follows from the lemma below.
 
\begin{lem}
\label{tangentialcase}
Let $\gamma$ be a periodic orbit of $Y$ included in $\Kta$. Then
$$\Ind(X,\gamma)=0.$$
\end{lem}

The rest of this paragraph is devoted to proving Lemma~\ref{tangentialcase}. Thus, let $\gamma$ be a periodic orbit of $Y$ included in $\Kta$ and take a nice tubular neighbourhood $\cU$ of $\gamma$ (see Definition \ref{nicenbd}). Since $\gamma$ is isolated in $K$, it is possible to choose $\cU$ so that $K\cap\overline{\cU}=\gamma$.

We will assume that $\gamma$ has a local stable manifold $W^s_{loc}(\gamma)$ which is tangent to $\mu^{-1}(0)$. The case of an unstable manifold follows from a symmetric argument.

We will now assume that $\cU$ is small enough so that for every $c$ such that $\mu^{-1}(c)\cap\cU\neq\vide$, Properties (1) and (2) of Lemma \ref{orbitesconfinees} hold. { By reducing it if necessary, we can assume that it is a nice tubular neighbourhood of $\gamma$ (recall Lemma~\ref{l.nicetub} above). In particular, $\cU$ comes equipped with a fibration by disks $\Sigma$ over $\gamma$ and a trivial foliation by annuli $\cU_c$}. Choose $z\in\gamma$ and consider the fiber $\Sigma=\Sigma(z)$ which is foliated by embedded intervals $I_c=\cU_c\cap\Sigma$. Consider $P$, the first return map to $\Sigma$ defined in a neighbourhood of $z$.

By hypothesis $T_{z}\Sigma=E^{s}(z)\oplus E^c(z)$ where $E^{s}(z)$ is tangent at $z$ to  $I_0=\cU_0\cap\Sigma$ and $E^c(z)$ is the eigenspace of $D_{z}P$ corresponding to the eigenvalue $1$. By the Center Manifold Theorem there exists a $C^1$-embedded interval $W^c_{loc}(z)\dans\Sigma$ tangent at $z$ to $E^c(z)$ satisfying the dynamical property (for some $\eta>0$ small enough) 
$$\fix(P)\cap\Sigma_\eta(z)\subset W^c_{loc}(z).$$

\begin{lem}
\label{cletangent}
There exists a smaller nice tubular neighbourhood $\cV\dans\cU$ satisfying the same properties as $\cU$ (the fiber containing $z$ is still denoted by $\Sigma$), as well as the following properties.

\begin{enumerate}
\item $\Col_{\cV}(X,Y)\cap\Sigma=W^c_{loc}(z)\cap\cV$.
\item All points of $W^c_{loc}(z)\cap\cV$ are fixed points of the Poincar\'e map $P$.
\end{enumerate}
\end{lem}

\begin{proof}
We start the proof by making an elementary observation. The center manifold $W^c_{loc}(z)$ is an embedded {interval} transverse to $I_0$. As a consequence, we can shrink $\cU$ so that it keeps the properties of Lemma \ref{orbitesconfinees} (note that $\cU\cap K=\gamma$) and $W^c_{loc}(z)$ is { an interval} embedded in $\cU$ that crosses \emph{transversally} all level sets $I_c$ contained in $\cU$. In particular $W^c_{loc}(z)$ is a {transverse} section of the trivial foliation defined by intervals $I_c$, and intersects every such interval at exactly one point. Since $W^c_{loc}(z)$ contains every fixed point of $P$ this implies that for $|c|$ small enough, $I_c$ contains at most one fixed point of $P$.

\begin{figure}[!h]
\centering
\includegraphics[scale=0.5]{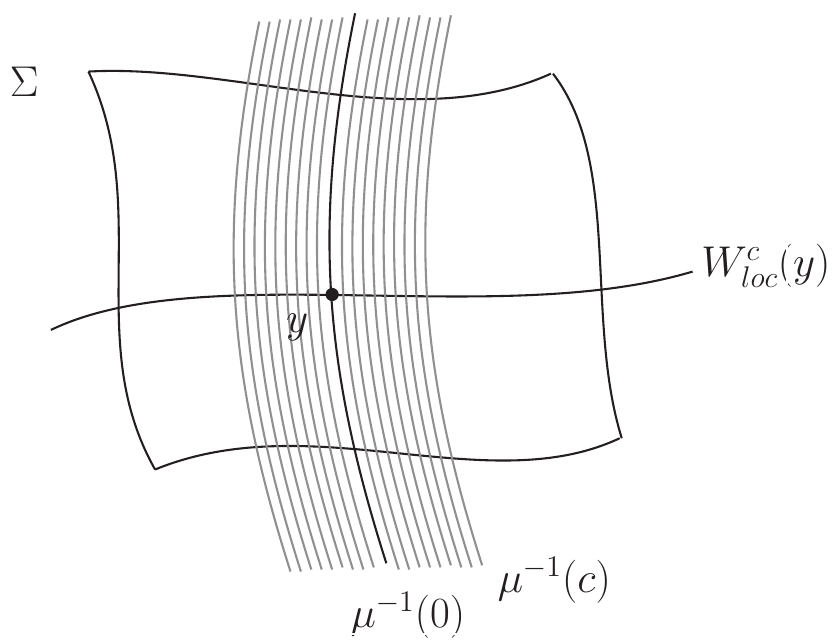}
\caption{The center manifold is transverse to the level sets}
\label{Bis-eps-converted-to}
\end{figure}

Consider an interval $I_c$ intersecting $\cV$. Since $\cV$ enjoys the properties of Lemma \ref{orbitesconfinees} there exists $x\in\cV\cap K_c$. Using the second property of Lemma \ref{orbitesconfinees} we see that $P^n(x)$ is well defined for every $n\in\Z$ and belongs to $I_c$. Hence by compactness, the points $q_c=\lim_{n\to-\infty}P^n(x)$ and $p_c=\lim_{n\to\infty}P^n(x)$ exist, belong to $I_c$ and are fixed points of $P$. 

Since $I_c$ contains at most one fixed point we must have $p_c=q_c$. Recall that intervals $I_c$ are endowed with a coherent order. Using the vector field $Y'$ we see that the sequence $(P^n(x))_{n\in\Z}$ is monotone. So we must have $q_c=x=p_c$.

Consequently, every level set $I_c$ intersects the collinearity set at a unique point, which must be a fixed point of $P$, and therefore must belong to $W_{loc}^c(z)$. Since $W^c_{loc}(z)$ meets such a level set at a unique point, and since these level sets foliate $\cV$, the lemma is proven.
\end{proof}
{ To complete the proof of Lemma~\ref{tangentialcase} we shall need a technical version of the main result in \cite{BS}. To avoid introducing more heavy notation we shall not state it with the exact same words as in \cite{BS} - for that we refer the reader to Definition 4.2 of that paper, which contains the assumptions of the result below with a very detailed presentation. 
\begin{thm}[Theorem A of \cite{BS} - Technical version]
	\label{t.bsdenovo}
Let $(U,X,Y)$ be a prepared triple satisfying the following additional assumptions
\begin{enumerate}
	\item $\Zero(X)\cap U=\gamma$ a $Y_t$-periodic orbit.
	\item $U$ is a nice tubular neighbourhood of $\gamma$.
	\item There exists an embedded annulus $S$ which is trivially foliated by periodic orbits of $Y$ and so that $\Col_{U}(X,Y)=S$, with $\Zero(X-cY)$ being a $Y_t$-periodic orbit
\end{enumerate}
Then, $\ind(X,U)=0$.
\end{thm}

This result, although not precisely stated in \cite{BS} is the main step in the course of proving Theorem~\ref{t:BS}. Indeed, in that paper one assumes that there exists a counterexample to Theorem~\ref{t:BS} (the main theorem in \cite{BS}). Then, Lemma 4.3 of that reference proves that there exists another counter-example satisfying precisely the assumptions of Theorem~\ref{t.bsdenovo}, but with $\ind(X,U)\neq 0$. The rest of the paper then shows that this leads to a contradiction. Thus, from the proof in \cite{BS} one can easily extract Theorem~\ref{t.bsdenovo}.

\begin{figure}[!h]
	\centering
	\includegraphics[width=200pt,height=200pt]{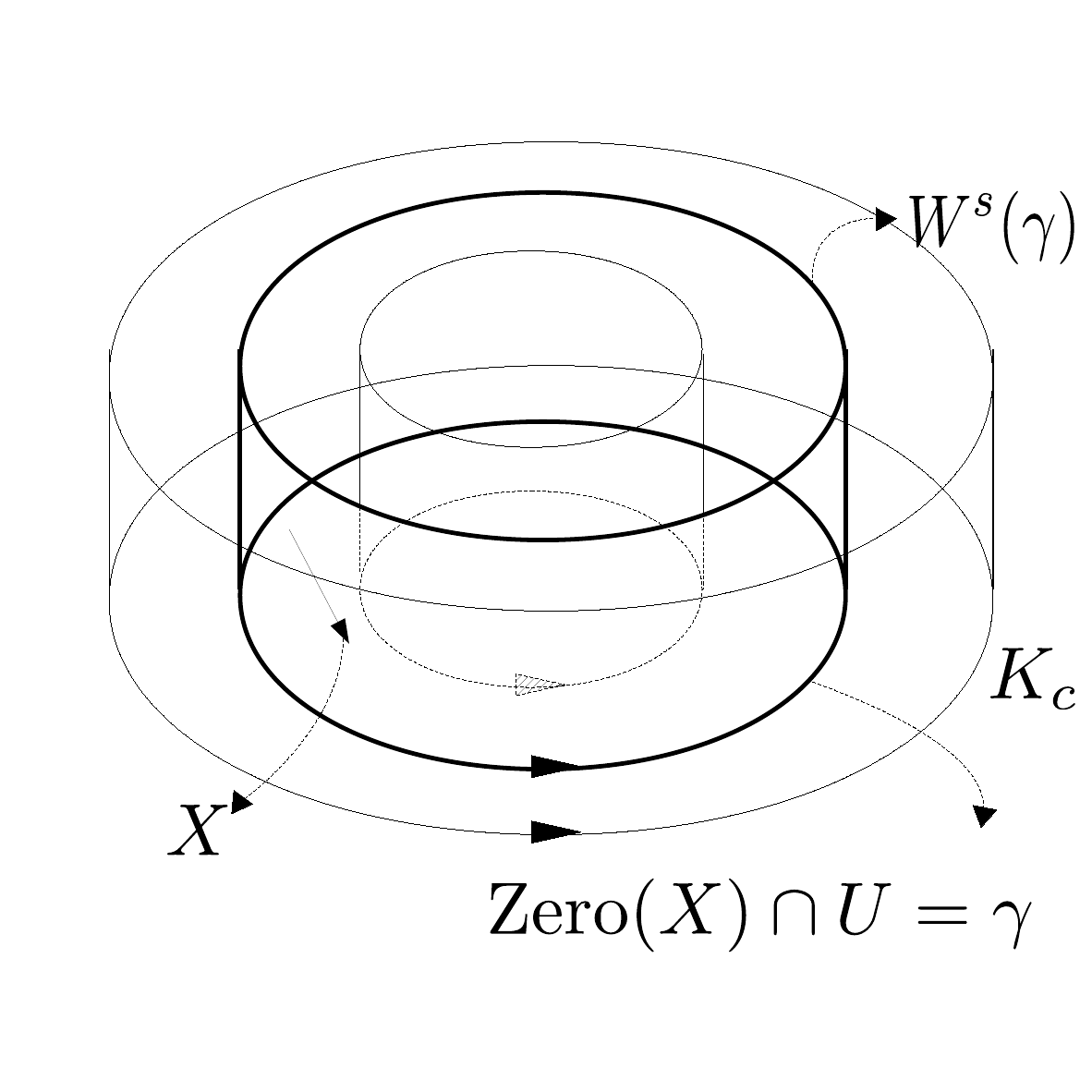}
	\caption{A particular case of Theorem~\ref{t.bsdenovo}: the collinearity locus is an annulus foliated by $Y$ periodic orbits, all of them which are \emph{partially hyperbolic}. The commutation obliges $X$ to be tangent to the foliation by stable manifolds, leaving no room for $X$ to turn, forcing the vanishing of the index. In Lemma~\ref{tangentialcase} we are in fact in this particular case.}  
	\label{f.anneau}
\end{figure}

\begin{proof}[Proof of Lemma \ref{tangentialcase}] Consider $\cV$, the tubular neighbourhood of $\gamma$ obtained in Lemma \ref{cletangent}. The surface $S$ obtained by pushing $W^c_{loc}(z)\cap\Sigma_0$ is an annulus trivially foliated by periodic orbits of $Y$. We can modify slightly $\cV$ so that $\cV\cap\Col_{U}(X,Y)=S$.
Lemma \ref{tangentialcase} now follows from Theorem~\ref{t.bsdenovo}.
\end{proof}
}
\subsubsection{The transverse case}

Recall that $\Ktr$ consists of periodic orbits of $Y$ which are not accumulated by non-periodic orbits included in $K$, and whose stable manifolds are transverse to the level set $\mu^{-1}(0)$. 

\paragraph{Nice decomposition of $\Ktr$ --}We want to prove that $\Ind(X,\Ktr)=0$. Here the difficulty is that there is no reason why this compact set should consist of finitely many periodic orbits of $Y$. To overcome this difficulty, we will need the next lemma, which gives a nice decomposition of $\Ktr.$

\begin{lem}
\label{structurethm}
There exists a finite collection of periodic orbits $\gamma_1,...,\gamma_k$ contained in $\Ktr$ and of open sets $\cU_1,...\cU_k$ satisfying the following properties.
\begin{enumerate}
\item For every $i=1,...k$, $\cU_i$ is a nice tubular neighbourhood of $\gamma_i$.
\item For $i\neq j$, $\cU_i\cap\cU_j\cap K=\vide$.
\item For every $i=1,...k$, $K\cap\partial\cU_i=\vide$.
\item $\Ktr\dans\bigcup_{i=1}^k\cU_i$.
\end{enumerate}
\end{lem}

\begin{proof}
We consider a periodic orbit $\gamma$ included in $\Ktr$ and a nice tubular neighbourhood of $\gamma$ as in Definition \ref{nicenbd}. Let $z\in\gamma$ and $\Sigma=\Sigma(z)$ be the fiber through $z$. As in the proof
of Lemma~\ref{l.lesorbitesliessonfini} we have that the only elements of $K$ which accumulate on $z$ belong to $\Ktr$. Thus we can choose $\cU$ small enough so that $\Sigma\cap K\dans\Ktr$. 

\begin{figure}[!h]
\centering
\includegraphics[scale=0.6]{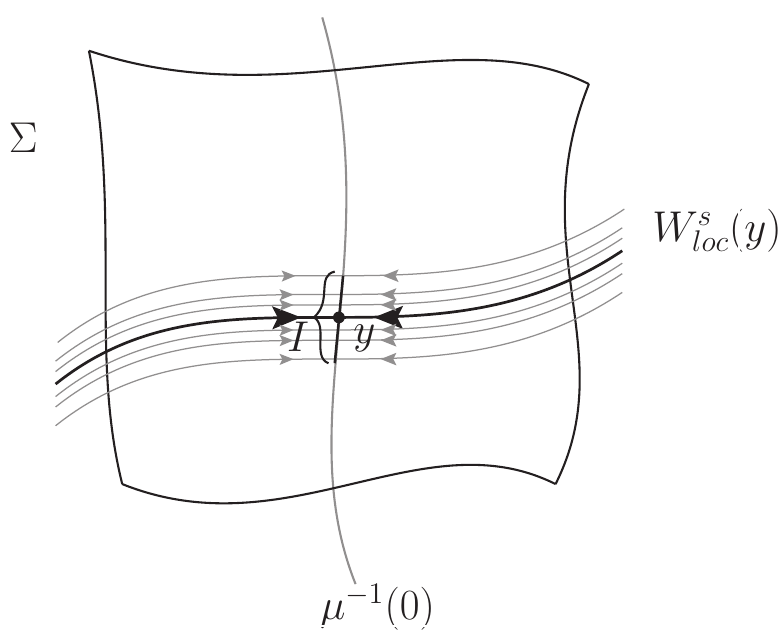}
\label{Empty_interior}
\caption{$I_0\cap K$ has empty interior}
\end{figure}

\emph{We claim that no point $y\in \Sigma\cap K=I_0\cap K$ is an interior point of $I_0\cap K$}. Indeed assume that there exists an interval $I\dans I_0$ containing $y$ which consists entirely of elements of $K$ (and so, all of them must be elements of $\Ktr$). Thus every point of $I$ has a local stable manifold, which is transverse to $I_0$. The union of these local stable manifolds is therefore a neighbourhood of $y$ in $\Sigma$, that we denote by $V$.

By the lower semi-continuity at $0$ of $Z:c\mapsto K_c\cap\Sigma$ there must exist $c\neq 0$ such that $K_c\cap V\neq\vide$. A point of this set must belong to the stable manifold of some point of $I\subset I_0\cap K$. This contradicts the fact that the orbit of any point $x\in K_c\cap V$ under the Poincar\'e map $P$ must stay inside the level set $I_c$. The claim is proven.

It follows that we can choose the nice tubular neighbourhood $\cU$ so that $K\cap\partial\cU=\vide$. Moreover, if two such open sets $\cU$ and $\cV$ intersect, then we can modify one of them such that $\cU\cap\cV\cap K=\vide$. 

Indeed, since $\cU_0$ is an annulus and $K$ is compact, we can modify $\cU$ so that $\partial\cU\cap\cU_0$ has
two connected components (say $S_{\cU}^+$ and $S_{\cU}^-$), each one of them being an embedded circle disjoint from $K$. Now, if some $\cV$ intersects $\cU$ we modify $\cV$ by requiring that for every $z\in\gamma$ and every $x\in\cV\cap I_0(z)$, if $a_z=S_{\cU}^+\cap I_0(z)$ and $b_z=S_{\cU}^-\cap I_0(z)$ then 
$x>a_z\:\:\textrm{or}\:\:b_z>x.$ This implies that $\cU\cap\cV\cap K=\vide$, as claimed.   

These open sets cover $\Ktr$ by definition. Since $\Ktr$ is compact it can be covered by finitely many such open sets. This achieves the proof of the lemma.
\end{proof}

\paragraph{End of the proof of Proposition \ref{Knouille}--} By Lemma \ref{structurethm} we must have
$$\Ind(X,\Ktr)=\sum_{i=1}^k\Ind(X,\cU_i).$$
Hence we are reduced to prove the following lemma.

\begin{lem}
\label{transversecase}
If $\cU$ is a sufficiently small nice tubular neighbourhood of a periodic orbit included in $\Ktr$, satisfying $K\cap\partial\cU=0$, then $\Ind(X,\cU)=0$.
\end{lem}

Here again the idea is to use the Center Manifold Theorem, in order to get down to a case already treated in \cite{BS}. The problem now is that the center manifold does not need to be everywhere transverse to the level sets of $\mu$. We will perform a small modification of $X$ in order to get down to a situation similar to that of the tangential case.

\begin{proof}[Proof of Lemma~\ref{transversecase}]
Let $z\in\gamma$, $\Sigma=\Sigma(z)$ and $W^c_{loc}(z)$ be a center manifold included in $\Sigma$. By reducing $\cU$ if necessary we can assume that every fixed point of the first return map $P$ to $\Sigma$ has a stable manifold, which is transverse to all the intervals $I_c$ it meets.

Look at the restriction of $\mu$ to $W^c_{loc}(z)$. We first note that it is not constant. This is because $W^c_{loc}(z)$ contain all fixed points of $P$ in a neighbourhood of $z$ and we can prove, as in the proof of Lemma \ref{cletangent}, that close to $z$ there are fixed points $z_c$ of $P$ contained inside $K_c$.

Now we observe that this is a function of class $C^1$ (in fact of class $C^3$) between $1$-dimensional manifolds. By Sard's theorem, there exists an interval $[\eps_1,\eps_2]$ of regular values of this function. Let $c\in[\eps_1,\eps_2]$ be a continuity point of the map $Z:c\mapsto K_c\cap\Sigma$.

\begin{figure}[!h]
\label{sardonique}
\centering
\includegraphics[scale=0.4]{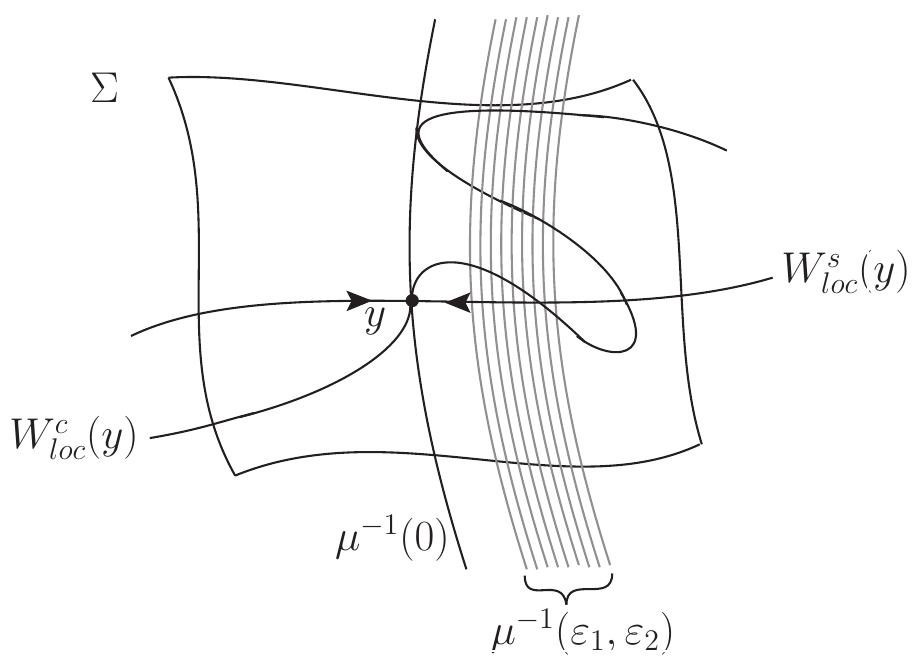}
\caption{Using Sard's theorem}
\end{figure}

The numbers $\eps_1$ and $\eps_2$ can be chosen arbitrarily close to $0$, so we can assume that $\Zero(X-cY)\cap\partial\cU=\vide$ and that
$$\Ind(X,\cU)=\Ind(X-cY,\cU).$$

The set $K_c\cap\Sigma$ must consist of fixed points of $P$. Indeed, since $c$ is a continuity point of $Z$ we can show as in the proof of Lemma \ref{cletangent} that the forward orbit of a point $x\in K_c$ is entirely included inside $I_c$. Therefore it must accumulate to some fixed point of $P$. By Proposition~\ref{confined} we have that the local stable manifold of this fixed point is tangent to $I_c$, which is a contradiction

Since $W^c_{loc}(z)$ contains the fixed points of $P$ close enough to $z$ we deduce that $K_c\cap\Sigma\dans I_c\cap W^c_{loc}(z)$. This set must be finite since it is the transverse intersection of two relatively compact embedded sub-manifolds of $\Sigma$. We deduce that $K_c\cap\cU$ is a finite union of periodic orbits of $Y$, denoted $\gamma_1,...,\gamma_m$. We can consider  $\cV_i\dans\cU$, a nice tubular neighbourhood of $\gamma_i$ such that $\Zero(X-cY)\cap\cV_i=\gamma_i$ and $\Zero(X-cY)\cap\partial\cV_i=\vide$. We then have
$$\Ind(X-cY,\cU)=\sum_{i=1}^m\Ind(X-cY,\cV_i).$$

We need to prove that for every $i=1,...,m$, $\Ind(X-cY,\cV_i)=0$. The situation is now the following. There is a embedded disc included in $\cV_i$ that contains all fixed points of $P$ (this is $W_{loc}^c(z)\cap\cV_i$), and there is no point of collinearity inside $\partial\cV_i$. Thus we are in condition to apply Lemma \ref{cletangent}. Once again Theorem~\ref{t.bsdenovo} allows us to conclude that $\Ind(X-cY,\cV_i)=0$. The proof of the lemma is now complete.
\end{proof}

\section{Conclusion}\label{s.conclusion}

As a conclusion, we propose to explain how the present paper inscribes inside a wider program to prove the $C^3$-case of the main Conjecture in its alternative form (see the introduction). We give below the main steps of a strategy aiming at achieving that. We use the notations and terminology introduced in the present work.

\paragraph{Warning --} We describe below the ideas of an ongoing work in order to illustrate that the present paper and \cite{BS} should be the two legs on which a proof of the Conjecture would stand. Even if we feel that our strategy is mature enough to be described below, we don't announce the resolution of the Conjecture. However we hope to answer the natural question: \emph{how far are we, with \cite{BS} and the present paper, from the resolution to the conjecture?}

\paragraph{Prepared counterexamples --} As in both this paper and \cite{BS} the strategy consists in showing that there is no prepared counterexample to the conjecture. The proof should then be a long argument by contradiction, supposing that there is one. The difference is that we don't make any hypothesis on the geometric configuration of the collinearity locus, nor on the dynamics close to periodic orbits included in that set. First, let us explain how to organize the collinearity locus.

\paragraph{Organization of the collinearity locus --} As explained in the present paper, for prepared counterexamples, the collinearity locus consists of periodic orbits of $Y$ and of connexions between them. Note that a periodic orbit of $Y$ inside $\Col_U(X,Y)$ is homotopically non-trivial in the corresponding level set of $\mu$ since otherwise it would enclose a zero of $Y$. It is possible to define an equivalence relation of such orbits saying that two of them are equivalent if they bound an annulus. Equivalence classes are included in essential annuli whose boundary components are periodic orbits. 

Since level sets are compact there are finitely many equivalence classes, which may be linked to others, in the sense that extremal periodic orbits are connected by non-periodic orbits of $Y$. This defines the \emph{combinatorics of equivalence classes} inside level sets of $\mu$.  As in the present paper we can assume that the intersection of $\Col_U(X,Y)$ with each level set is connected (by additivity of index). And using a (non-trivial) genericity argument, we can assume that the combinatorics of equivalence classes of periodic orbits is the same in all level sets. Compare with \S \ref{s.topo_col_locus} where the same situation was obtained by using tools from hyperbolic dynamics.

\paragraph{Using the index formula --} In order to prove that, in the context described above, $\Ind(X,U)=0$, our strategy is to use the index formula (see Theorem \ref{t:indexf}) and to prove that $l(\gamma)=0$ for every boundary component $\gamma$ of $\mu^{-1}(0)$. Using the same arguments as in this paper, we obtain some flexibility and we can ask that $\gamma\dans \mu^{-1}(c)$ for $|c|$ small enough (and the parameter $c$ might not be the same for all $\gamma$). There are two cases that will be treated by completely different methods and that correspond to the two ideal cases identified in the introduction.
\begin{enumerate}
\item $\gamma$ is isotopic to a periodic orbit of $Y$ included inside $\Col_U(X,Y)$;
\item $\gamma$ can be retracted to a cycle of periodic orbits and non-periodic orbits of $Y$ linking them contained in $\Col_U(X,Y)$.
\end{enumerate}

\emph{Case 1} can be treated with an adaptation of the arguments developped in \cite{BS}.

\emph{Case 2} is much more similar to the spirit of the present paper. We cannot hope to find a foliation by surfaces to which $X$ and $Y$ are tangent because we are not allowed to use tools from hyperbolic dynamics. However using a very careful dynamical study close to periodic orbits of the cycle, we can hope to build stable/unstable manifolds for periodic orbits of the cycle corresponding to the level set $\mu^{-1}(c)$ \emph{for a generic set of parameters $c$}. Obviously if, for example, a non-periodic orbit of $Y$ inside $K_c$ accumulates to a periodic orbit of the cycle in the future then it must belong to the stable manifold of this periodic orbit.

Using an adaptation of our glueing lemma we can build, for a generic set of parameters $c$, topological surfaces to which both $X$ and $Y$ are tangent, obtained by glueing stable and unstable manifolds of periodic orbits of the cycle (that is why we need to use the genericity of our set of parameters). These surfaces, denoted by $S_c$, satisfy the following properties for every parameter $c$ belonging to this generic set
\begin{itemize}
\item $S_c$ contain the cycle corresponding to $\mu^{-1}(c)$;
\item $X$, $Y$ and thus $N$ are tangent to $S_c$;
\item $S_c$ is homotopic to $\mu^{-1}(c)$ \emph{modulo the collinearity locus}.
\end{itemize}

\begin{figure}[!h]
\label{conclusion}
\centering
\includegraphics[scale=0.14]{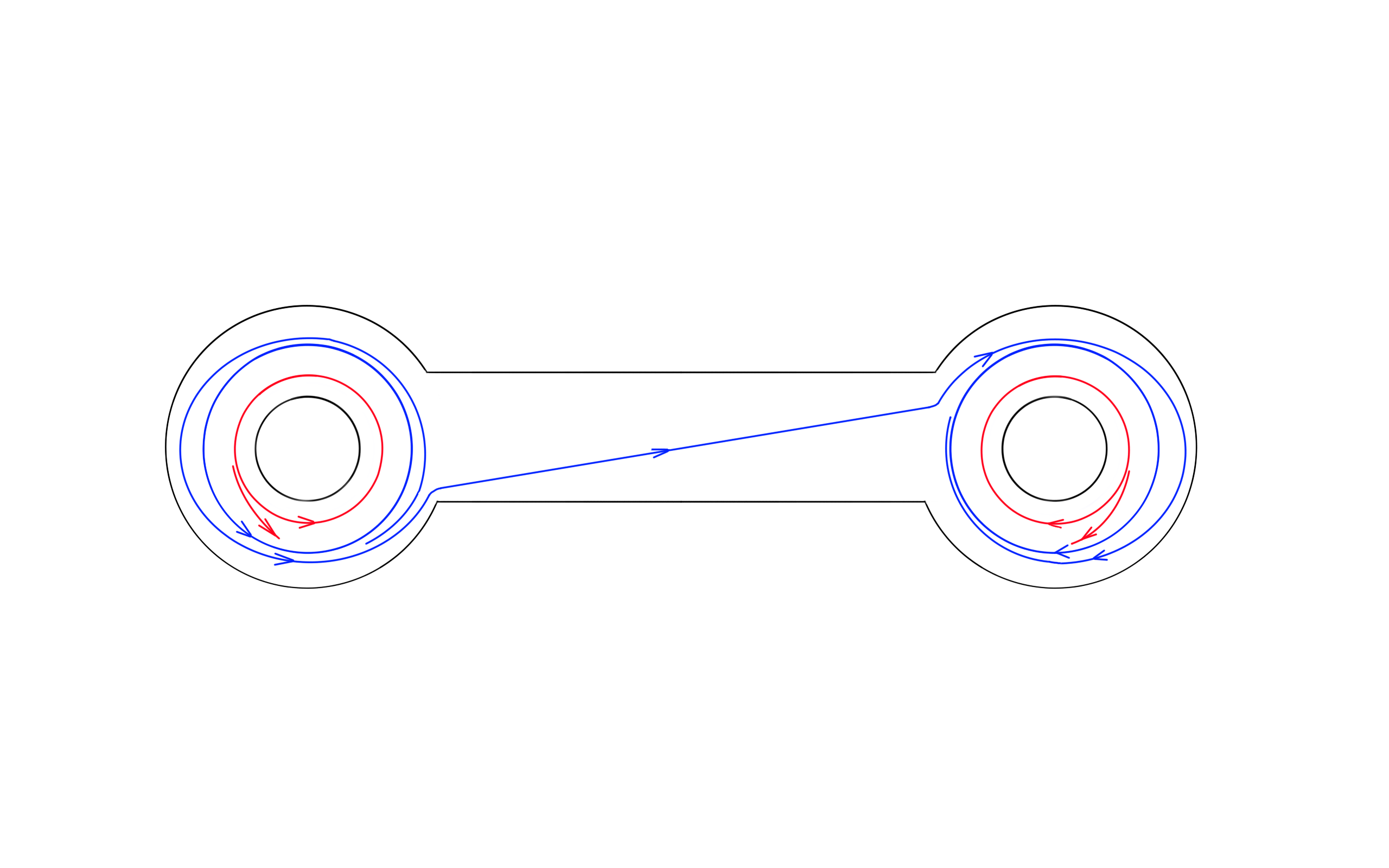}
\caption{In the example above, the level set is a pair of pants and the collinearity locus consists of two linked equivalence classes of periodic orbits. Two boundary components are isotopic to the two red periodic orbits: they are treated by Case 1. The last one can be retracted to the blue cycle and will be treated by Case 2. The blue periodic orbit on the left hand side possesses an unstable manifold for $Y$ and the blue periodic on the right hand side possesses a stable manifold. Both contain the connexion and will be glued just as in \S \ref{s.calc-index}.}
\end{figure}

Here again, the reader can compare with the situation described in \S\ref{s.calc-index} to which we came down using tools from hyperbolic dynamics. Proving this is the most difficult and technical step and the heart of our strategy. This is where we need new ideas. Let us however note that creating stable manifolds without using hyperbolic dynamics, but only the fact that $X$ and $Y$ commute, was also the technical heart of \cite{BS}.

The surfaces $S_c$ don't form a foliation as in \S\ref{s.calc-index}. But given the boundary component $\gamma$ of $\mu^{-1}(c)$, they will simplify the computation of $l(\gamma)$.  Indeed we can use a homotopy from $S_c$ to $\mu^{-1}(c)$, modulo collinearity, to lift $\gamma$ to a closed curve $\gamma'$ on $S_c$ satisfying $l(\gamma)=l(\gamma')$. Since the vector field $N$ along $\gamma'$ is tangent to $S_c$ it will be easy to show that $l(\gamma')=0$.

After these steps we can deduce that $\Ind(X,U)=0$. We hope to be able to carry on a proof based on this strategy in a forthcoming paper.

\paragraph{Acknowledgments} During the preparation of this work, S.A was supported by a post-doctoral grant financed by CAPES (Brazil) and by the Programa Contrataci\'on de Acad\'emicos Provenientes del Exterior of CSIC (Uruguay). He was also partially supported by the project \emph{Geometric theory of dynamical systems and France-Brazil cooperation in mathematics}, sponsored by Marcelo Viana's prix Louis D. He also acknowledges the support of  ANII via the project FCE-1-2017-1-135352.

C.B.was partially supported by IFUM and Ci\^encia sem Fronteiras CAPES.

B.S was supported by Marco Brunella's post-doctoral fellowship at Universit\'e de Bourgogne. Thanks to the generosity of Brunella's family. 

We acknowledge the kind hospitality of IMPA, PUC (Rio de Janeiro), UB (Dijon), UdelaR (Montevideo) and UFF (Niteroi), where this work was carried on.

{ Last but not least we wish to thank the referees for their many valuable comments that improved the presentation, and for their interest that led us to write a conclusion to put our work in the context of a global strategy.}

\bibliographystyle{plain}

\begin{thebibliography}{10}

\bibitem{Bonatti_analiticos}
C.~Bonatti.
\newblock Champs de vecteurs analytiques commutants, en dimension {$3$} ou
  {$4$}: existence de z\'eros communs.
\newblock {\em Bol. Soc. Brasil. Mat. (N.S.)}, 22(2):215--247, 1992.

\bibitem{bonatti_crovisier2015center}
C.~Bonatti and S.~Crovisier.
\newblock Center manifolds for partially hyperbolic sets without strong
  unstable connections.
\newblock {\em J. Inst. Math. Jussieu}, 15(4):785--828, 2016.

\bibitem{BS}
C.~Bonatti and B.~Santiago.
\newblock Existence of common zeros for commuting vector fields on
  $3$-manifolds.
\newblock {\em Ann. Inst. Fourier}, 67(04): 1741--1781, 2017. 


\bibitem{Denjoy}
A.~Denjoy.
\newblock Sur les courbes d\'efinies par les \'equations diff\'erentielles \`a
  la surface du tore.
\newblock {\em J. Math. Pures Appl. (9)}, 11:333--376, 1963.

\bibitem{Haef}
A.~Haefliger.
\newblock Vari\'et\'es feuillet\'ees.
\newblock {\em Ann. Scuola Norm. Sup. Pisa (3)}, 16:367--397, 1962.

\bibitem{HPS}
M.~W. Hirsch, C.~C. Pugh, and M.~Shub.
\newblock {\em Invariant manifolds}.
\newblock Lecture Notes in Mathematics, Vol. 583. Springer-Verlag, Berlin-New
  York, 1977.

\bibitem{HirschSmale}
M.~W. Hirsch and S.~Smale.
\newblock {\em Differential equations, dynamical systems, and linear algebra}.
\newblock Academic Press, New York-London, 1974.
\newblock Pure and Applied Mathematics, Vol. 60.

\bibitem{Kurouille}
K.~Kuratowski.
\newblock {\em Topology. {V}ol. {II}}.
\newblock Academic Press, New York-London; Pa\'nstwowe Wydawnictwo Naukowe
  Polish Scientific Publishers, Warsaw, 1968.

\bibitem{Lima_geral}
E.~L. Lima.
\newblock Common singularities of commuting vector fields on {$2$}-manifolds.
\newblock {\em Comment. Math. Helv.}, 39:97--110, 1964.

\bibitem{Lima_esfera}
E.~L. Lima.
\newblock Commuting vector fields on {$S^{2}$}.
\newblock {\em Proc. Amer. Math. Soc.}, 15:138--141, 1964.

\bibitem{Milnor}
J. Milnor.
\newblock {\em Topology from the differential viewpoint}.
\newblock The University Press of Virginia, Charlottesville, Va, 1965.

\bibitem{PdM}
J.~Palis and W.~de~Melo.
\newblock {\em Geometric theory of dynamical systems}.
\newblock Springer-Verlag, New York-Berlin, 1982.
\newblock An introduction, Translated from the Portuguese by A. K. Manning.

\bibitem{Santiague_these}
B.~Santiago.
\newblock Commuting vector fields and generic dynamics.
\newblock {\em PhD thesis, \emph{avalaible at
  \url{http://www.professores.uff.br/brunosantiago/wp-content/uploads/sites/17/2017/07/12.pdf}}}, 2012.

\bibitem{Santiague_note}
B.~Santiago.
\newblock The semicontinuity lemma.
\newblock {\em Unpublished note, \emph{avalaible at
  \url{http://www.professores.uff.br/brunosantiago/wp-content/uploads/sites/17/2017/07/01.pdf}}}, 2012.


\bibitem{Schwartz}
A.~Schwartz.
\newblock A generalization of a {P}oincar\'e-{B}endixson theorem to closed
  two-dimensional manifolds.
\newblock {\em Amer. J. Math.}, 85:453--458, 1963.

\bibitem{wiggins2013normally}
S.~Wiggins.
\newblock {\em Normally hyperbolic invariant manifolds in dynamical systems},
  volume 105 of {\em Applied Mathematical Sciences}.
\newblock Springer-Verlag, New York, 1994.
\newblock With the assistance of Gy\"orgy Haller and Igor Mezi\'c.


\end{thebibliography}

\begin{flushleft}
{\scshape S\'ebastien Alvarez}\\
CMAT, Facultad de Ciencias, Universidad de la Rep\'ublica\\
Igua 4225 esq. Mataojo. 11400 Montevideo, Uruguay.\\
email: \texttt{salvarez@cmat.edu.uy}

\smallskip
{\scshape Christian Bonatti}\\
CNRS, Institut de Math\'ematiques de Bourgogne (IMB, UMR 5584)\\
9 av.~Alain Savary, 21000 Dijon, France\\
email: \texttt{bonatti@u-bourgogne.fr}

\smallskip
{\scshape Bruno Santiago}\\
IME, Universidade Federal Fluminense\\
Rua Prof. Marcos Waldemar de Freitas Reis, S/N -- Bloco G, 4o Andar\\
Gabinete 72. Campus do Gragoatá, Niterói-RJ CEP 24210-201, Brasil\\
email: \texttt{brunosantiago@id.uff.br}
\end{flushleft}

\end{document}